\documentclass[12pt,reqno,oneside]{amsart}

\usepackage[colorlinks=true, linkcolor=blue, urlcolor=blue, citecolor=blue,%
pagebackref=true]{hyperref}

\usepackage{graphicx}

\usepackage{amsthm}
\usepackage{amsmath}
\usepackage{amssymb}
\usepackage{esint}

\renewcommand{\a}{\alpha}
\renewcommand{\b}{\beta}
\renewcommand{\d}{\delta}
\newcommand{\D}{\Delta}
\newcommand{\e}{\varepsilon}
\newcommand{\f}{\varphi}
\newcommand{\g}{\gamma}
\renewcommand{\l}{\lambda}
\newcommand{\s}{\sigma}
\renewcommand{\t}{\theta}
\newcommand{\R}{{\bf R}}
\newcommand{\Z}{{\bf Z}}
\newcommand{\C}{{\bf C}}
\newcommand{\E}{{\bf E}}
\newcommand{\wh}{\widehat}
\newcommand{\wt}{\widetilde}
\newcommand{\HH}{\mathcal{H}}
\newcommand{\BB}{\mathcal{B}}
\newcommand{\PP}{\mathcal{P}}

\newcommand{\sm}{\backslash}
\newcommand{\supp}{{\rm supp}\,}
\newcommand{\vspan}{{\rm span}\,}

\newcommand{\be}{\begin{equation}}
\newcommand{\ee}{\end{equation}}

\newcommand{\Fou}[2]{#1^{\wedge}(#2)}

\newtheorem{lem}{Lemma}
\newtheorem{thm}[lem]{Theorem}
\newtheorem{prp}[lem]{Proposition}
\newtheorem*{thma}{Theorem A}

\newcommand{\Isom}{{\rm Isom}}
\newcommand{\dist}{{\rm dist}}
\newcommand{\Lip}{{\rm Lip}}
\newcommand{\Res}{{\rm Res}}
\newcommand{\Vol}{{\rm Vol}}
\renewcommand{\O}{{\rm O}}
\newcommand{\SU}{{\rm SU}}
\newcommand{\SO}{{\rm SO}}


\title{Random walks in Euclidean space}
\author{P\'eter P\'al Varj\'u}
\thanks{
I acknowledge the support
of the European Research Council
(Advanced Research Grant 267259)
and the Simons Foundation}
\address{
Centre for Mathematical Sciences, Wilberforce Road, Cambridge CB3 0WA,
England
and
The Einstein Institute of Mathematics, Edmond J. Safra Campus,
Givat Ram, The Hebrew University of Jerusalem, Jerusalem, 91904, Israel}
\email{pv270@dpmms.cam.ac.uk}

\date{\today}

\begin{document}

\begin{abstract}
Fix a probability measure on the space of isometries of
Euclidean space $\R^d$.
Let $Y_0=0,Y_1,Y_2,\ldots\in\R^d$ be a sequence of random points
such that $Y_{l+1}$ is the image of $Y_l$ under a random isometry
of the previously fixed probability law,
which is independent of $Y_l$.
We prove a Local Limit Theorem for $Y_l$ under necessary non-degeneracy
conditions.
Moreover, under more restrictive but still general conditions we give a
quantitative estimate which describes the behavior of the law of $Y_l$
on scales $e^{-cl^{1/4}}<r<l^{1/2}$.
\end{abstract}

\maketitle

%%%%%%%%%%%%%%%%%%%%%%%%%%%%%%%%%%%%%%%%%%%%%%%%%%%%%%%%%%%%%%%%
\section{Introduction}
%%%%%%%%%%%%%%%%%%%%%%%%%%%%%%%%%%%%%%%%%%%%%%%%%%%%%%%%%%%%%%%%
\label{sc_intro}

Let $X_1,X_2,\ldots$ be independent identically distributed
random isometries of Euclidean space $\R^d$.
Let $x_0\in\R^d$ be any point and consider the sequence of points
\[
Y_0=x_0,\ldots, Y_l=X_l(X_{l-1}(\ldots(x_0))), \ldots .
\]
We call this sequence the {\em random walk} started from the point $x_0$,
and $Y_l$ is its $l$th step.

The purpose of this paper is to understand the distribution of
$Y_l$.

This problem can be traced back to Arnold and Krylov \cite{AK-uniform}
who studied the mixing of the random walk on the sphere
where the steps are rotations.
They asked if their results extend to isometries of Euclidean
or hyperbolic space.

Existing results in the literature can be divided into two classes.
Some papers describe the behavior of the measure on scale $O(1)$
others do it on scale $O(\sqrt{l})$.
We begin by discussing the first category.

Ka\v zdan \cite{Kaz-uniform} and Guivarc'h \cite{Gui-uniform}
proved a Ratio Limit Theorem for $d=2$.
This result describes the local behavior of the distribution of $Y_l$.
It states that the conditional distribution of $Y_l$ on a fixed
compact set is asymptotically uniform, i.e. Lebesgue.
More precisely, for any two smooth compactly supported
functions $f$ and $g$ we have
\[
\lim_{l\to\infty}\frac{\E[f(Y_l)]}{\E[g(Y_l)]}
=\frac{\int f(x)dx}{\int g(x)dx}
\]
provided that the denominator on the right do not vanish.
The law of $X_1$ of course need to satisfy some natural
non-degeneracy conditions for which we refer the reader to the
original articles.
The proofs rely on the fact that $\SO(2)$
is commutative. 

In the papers \cite{BBC-LLT}, \cite{Kho-LLT}, \cite{Max-LLT} the Local Limit Theorem
is generalized to higher dimension, however the arguments require some
restrictive assumption on the law of $X_1$, e.g. absolute continuity,
which implies that the group generated by the support of $X_1$ contains
translations.
In the absence of translations new ideas are
required to obtain
the Local Limit Theorem in full generality
which is the main goal of our paper.

Very recently, Conze and Guivarc'h \cite{CG-spectralgap} proved a Ratio Limit Theorem
under a certain assumption on the associated random walk on $\SO(d)$.
This assumption may hold in full generality but it has been
verified only under special circumstances yet.
(We elaborate on this assumption in Section \ref{sc_high} after Theorem A.)
Their approach also does not rely on translations,
but differ from the methods of this paper.

Tutubalin \cite{Tut-CLT} proved a Central Limit Theorem for dimension $d=2$
and $d=3$,
which was later generalized to higher dimension by Gorostiza \cite{Gor-CLT}
and Roynette \cite{Roy-CLT}.
The Central Limit Theorem describes the behavior of the distribution of
$Y_l$ on scale $O(\sqrt{l})$.
More precisely, it claims that $Y_l/\sqrt{l}$ converges weakly
to a Gaussian distribution if $Y_1$ has finite second moments.
The Central Limit Theorem was revisited by many authors,
see e.g. \cite{Gri-CLT}, \cite{RY-CLT}, \cite{Kho-CLT} and \cite{AMN-Euclidean}.
In these works the Central Limit Theorem was even generalized to cases
when the distribution  $Y_1$ has infinite second moment and
the limit distribution is not Gaussian.

To formulate our results we need to make some non-degeneracy
condition on the law of $X_i$.
We say that the law of $X_i$ is {\em degenerate} if there is a proper closed subset
$A$ of $\R^d$ and an isometry $\g\in\Isom(\R^d)$ such that $Y_l$ is almost
surely contained in $\g^l(A)$.
We say that the law of $X_i$ is {\em non-degenerate} if it is not degenerate.
Before we state the main result of the paper we state two simpler
results which can be deduced from our method.
The following version of the Central Limit Theorem follows from our work:

\begin{thm}[Central Limit Theorem]
\label{th_central}
Let $X_i, Y_i$ and $x_0$ be as above.
Suppose that $Y_1$ has finite second moments and the law of $X_i$
is non-degenerate.
Then there is a vector $v_0\in \R^d$ such that
the distribution of $(Y_l-lv_0)/\sqrt{l}$ weakly
converges to a Gaussian distribution.
\end{thm}

The limit distribution of course depends on the
distributions of $X_i$.
We do not describe this dependence explicitly, but
the mean and covariance matrix could be computed
from the proof.
Here we only mention, that the covariance matrix is invariant
under the rotation parts of the elements in the support of $X_i$.

This result is not new, it is covered by some of the references mentioned
earlier.
We revisit Tutubalin's argument in the greater generality considered in this paper.
(Tutubalin assumes that $\supp(X_i)$ generates a dense subgroup of
$\Isom(\R^d)$.
Moreover, he discusses the cases $d=2,3$ only, although
he does not seem to use this restriction in an essential way.)
Our main purpose is to obtain quantitative bounds, which
will be necessary for proving the error estimates in
Theorem \ref{th_main}
below.

\begin{thm}[Local Limit Theorem]
\label{th_local}
Let $X_i, Y_i$ and $x_0$ be as above.
Suppose that $Y_0$ has finite moments of order $d^2+3d+1$ and
$X_i$ are non-degenerate.
Let $f$ be any continuous and compactly supported function.
Then there is $v_0\in\R^d$ and $c>0$ depending only on the
distribution of $X_i$, such that
\[
\lim_{l\to \infty} l^{d/2}\E[f(Y_l-lv_0)]=c\int f(x)dx.
\]
\end{thm}

We remark that $v_0$ is the same as in the previous theorem
and $c$ can be computed from the covariance matrix of the
limit distribution.
Moreover, it turns out from the proof, that $v_0$ is almost surely
fixed by the rotation part of $X_1$, hence it is 0, if the rotation
part of the support of $X_1$ is sufficiently rich.

When $v_0=0$, the Local Limit Theorem can be interpreted as follows:
The probability that $Y_l$ belongs to a fixed compact set with smooth
boundary is asymptotic to $cl^{-d/2}$ times the Lebesgue measure
of the set.

In the Local Limit Theorem, we need the finiteness of high order
moments for technical reasons.
However, if the group generated by the support of $X_i$ is dense in
$\Isom(\R^d)$ then our arguments imply the Local Limit Theorem under
the assumption of finite second moment only.
In fact, this is true under much weaker assumptions on the group
generated by $\supp(X_i)$ (see Theorem \ref{th_main} below).

Now we formulate the main result of the paper which gives a
quantitative description of the distribution of $Y_l$ on multiple
scales.
However, we need a more restrictive assumption that we call $(SSR)$.
We postpone the definition to the next section, where we explain the
notation used through the paper.
For now, we only mention that $(SSR)$ holds for example if $\supp(X_i)$ generates
a dense subgroup of $\Isom(\R^d)$ and $d\ge 3$.
In addition, we can improve the error terms under stronger conditions,
i.e. symmetricity or $(E)$.
These will be defined in the next section, as well.

\begin{thm}
\label{th_main}
Let $X_i, Y_i$ and $x_0$ be as above.
Suppose that the law of $X_i$ is non-degenerate, satisfies $(SSR)$
and $Y_l$ has finite moments of order $\a$
for some $\a>2$.
Furthermore, let $f$ be a smooth
function of compact support.
Then there is a point $y_0\in \R^d$, a quadratic form $\D(x,x)$
and constants $C_\Delta$ and $c>0$ that depend only on the law of $X_i$,
such that
\begin{align}
\E[ f&(Y_l)]=C_\D l^{-d/2}\int f(x)e^{-\D(x-y_0,x-y_0)/l}dx\nonumber\\
&+O(l^{-\frac{d+\min\{1,\a-2\}}{2}}+|x_0|^2l^{-\frac{d+2}{2}})\|f\|_1
+O(e^{-cl^{1/4}})\|f\|_{W^{2,(d+1)/2}}.\label{eq_thmain}
\end{align}
In addition, if $\mu$ is symmetric or satisfies $(E)$, we have
\begin{align*}
\E[ f(Y_l)]
&=C_\D l^{-d/2}\int f(x)e^{-\D(x-y_0,x-y_0)/l}dx\\
&+O(l^{-\frac{d+\min\{2,\a-2\}}{2}}+|x_0|^2l^{-\frac{d+2}{2}})\|f\|_1
+O(e^{-cl^{1/4}})\|f\|_{W^{2,(d+1)/2}}.
\end{align*}
The implied constants depend only on the law of $X_i$.
\end{thm}

A few remarks are in order about the conclusion of this theorem.
The norm $\|\cdot\|_1$ is the $L^1$-norm and $\|\cdot\|_{W^{2,(d+1)/2}}$
is an $L^2$ Sobolev norm defined by
\[
\|f\|_{W^{2,(d+1)/2}}^2=\int |\wh f(\xi)|^2(1+|\xi|)^{d+1}d\xi.
\]
(The exponent $d+1$ could be replaced by any number greater than $d$
and the theorem would still hold.)

The first term on the right hand sides is the main term, it is the
integral of $f$ with respect to a Gaussian measure centered at $y_0$,
and with covariance matrix $\D/l$;
$C_\D$ is simply a normalizing factor.
It will follow from the proof that $\D$ is invariant under the rotation
parts of elements of the support of $\mu$.

The other two terms are error terms.
The first is responsible for the large scale, and the second is for the
small scale behavior of the random walk.
To illustrate this, fix a smooth compactly supported function $F$,
and consider the family $f_l(x)=r_l^{-d}F(x/r_l)$ associated to a
sequence of scales $r_l$.
(I.e. the diameter of the support of $f_l$ is proportional to $r_l$.)
It is easily seen that as long as $r_l<\sqrt{l}$,
the order of magnitude of the main term
is $l^{-d/2}$, the first error term is $l^{-\frac{d+\min\{1,\a-2\}}{2}}$
while the second
error term is $e^{-cl^{1/4}}r_l^{-d-1/2}$.
This shows that the theorem gives a good approximation in the
scale range $\sqrt{l}\ge r_l\ge e^{-c'l^{1/4}}$.

The factor $O(e^{-cl^{1/4}})$ is probably not optimal.
In fact, our proofs lead to better estimates in some cases.
This is discussed in detail in Section \ref{sc_high} after
Theorem A, after the necessary background is explained.

\noindent{\bf Acknowledgments.}
I am greatly indebted to Elon Lindenstrauss for telling me about
the problem and his continued interest in my project.
I am also grateful to Jean Bourgain and Emmanuel Breuillard
for helpful conversations about various aspects of the problem.
I thank Noam Berger, Alex Lubotzky and Nikolay Nikolov for suggesting
references \cite{Bur-martingale}, \cite{Bas-Jordan-Burnside} and \cite{Got-commutator} respectively.

I thank the referee for his or her very careful reading of the paper,
and for suggestions that greatly improved the presentation.

I am grateful for the hospitality of the Mathematical
Sciences Research Institute, Berkeley, USA.

%%%%%%%%%%%%%%%%%%%%%%%%%%%%%%%%%%%%%%%%%%%%%%%%%%%%%%%%%%%%%%%%
\section{Notation and outline}
%%%%%%%%%%%%%%%%%%%%%%%%%%%%%%%%%%%%%%%%%%%%%%%%%%%%%%%%%%%%%%%%
\label{sc_nota}

We identify the isometry group of the $d$-dimensional
Euclidean space with the
semidirect product $\Isom(\R^d)=\R^d\rtimes \O(d)$.
For $\g=(v,\t)\in\R^d\rtimes \O(d)$ and a point $x\in\R^d$
we write
\[
\g(x)=v+\t x,
\]
and we define the product of two isometries by
\[
(v_1,\t_1)(v_2,\t_2)=(v_1+\t_1 v_2,\t_1\t_2).
\]
If $\g$ is an isometry, we write $v(\g)$ for the translation
component and $\t(\g)$ for the rotation component of $\g$ in the
above semidirect decomposition.

Let $\mu$ be a probability measure on $\Isom(\R^d)$.
Define the {\em convolution} $\mu*\mu$ in the usual way by
\[
\int_{\Isom(\R^d)} f(\g)d\mu*\mu(\g)
=\int_{\Isom(\R^d)}\int_{\Isom(\R^d)} f(\g_1\g_2)d\mu(\g_1)d\mu(\g_2),
\]
for $f\in C(\Isom(\R^d))$ and
write
\[
\mu^{*(l)}=\underbrace{\mu*\cdots *\mu}_{l-{\rm fold}}
\]
for the $l$-fold convolution.
With this notation, $\mu^{*(l)}$ is the distribution of the
product of $l$ independent random element of $\Isom(\R^d)$
of law $\mu$.
We define the measure $\wt\mu$ by the formula
\[
\int_{\Isom(\R^d)}f(\g)d\wt\mu(\g)
=\int_{\Isom(\R^d)}f(\g^{-1})d\mu(\g),
\]
for $f\in C(\Isom(\R^d))$ and say that $\mu$ is {\em symmetric} if $\wt\mu=\mu$.
The measure $\mu$ also acts on measures on $\R^d$ in the following way:
If $\nu$ is a measure on $\R^d$, we can define another measure
$\mu.\nu$ on $\R^d$ by:
\[
\int_{\R^d} f(x)d\mu.\nu(x)
=\int_{\R^d}\int_{\Isom(\R^d)}f(\g(x))d\mu(\g)d\nu(x),
\]
for $f\in C(\R^d)$.

We write $\d_{x_0}$ for the Dirac delta measure concentrated
at the point $x_0$.
With this notation, the law of $Y_l$, the $l$th step of the random walk, is
$\mu^{*(l)}.\d_{x_0}$.

Write $\t(\mu)$ for the projection of $\mu$ on $\O(d)$, i.e.
for $f\in C(\O(d))$
\[
\int_{\O(d)}f(\s)d\t(\mu)(\s)=\int_{\Isom(\R^d)}f(\t(\g))d\mu(\g).
\]

Denote by $G\subset \Isom(\R^d)$ the closure of the group
generated by $\supp(\widetilde\mu*\mu)$.
Fix any element $\g_0\in\supp \mu$.
Then it is clear that $\supp\mu\subset\g_0 G$.

We can replace $\mu$ by $\mu'=\mu^{*(k)}$ for some fixed integer $k>1$
without loss of generality,
since $Y_{lk+j}$, the $lk+j$th step of the original random walk,
is the $l$th
step of the modified random walk started from the random point $Y_j$.
If we do so
then we replace $G$ by the group $G'$ as defined as the closure of the group generated by
$\supp(\widetilde{\mu^{*(k)}}*\mu^{*(k)})$.
It can be seen easily that $G'$ is the closure of the group
generated by
\[
G\cup\g^{-1}_0G\g_0\cup\ldots\cup\g^{-k+1}_0G\g_0^{k-1}.
\]
In Lemma \ref{lm_normalize}
we will see that if we choose $k$ sufficiently large than
$\t(G')$ is normalized by $\t(\g_0^k)$.
To keep this section compact, we postpone the statement and proof of
Lemma \ref{lm_normalize}, as well as Lemmata \ref{lm_O} and \ref{lm_almostnd}
that we will mention in the next pages.

Denote by $K\subset \O(d)$
the closure of the group generated $\supp \t(\widetilde\mu*\mu)$.
By the previous paragraph, we can (and will throughout the paper)
assume without loss of
generality that $K$ is normalized by $\t_0:=\t(\g_0)$.
Denote by $K^\circ$ the connected component of $K$.
Denote by $\mu_K$ the Haar measure on the group $K$.

Now we list
the various conditions that we will stipulate
on $\mu$ in various parts of the paper.
Some of these were already mentioned in Theorem \ref{th_main}.
\begin{itemize}
\item[$(C)$] (``Centered")
The barycenter of the image of the origin in $\R^d$
under $\mu$ is the origin, i.e.
\[\int\g(0)d\mu(\g)=0.\]
\item[$(E)$] (``Even")
The action of $K$ on $\R^d$ is ``even", i.e. for every
$v\in\R^d$, there is $\t_v\in K$ such that $\t_v v=-v$.
\item[$(SSR)$] (``Semi-simple rotations") $K^\circ$ is semi-simple,
and there is no non-zero point in $\R^d$ which is fixed by $K^\circ$.
\end{itemize}

We also recall the conditions we already defined for convenient reference.
We say that $\mu$ is {\em non-degenerate}, if there is no proper
closed subset $A\subset \R^d$ and an isometry $\g\in\Isom (\R^d)$
such that $\mu^{*(l)}.\d_{x_0}$ is almost surely contained in $\g^{l}(A)$.

It will be useful for us in many places in the paper to symmetrize
$\mu$ by replacing it with $\wt\mu*\mu$.
Unfortunately, the measure we obtain this way might be degenerate.
Consider the following example in $\R^2$:
Let $\g_1$ and $\g_2$ be two rotations  about two different centers
through the same angle, which is not a rational
multiple of $\pi$.
We leave it to the reader to verify that $\g_1$ and $\g_2$ generate a dense subgroup
in the orientation preserving isometries, hence the measure
$\mu=(\d_{\g_1}+\d_{\g_2})/2$ is non-degenerate.
However, for any $k$, $\wt{\mu^{*(k)}}*\mu^{*(k)}$ is supported on pure translations.
Moreover, we can choose $\g_1$ and $\g_2$ to have matrices with rational entries, and then
the translations in the support of $\wt{\mu^{*(k)}}*\mu^{*(k)}$ will all be rational.
Hence they preserve the lattice $(1/q)\Z^2$, where
$q$ is the common denominator.
This shows that $\wt{\mu^{*(k)}}*\mu^{*(k)}$ is degenerate.

For the above reason, we introduce a different notion which is easily
seen to descend to $\wt\mu*\mu$.
We say that $\mu$ is {\em almost non-degenerate} if for every point
$x\in\R^d$, the set $\{\g(x):\g\in\supp \mu\}$ does not lie in a proper
affine subspace.
As we will see in Lemma \ref{lm_almostnd},
if $\mu$ is non-degenerate then $\mu^{*(k)}$
is almost non-degenerate for some integer $k\ge1$, but it may happen that $\mu$
itself is {\em not} almost non-degenerate.
The implication in the other direction is often true, as well.
In particular, almost non-degeneracy is sufficient for most of the paper,
except for Section \ref{sc_smallSSA}.

We say that $\mu$ have
{\em finite moments of order} $\a>0$, if
\[\int |v(\g)|^\a d\mu(\g)<\infty.\]

A few remarks are in order regarding the role of these conditions.
{\em Non-degeneracy} is clearly necessary for the Local Limit Theorem.
However, we cannot impose it always for reasons discussed above.
On the other hand, almost non-degeneracy is required throughout the paper.
Condition $(SSR)$ is needed to control the behavior of $\mu^{*(l)}$
on very small scales (up to $e^{-cl^{1/4}}$).
Under this assumption we can utilize some powerful results about random
walks on semi-simple compact Lie groups.
We assume $(SSR)$ throughout Section \ref{sc_high} and some other
parts of the paper.
{\em Symmetry} or $(E)$ allows us to improve the error terms in 
Theorem \ref{th_main}.
They will be assumed in certain parts of Section \ref{sc_low}
to show that the
cubic terms in certain Taylor expansions cancel with each other.
We assume throughout the paper that $\mu$ has {\em finite moments of
order} 2.
In Section \ref{sc_low} we assume finite moments of order $\a$
for $2\le\a\le4$ and the quality of our error terms depend on $\a$.
To be able to conclude the Local Limit Theorem without using
$(SSR)$ we assume the finiteness of higher order moments in Section
\ref{sc_smallSSA}.
Finally, $(C)$ is an assumption which does not restrict generality
as we will see in Lemma \ref{lm_O}.
Therefore we assume it throughout the paper to simplify our arguments.

Now we introduce some further notation and indicate the general
strategy of the proof of Theorem \ref{th_main}.
Recall that the distribution of the random walk started at the point $x_0$
after $l$-steps is the measure
$\mu^{*(l)}.\d_{x_0}$.
As a consequence of the definitions, we see that
\[
\mu^{*(l+1)}.\d_{x_0}=\mu.(\mu^{*(l)}.\d_{x_0}).
\]
Hence our main goal is to understand the operation $\nu\mapsto\mu.\nu$.

This is achieved by studying the Fourier transform, which is given
by the formula
\[
\wh{\nu}(\xi)=\int e(\langle\xi, x\rangle)d\nu(x),
\]
where $e(x):=e^{-2\pi ix}$.
For the Fourier transform of $\mu.\nu$ we get
\begin{align}
\Fou{(\mu.\nu)}{\xi}
&=\int e(\langle\xi,\g(x)\rangle) d\mu(\g)d\nu(x)\nonumber\\
&=\int e(\langle\xi,v(\g)+\t(\g)(x)\rangle)
d\mu(\g)d\nu(x)\nonumber\\
\label{eq_Fourier}
&= \int e(\langle\xi,v(\g)\rangle)
\wh{\nu}(\t(\g)^{-1}\xi)d\mu(\g).
\end{align}

This formula shows that the action of $\mu$ on the Fourier transform
of $\nu$ can be disintegrated with respect to spheres centered at
the origin.
For every $r\ge0$,
we define a unitary representation of the group $\Isom(\R^d)$
on the space $L^2(S^{d-1})$.
Let
\be\label{eq_defrho}
\rho_r(\g)\f(\xi)=e(r\langle\xi,v(\g)\rangle)
\f(\t(\g)^{-1}\xi)
\ee
for $\g\in\Isom(\R^d)$, $\f\in L^2(S^{d-1})$ and $\xi\in S^{d-1}$.
We also define the operator
\be\label{eq_defSr}
S_r(\f)=\int\rho_r(\g)(\f)d\mu(\g).
\ee
For a function $\f\in C(\R^d)$ and $r\ge0$,
we denote by $\Res_r\f$ its restriction
to the sphere of radius $r$.
I.e. $\Res_r: C(\R^d)\to C(S^{d-1})$ is an operator
defined by
$[\Res_r \f](\xi)=\f(r\xi)$
for $|\xi|=1$.
With this notation, we can write (\ref{eq_Fourier}) as
\[
\Res_r(\wh{\mu.\nu})(\xi)=S_r(\Res_r \wh\nu)(\xi).
\]

Operators similar to $S_r$ were introduced by Ka\v zdan \cite{Kaz-uniform}
and Guivarc'h \cite{Gui-uniform}.
Guivarc'h proved in the $d=2$ case, when $K$ is Abelian,
that $\|S_r\|<1-cr^2$
for $r<1$ and $\|S_r\|<1-c_r$ for $r\ge1$, where $c>0$ is a constant
depending only on $\mu$, while $c_r$ also depend on $r$.
These estimates are sufficient for proving a Ratio Limit Theorem,
and, as Breuillard \cite{Bre-survey} pointed out, combined with the
Central Limit Theorem, it is sufficient even
for a Local Limit Theorem.
We are unable to prove such strong estimates, but we will prove
in Section \ref{sc_high} a weaker version: Proposition \ref{pr_high},
which is still sufficient for our application.
In brief, we prove the estimate with constants $c$ and $c_r$
which (mildly) depend on the oscillations of $\f$.
The proof is based on 
mixing properties of random walks on semi-simple compact Lie groups
(see Theorem A below).

Using the estimates given in Section \ref{sc_high}, we can
show that the Fourier transform of the random walk after
$l$ steps ``lives in" the ball of radius $l^{-1/2}\log l$.
These estimates alone are sufficient for the Ratio Limit Theorem
but not for the Central or Local Limit Theorems.
The frequency range $r>l^{-1/2}\log l$ is
responsible for the second error term in Theorem \ref{th_main}.

We need a more precise understanding of the Fourier transform of
$\mu^{*(l)}.\d_{x_0}$
in the range $r\le l^{-1/2}\log l$.
This frequency range contributes the main term and the first error term in
Theorem \ref{th_main}.
In section \ref{sc_low}, we give Tutubalin's \cite{Tut-CLT} argument
for the Central Limit Theorem in the more general setting that we consider
and obtain error estimates.
In brief, this argument is based on decomposing $L^2(S^{d-1})$
as the orthogonal sum of several subspaces and using the Taylor
expansion of the function $e(x)$ showing that these subspaces
are almost invariant for $S_r$.
We show that the contribution of only one of these subspaces is
significant and that on this subspace rotations act trivially.
Hence the problem is reduced to the easy case of sums of independent
random variables.

There is some interdependency between the arguments of Sections
\ref{sc_high} and \ref{sc_low}.
We explain this to demonstrate that our proof is not circular.
The arguments of Section \ref{sc_high} depend on the Central Limit Theorem,
which in turn depends on Section \ref{sc_low}.
However, Proposition \ref{pr_low} is sufficient for the Central Limit Theorem
and its refinement, Proposition \ref{pr_low2} is not needed.
Among the results of Section \ref{sc_high}, only Proposition \ref{pr_low2}
depends on the arguments of  Section \ref{sc_high}.

We will encounter $L^2$ spaces on various submanifolds of $\R^d$.
We always consider them with respect to the ``natural" measure, i.e.
which is invariant under isometries.
When the manifold is compact we normalize the measure to be probability.

Throughout the paper the letters $c,C$ and various subscripted versions
refer to constants and parameters.
The same symbol occurring in different places need not have the same
value unless the contrary is explicitly stated.
For convenience, we use lower case for constants which are best
thought of to
be small and upper case for those which are best
thought of to be large.
In addition, we occasionally use Landau's $O$ and $o$ notation.

The organization of the rest  of the paper that we have not explained
yet is as follows:
In Section \ref{sc_proof}, we combine the estimates of Section
\ref{sc_low} and \ref{sc_high} to conclude Theorem \ref{th_main}.
In Section \ref{sc_Pcentral} we derive Theorem \ref{th_central}
as a corollary of the results in Section
\ref{sc_low}.
Finally, in Section \ref{sc_Plocal} we prove Theorem \ref{th_local}.
When $K$ is Abelian and its action on $\R^d$ has a trivial component
some additional difficulties arise which prevents us from using the
method of Guivarc'h \cite{Gui-uniform}.
To address these issues,
we use Taylor expansions in Section \ref{sc_smallSSA}  motivated by
Tutubalin's paper.
For this argument we need to assume the finiteness of high order
moments.

%%%%%%%%%%%%%%%%%%%%%%%%%%%%%%%%%%%%%%%%%%%%%%%%%%%%%%%%%%%%%%%%%
\section{Justifying the simplifying assumptions}
\label{sc_O}
%%%%%%%%%%%%%%%%%%%%%%%%%%%%%%%%%%%%%%%%%%%%%%%%%%%%%%%%%%%%%%%%%

We prove three technical Lemmata in this section that we referred to in
the previous section.
Their common feature is that they allow us to make certain
simplifying assumptions on the law $\mu$ generating the random walk
without loss of generality.

First we prove that there is a suitable choice of origin for the coordinate system,
so that assumption $(C)$ is satisfied.
Then we prove that our assumption that $K$ is normalized by $\t_0$ is justified if
we replace $\mu$ by a convolution power of itself.
Finally we prove that the same replacement allows us to assume that $\mu$ is almost
non-degenerate.

\begin{lem}\label{lm_O}
Assume that there is no point in $\R^d$ except for the origin
which is fixed by all elements of $K$.
Then there is a unique point $x\in \R^d$ such that
\[
\int \g(x)d\mu(\g)=x.
\]
\end{lem}

The conclusion implies that if we change our coordinate system,
and set $x$ to be the origin, then $(C)$ is satisfied.

\begin{proof}
Consider the map $\R^d\to\R^d$:
\[
T(x)=\int \g(x)d\mu(\g)-\int \g(0)d\mu(\g)=\int \t(\g)x d\mu(\g).
\]
It is clear that $T$ is a linear transformation.

We show that $x-T(x)$ has trivial kernel.
Suppose that $x=T(x)$ for some $x\in\R^d$.
Since $|\t(\g)x|=|x|$ for all $\g$, and $T(x)$
is the average of these points, we must have
$\t(\g)x=x$, for $\mu$-almost all $\g$.
By our assumption, $x=0$, hence the kernel of $x-T(x)$
is indeed trivial.

Therefore, there is a unique point $x$ such that
$x-T(x)=\int\g(0)d\mu(\g)$, and this is exactly
what we wanted to prove.
\end{proof}

\begin{lem}\label{lm_normalize}
Let $K< \O(d)$ be a compact group, and $\t_0\in \O(d)$.
There is a positive integer $l$ such that $\t_0$ normalizes the
group generated by
\[
K\cup\t_0^{-1}K\t_0\cup\ldots\cup\t_0^{-(l-1)}K\t_0^{l-1}
\]
\end{lem}
\begin{proof}
It is a well-known fact that if $K$ is a compact Lie group,
then there is a
chain of normal subgroups $K^{\circ}\lhd H\lhd K$ such that
$H/K^{\circ}$ is commutative and $[K:H]<C_d$ for a constant $C_d$
depending on $d$.
For a proof in the context of algebraic groups which carries over
to compact groups without any changes see \cite[Theorem J]{Bas-Jordan-Burnside}.

Write $K_l$ for the closure of the group generated by
\[
K\cup\t_0^{-1}K\t_0\cup\ldots\cup\t_0^{-(l-1)}K\t_0^{l-1}
\]
and write $K_l^{\circ}\lhd K_l$ for its connected component.

Let $l_0\ge 1$ be an integer such that $K^{\circ}_l=K^{\circ}_{l_0}$
for $l\ge l_0$.
(The sequence $K_l^\circ$ stabilizes, since $\dim K^{\circ}_l$ may
grow at most finitely many times.)
Denote by $L$ the closure of the union of the groups $K_l$.
Then $K^\circ_{l_0}\lhd L$ since $K^\circ_{l_0}$ is normal in all
$K_l$ for $l\ge l_0$.

We show that $L^\circ/K^\circ_{l_0}$ is commutative.
For any $l\ge l_0$ and $g,h\in K_l$, we have
\[
[g^{C_d!},h^{C_d!}]\in K^{\circ}_{l_0}
\]
hence this property descends to $L$.
Since all elements in a connected compact Lie group are $C_d!$ powers,
we have $[L^\circ,L^\circ]<K^\circ_{l_0}$, thus
$L^\circ/K^\circ_{l_0}$ is indeed commutative.
Note that $L$ and $L^\circ$ are both normalized by $\t_0$ which is of
crucial importance for what follows.

Write $H_l=K_l\cap L^\circ$.
Then clearly $K^\circ_l\lhd H_l\lhd K_l$, $[K_l:H_l]\le[L:L^\circ]$
and $H_l/K_l^\circ$ is commutative for $l\ge l_0$.
Let $l_1\ge l_0$ be such that $[K_{l_1}:H_{l_1}]=[L:L^\circ]$
and let $g_1,\ldots,g_m$ be a system of representatives for $H_{l_1}$
cosets in $K_{l_1}$.

We show that $\exp(H_l/K^\circ_{l_0})$ is constant for $l\ge l_1+1$.
The {\em exponent} $\exp(G)$ of a group $G$ is the smallest integer $n$
such that $g^n=1$ for all $g\in G$.
Since the elements of $H_l$ approximate those of $L^\circ$, this would imply
that $\exp(L^\circ/K^\circ_{l_0})<\infty$.
Then $L^\circ=K^\circ_{l_0}$, as both of them are connected Lie groups.
Thus $H_l=K^\circ_{l_0}$ for all $l\ge l_0$, and $K_l=\{g_1,\ldots,g_m\}K^{\circ}_{l_0}$
for $l\ge l_1$.
That is, the sequence $K_l$ stabilizes, which was to be proved.

Let $l\ge l_1+1$.
Then all elements of $K_{l+1}$ are of the form
\[
g=\prod_\a \g_\a^{-1}(g_{i_\a}h_{i_\a})\g_\a,
\]
where $h_{i_\a}\in H_l$ and $\g_\a\in\{1,\t_0\}$.
For each $\a$, we can write
\[
\g_\a^{-1}(g_{i_\a}h_{i_\a})\g_\a=
g_{j_\a}h_{j_\a}\g_\a^{-1}h_{i_\a}\g_\a,
\]
where $g_{j_\a}$ is the appropriate coset representative and
\[
h_{j_\a}=g_{j_\a}^{-1}\g_\a^{-1}g_{i_\a}\g_\a\in H_{l_1+1}<H_l.
\]

We bring all $g_{j_\a}$ to the left hand side
of the product and get that each element of $H_{l+1}$ is of the form
\[
h=\prod_\b \g_\b^{-1}h_\b\g_\b,
\]
where $h_\b\in H_l$ and $\g_\b\in\{1,\t_0\}\cdot K_{l_1}$.
Thus all $\g_\b^{-1}h_\b\g_\b$ are in $L^\circ$,
in particular they commute modulo $K_{l_0}^{\circ}$.
In addition, the degree of each
$h_\b\cdot K_{l_0}^{\circ}\in H_l/K_{l_0}^{\circ}$ divides
$\exp(H_l/K_{l_0}^{\circ})$, hence
so is the degree of $h\cdot K_{l_0}^{\circ}\in H_{l+1}/K_{l_0}^{\circ}$.
This implies that $\exp(H_{l+1}/K_{l_0}^{\circ})=\exp(H_l/K_{l_0}^{\circ})$
which was to be proved.
\end{proof}

\begin{lem}\label{lm_almostnd}
Suppose that $\mu$ is non-degenerate.
Then there is a positive integer $l$ such that
$\mu^{*(l)}$ is almost non-degenerate.
\end{lem}
\begin{proof}
By the non-degeneracy assumption, it follows that for each
point $x\in\R^d$, there is $l(x)$ such that the set
$\{\g(x):\g\in\supp \mu^{*(l(x))}\}$ is not
contained in a proper affine subspace.
Indeed, assume to the contrary that this fails, and $l_0$ and $W$
are such that $\{\g(x):\g\in\supp \mu^{*(l_0)}\}$ spans $W$ and $W$
is of largest possible dimension.
Then $\g(x)$ is $d\mu^{*(l)}(\g)$-almost surely contained in
$\g_0^{l-l_0}(W)$, where $\g_0\in\supp \mu$ is arbitrary.
This contradicts to the non-degeneracy of $\mu$.

It is left to show that $l(x)$ is bounded on $\R^{d}$.
It is easy to see that $\{x:l(x)\le L\}$ is a Zariski open set
for every $L\in\Z$.
As $L\to \infty$ this is an ascending chain which eventually
covers $\R^d$.
Therefore the claim follows from the Noetherian property of Zariski
open sets.
\end{proof}

%%%%%%%%%%%%%%%%%%%%%%%%%%%%%%%%%%%%%%%%%%%%%%%%%%%%%%%%%%%%%%%%
\section{Estimates for high frequencies}
%%%%%%%%%%%%%%%%%%%%%%%%%%%%%%%%%%%%%%%%%%%%%%%%%%%%%%%%%%%%%%%%
\label{sc_high}

The goal of this section is to estimate the norm of the operator
$S_r$ defined in Section \ref{sc_nota}.
We are not able to show that $\|S_r\|<1$, but we can give an
estimate for $\|S_r\f\|_2$ in terms of the following Lipschitz
type norm of $\f$:
\[
\|\f\|_{\Lip(K)}:=\|\f\|_\infty+
\sup_{\xi\in S^{d-1},\t\in K\backslash\{1\}}\frac{|\f(\xi)-\f(\t(\xi))|}{\dist(1,\t)},
\]
where $\dist(\cdot,\cdot)$ is a distance function on $K$, which is induced by the
invariant Riemannian metric on $K$.
Note that there is a constant $C$ depending on the geometry of the embedding of $K$ inside
$\O(d)$ such that $|\xi-\t(\xi)|\le C\dist(1,\t)$ for every $\xi$ and $\t$.
Thus $\|\f\|_{\Lip(K)}\le C \|\f\|_{\Lip}$ for any function, where $\|\cdot\|_{\Lip}$
is the ordinary Lipschitz norm on the sphere.

\begin{prp}
\label{pr_high}
Suppose that $\mu$ is almost non-degenerate, has finite moments of order $2$,
and satisfies $(SSR)$.
Then there is a constant $c>0$ depending only on $\mu$ such that
the following hold.
Let $\f\in L^2(S^{d-1})$ with $\|\f\|_2=1$.
Then
\be\label{eq_prhigh}
\|S_r\f\|_2\le1-c\min\left\{r^2,\frac{1}{\log^3((r+1)\|\f\|_{\Lip(K)}+2)}\right\}.
\ee
\end{prp}

This estimate allows us to control the Fourier transform of the
random walk in the frequency range $e^{cl^{1/4}}>r>l^{-1/2}\log l$.

Mixing properties of random walks on semi-simple compact Lie groups
is a crucial ingredient of our proof.
We state the result that we use in the next theorem.
The proof is given in the paper \cite[Corollary 7]{Var-compact}.
A quantitatively weaker version, but
essentially sufficient for our purpose could be deduced
form the Solovay-Kitaev algorithm,
at least in the case $K=\SU(d)$.
The Solovay-Kitaev algorithm was first described in an e-mail discussion
list by Solovay in 1995.
Kitaev independently discovered it and published it in 1997 \cite{Kit-Solovay-Kitaev}.
For a recent exposition see \cite{DN-Solovay-Kitaev}.
See also the paper of Dolgopyat \cite[Theorems A.2 and A.3]{Dol-mixing},
which provides similar estimates.

\begin{thma}
Let $K$ be a compact Lie group with semi-simple connected
component.
Let $\mu$ be a symmetric probability measure on $K$ such that
$\supp\mu$ generates a dense subgroup in $K$.
Then there is a constant $c>0$ depending only on $\mu$ such that
the following hold.
Let $\f\in L^2(K)$ be a function such that $\|\f\|_2=1$
and $\int \f dm_K=0$.
Then
\be\label{eq_tha}
\left\|\int\f(\t^{-1}\s)d\mu(\t)\right\|_2<1-\frac{c}{\log^2(\|\f\|_\Lip+2)}.
\ee
\end{thma}

Recall that $m_K$ denotes the Haar measure on $K$.

As we mentioned in the introduction,
Conze and Guivarc'h \cite{CG-spectralgap} proved the Ratio Limit Theorem under a certain
assumption.
This assumption is that $K=\SO(d)$, and $\t(\mu)$ satisfies
(\ref{eq_tha}) with $1-c$ on the right independently of $\f$.
We add that Bourgain and Gamburd \cite{BG-SU2}, \cite{BG-SUd} proved
(in the $K=\SU(d)$ case)
that if $\mu$ satisfies some additional conditions (e.g. the support of
$\mu$ consists of matrices with algebraic entries)
then the stronger version of (\ref{eq_tha}) needed by Conze and Guivarc'h holds.

If one improves the estimate in Theorem A, then our argument presented
below provides better estimates in Proposition \ref{pr_high} and Theorem
\ref{th_main}.
In particular, if one can replace the right hand side of (\ref{eq_tha})
with $1-c\log^{-A}(\|\f\|_\Lip+2)$, then one can write
$1- c\min\{r^2,\log^{-A-1}((1+r)\|\f\|_{\Lip(K)}+2)\}$ on the right hand side
of (\ref{eq_prhigh}) and $O(e^{-cl^{1/(A+2)}})\|f\|_W$ instead of the
second error term in (\ref{eq_thmain}).
In fact, Theorem A is proved with better bounds for most Lie groups; except
for those which project onto $\SO(3)$.
For details, we refer to \cite{Var-compact}.
Moreover, for certain generators (e.g. when they are given with algebraic
entries), the estimates are available even with $A=0$,
as we discussed above.

The rest of the section is devoted to the proof of Proposition \ref{pr_high}.
A simple observation shows that it is enough to prove
it for symmetric measures.
Indeed, we have
\be\label{eq_symmetrize}
\|S_r\f\|_2^2=\langle S_r\f, S_r\f\rangle=
\langle \f, S_r^*S_r\f\rangle\le\|S_r^*S_r\f\|_2
\ee
and $S_r^*S_r$ is the operator analogous to $S_r$
corresponding to the symmetric measure $\wt\mu*\mu$.

We check that the assumptions of Proposition \ref{pr_high} hold for $\wt\mu*\mu$
if they hold for $\mu$.
Since $1\in\supp(\t(\wt\mu*\mu))$, $\supp(\t(\wt\mu*\mu))\subset\supp(\t((\wt\mu*\mu)^{*(2)}))$,
hence $\supp(\t((\wt\mu*\mu)^{*(2)}))$ generates a dense subgroup of $K$,
so $(SSR)$ holds for $\wt\mu*\mu$.
We have
\begin{align*}
\int |v(\g_1\cdot \g_2)|^2 d\wt\mu(\g_1)\mu(\g_2)
&\le\int |v(\g_1)+ v(\g_2)|^2 d\wt\mu(\g_1)\mu(\g_2)\\
&\le2\int |v(\g_1)|^2+ |v(\g_2)|^2 d\wt\mu(\g_1)\mu(\g_2)\le\infty,
\end{align*}
so $\wt\mu*\mu$ has finite second moments, too.
Let $\g_1\in\supp(\wt\mu)$ be arbitrary.
Then for any point $x\in\R^d$, the set
\[
\{\g(x):\g\in\supp(\wt\mu*\mu)\}\supset\{\g_1(\g(x)):\g\in\supp(\mu)\}
\]
cannot be contained in a proper affine subspace.
Thus $\wt\mu*\mu$ is also almost non-degenerate.
For the rest of the section, we write $\mu$ for $\wt\mu*\mu$ and $S_r$ for $S_r^*S_r$.
In addition, this argument shows that we can assume that $S_r$
is non-negative.

By Lemma \ref{lm_O}, we can change the origin in such a way that \label{pg_high}
$(C)$ holds for $\mu$.
Denote by $u$ the new origin in the old coordinate system.
Then the isometry $(v,\t)$ becomes $(v-u+\t u,\t)$ in the new coordinates.
Hence the operator $S_r$ will be replaced by the operator
\begin{align*}
S_r'\f(\xi)&=\int e (r\langle v(\g)-u+\t(\g)u,\xi\rangle)\f(\t(\g)^{-1}\xi) d\mu(\g)\\
&=e (r\langle-u,\xi\rangle)S_r(e (r\langle u,\xi\rangle)\f(\xi)).
\end{align*}
By setting $\f'(\xi)=e(r\langle u,\xi\rangle)\f(\xi)$,
we see that $\|S_r\f\|_2=\|S_r'\f'\|_2$.
Note that
\[
\|e(r\langle u,\xi\rangle)\f(\xi)\|_{\Lip(K)}\le C((r+1)\|\f\|_\infty+\|\f\|_{\Lip(K)})
\le C(r+1)\|\f\|_{\Lip(K)},
\]
where $C$ is a constant depending only on $u$.
This shows that if Proposition \ref{pr_high} holds for $S_r'$ and $\f'$,
then it also holds for $S_r$ and $\f$.

From now on, until the end of the section, we assume that $\mu$ is
symmetric, almost non-degenerate, has finite second moments
and satisfies $(C)$ and $(SSR)$.
Moreover, we assume that $S_r$ is selfadjoint and non-negative.
By the above discussion, these assumptions are justified.

Until the end of the section, we fix $r>0$ and a function
$\f\in\Lip(S^{d-1})$ and prove Proposition \ref{pr_high} for these.
The strategy of the proof is the following.
We fix two integers
\[
l_1=[C_1(r^{-2}+\log^3(\|\f\|_{\Lip(K)}+2))],\quad
l_2=[C_2(r^{-2}+\log^3(\|\f\|_{\Lip(K)}+2))],
\]
where
$C_1,C_2$ are suitably chosen large constants depending on $\mu$
but not on $\f$ or $r$.
We will show that the set of isometries that almost fix $\f$ in the $\rho_r$
representation is of $\mu^{*(l)}$ measure at most $9/10$ for $l=l_1$
or $l=l_2$.
This implies Proposition \ref{pr_high} by a standard argument.
More precisely, we prove the following lemma:

\begin{lem}\label{lm_high2}
Define the set
\[
B(\e):=\{\g\in\Isom(\R^d):\|\rho_r(\g)\f-\f\|_2<\e\}.
\]
If $\e$  is sufficiently small depending on $\mu$,
and $C_1$ is sufficiently large depending on $\e$ and $\mu$,
and $C_2$ is sufficiently large depending on $C_1$, $\e$ and $\mu$, then 
\be\label{eq_contra}
\mu^{*(l_i)}(B(\e))<9/10
\ee
holds for $i=1$ or $i=2$.
\end{lem}

We show how to deduce Proposition \ref{pr_high} from Lemma \ref{lm_high2}.
\begin{proof}[Proof of Proposition \ref{pr_high}]
Let $l_i$ be the one for which \eqref{eq_contra} holds and
assume that $l_i$ is even for simplicity.
Then we can write
\begin{align*}
\|&S_r^{l_i/2}(\f)\|_2^2=\langle S_r^{l_i/2}(\f),S_r^{l_i/2}(\f)\rangle
=\langle S_r^{l_i}(\f),\f\rangle\\
&\;=\int\langle\rho_r(\g)\f,\f\rangle d\mu^{*(l_i)}(\g)
\le\frac{1}{2}\int\langle\rho_r(\g)\f+\rho_r(\g^{-1})\f,\f\rangle
d\mu^{*(l_i)}(\g)\\
&\;\le\frac{1}{2}\int_{\Isom(\R^d)\sm B(\e)}
\langle\rho_r(\g)\f+\rho_r(\g^{-1})\f,\f\rangle
d\mu^{*(l_i)}(\g)+\mu^{*(l_i)}(B(\e))\\
&\;\le(1-\e^2/2)/10+9/10.
\end{align*}
To deduce the last inequality, we used the identity
\[
\langle\rho_r(\g)\f+\rho_r(\g^{-1})\f,\f\rangle=2-\|\rho_r(\g)\f-\f\|_2^2.
\]

We concluded that $\|S_r^{l_i/2}(\f)\|_2\le e^{-c}$
for some $c>0$ depending on $\mu$.
By selfadjointness of $S_r$ we can deduce that
$\|S_r(\f)\|_2\le e^{-2c/l_i}$.
We compare this with the definition of $l_i$,
which finishes the proof.
\end{proof}

The rest of the section is devoted to the proof of Lemma \ref{lm_high2}.
We begin with a simple lemma which shows that the length
of the translation part of $\g$ is proportional to $\sqrt{l}$ with $\mu^{*(l)}$ probability
at least $9/10$.

\begin{lem}
\label{lm_uplength}
\[
\int|v(\g)|^2d\mu^{*(l)}(\g)=l\cdot\int|v(\g)|^2d\mu(\g).
\]
\end{lem}
\begin{proof}
For $l=1$ the statement is obvious, for $l$ larger the
proof is by induction:
\begin{align*}
\int|&v(\g)|^2d\mu^{*(l+1)}(\g)=
\iint|v(\g_1\g_2)|^2d\mu(\g_2)d\mu^{*(l)}(\g_1)\\
&=\iint|v(\g_1)+\t(\g_1)v(\g_2)|^2d\mu(\g_2)d\mu^{*(l)}(\g_1)\\
&=\iint|\t(\g_1^{-1})v(\g_1)+v(\g_2)|^2d\mu(\g_2)d\mu^{*(l)}(\g_1)\\
&=\iint\left(|v(\g_1)|^2+|v(\g_2)|^2
+2\langle\t(\g_1^{-1})v(\g_1),v(\g_2)\rangle\right)d\mu(\g_2)
d\mu^{*(l)}(\g_1).
\end{align*}
Integrating out $\g_2$, the third term in the last line vanishes
by $(C)$.
This proves the lemma by induction.
\end{proof}

For the rest of the section, we assume to the contrary that \eqref{eq_contra}
fails, and we proceed by various Lemmata which show under the indirect hypothesis
that $B(\e')$ contains larger and larger families of isometries if $\e'$ becomes larger and larger
(but still small).
We will reach contradiction when we show that we can find
translations of length comparable to $r^{-1}$ in many directions.
The following simple property of the set $B(\e)$ will be used repeatedly:

\begin{lem}
For any numbers $\e_1,\e_2$ we have $B(\e_1)B(\e_2)\subset B(\e_1+\e_2)$.
\end{lem}
\begin{proof}
Let $\g_1\in B(\e_1)$ and $\g_2\in B(\e_2)$ be arbitrary.
Then by the triangle inequality, we have
\be\label{eq_triangle0}
\|\rho_r(\g_1\g_2)\f-\f\|_2\le\|\rho_r(\g_1\g_2)\f-\rho_r(\g_1)\f\|_2+\|\rho_r(\g_1)\f-\f\|_2
\ee
Since $\rho_r$ is a unitary representation, the first term on the right is equal to
$\|\rho_r(\g_2)\f-\f\|_2\le\e_2$.
Thus $\eqref{eq_triangle0}\le\e_1+\e_2$, which proves the lemma.
\end{proof}

In the first lemma, we conclude that $B(4\e)$ contains isometries
with an arbitrary prescribed rotation part and translation part
proportional to $\sqrt{l}$.

\begin{lem}\label{lm_rota}
Suppose that (\ref{eq_contra}) fails for $i=1$ and some $\e>0$.
Suppose further that $C_1$ is sufficiently large depending on $\mu$
and $\e$.
Then there exist a constant $C$ which depends only on $\mu$
such that the following holds:
There is a set $X\subset B(4\e)$ such that
\[
\t(X)= K \qquad{\rm and}\qquad
|v(\g)|<C\sqrt{l_1}\quad{\rm for}\quad \g\in X.
\]
\end{lem}
\begin{proof}
We deduce the lemma from Theorem A.
Let $\BB$ be the $\e\|\f\|_{\Lip(K)}^{-1}$ neighborhood of the
identity in $K^\circ$.
It follows from the definitions that $\|\rho_r(\t)\f-\f\|_\infty\le\e$,
for every $\t\in\BB$.
Thus we have $\BB\subset B(\e)$.

Take an approximate identity $\psi$ on $K$, which has the
following properties:
\[
\supp(\psi)\subset\BB,\quad\int\psi dm_K=1\quad{\rm and}\quad
\|\psi\|_{\Lip}\le C\|\f\|_{\Lip(K)}^{1+\dim K}.
\]
Note that these imply that
$\|\psi\|_2\le C\|\f\|_{\Lip(K)}^{(\dim K)/2}$.
These constants again depend only on $K$ and $\e$.
Now we apply Theorem A successively $l_1$ times starting
with the function $(1-\psi)/\|1-\psi\|_2$,
and get
\[
\left\|1-\int\psi(\t^{-1}\s)d\t(\mu)^{*(l_1)}(\t)\right\|_2\le\frac{1}{10}
\]
provided $C_1$ is sufficiently
large depending only on $\mu$ and $\e$.
Recall that $l_1>C_1\log^3(\|\f\|_{\Lip(K)} +2)$.
We note that taking the average of translates of $\psi$
may only decrease the Lipschitz norm.

Now let $Y\subset B(\e)$ be such that $\mu^{*(l_1)}(Y)>8/10$, and
\[
|v(\g)|<C\sqrt{l_1}/2 \quad{\rm for}\quad \g\in Y.
\]
For a sufficiently large $C$ depending on (the second moment of)
$\mu$, this is possible due to the assumption that (\ref{eq_contra}) fails and
Lemma \ref{lm_uplength}.
Denote by $\nu$ the measure we obtain from $\t(\mu^{*(l_1)})$ if we
restrict it to the set $Y$, and normalize it to get a probability
measure.
Then we have
\[
\left\|\int\psi(\t^{-1}\s)d\nu(\t)\right\|_2\le(1+\frac{1}{10})
\cdot\frac{10}{8}<\sqrt2.
\]
Thus
\[
m_K(\t(Y)\BB))\ge m_K\left(\supp\left(\int\psi(\t^{-1}\s)d\nu(\t)\right)\right)>\frac{1}{2},
\]
which proves the lemma with the choice $X=Y\BB Y\BB$.
\end{proof}

\begin{lem}
\label{lm_lowlength}
There are constants $c,C>0$ which depend only on $\mu$ such that,
we have
\[
\mu^{*(l)}(\g:|\langle v(\g),u_0\rangle|>c\sqrt{l}
\;{\rm and}\;|v(\g)|<C\sqrt{l})>1/2,
\]
for any sufficiently large (depending only on $\mu$)
integer $l$, and any $u_0\in S^{d-1}$.
\end{lem}

\begin{proof}
The lemma is an easy consequence of the Central Limit Theorem,
i.e. Theorem \ref{th_central}.
\end{proof}

We could, off course, replace 1/2 in the lemma with any number less than 1.
Now we can show that under our standing assumption that
(\ref{eq_contra}) fails, $B(5\e)$ contains a nontrivial translation.

\begin{lem}
\label{lm_translation}
Suppose that $(\ref{eq_contra})$ fails with some $1/2>\e>0$ for both
$i=1$ and $i=2$.
Suppose further that $C_1$ is sufficiently large so that Lemma
\ref{lm_rota} holds and $C_2$ is sufficiently large depending on $\mu$
and $C_1$.
Then there are constants $c,C>0$ depending only on $\mu$
such that for any $u_0\in S^{d-1}$, there is an element $\g_1\in B(5\e)$
with the following properties:
\[
|v(\g_1)|<C\sqrt{l_2},\quad \langle v(\g_1),u_0\rangle>c\sqrt{l_2} \quad
{\rm and}\quad
\t(\g_1)=1.
\]
\end{lem}
\begin{proof}
Denote by $c_0$ and $C_0$ the constants from Lemma \ref{lm_lowlength}.
Using that lemma with $l=l_2$ and the failure of (\ref{eq_contra}) for $i=2$,
we find an element $\g_2\in B(\e)$ such that
\[
|\langle v(\g_2),u_0\rangle|>c_0 \sqrt{l_2}\quad
{\rm and}\quad |v(\g_2)|<C_0\sqrt {l_2}.
\]

On the other hand, applying to Lemma \ref{lm_rota}, we can find
$\g_3\in B(4\e)$ that satisfies:
\[
\t(\g_3)=\t(\g_2)^{-1} \quad {\rm and}\quad |v(\g_3)|<C'_0\sqrt{l_1},
\]
where $C'_0$ is the constant $C$ from that lemma.

We demand that $l_2/l_1=C_2/C_1$ is so large that
\[
c_0 \sqrt{l_2}>2C'_0\sqrt{l_1}.
\]
Then it is an easy calculation to verify that $\g_1=\g_2\g_3$
has the claimed properties.

\end{proof}

Recall the definition of $l_2$, in particular that it implies
$\sqrt{l_2}>\sqrt{C_2}r^{-1}$.
The next Lemma shows that we can find a translation $\g_1'$
with properties similar to that of $\g_1$ in the previous
lemma, but which is
shorter.

\begin{lem}
\label{lm_shorten}
Under the same hypothesis as in Lemma \ref{lm_translation},
there is a constant $c>0$ which depend only on $\mu$,
and there is an element $\g'_1\in B(26\e)$ with the following properties:
\[
|v(\g_1')|<r^{-1}/2,\quad \langle v(\g_1'),u_0\rangle>cr^{-1} \quad
{\rm and}\quad \t(\g_1')=1.
\]
\end{lem}
\begin{proof}
Let $\g_1\in B(5\e)$ be an isometry with the properties stated in Lemma
\ref{lm_translation}, and write $v=v(\g_1)$.
For simplicity we assume that $\langle v,u_0\rangle >0$;
the other case is similar.
By the assumption $(SSR)$, we have
\[
\int \langle\t v,u_0\rangle dm_{K^\circ}(\t)=0.
\]
Thus, there is $\t_1\in K^\circ$, such that
$\langle\t_1 v,u_0\rangle\le0$.

There is a curve $\Theta:[0,1]\to K^{\circ}$ such that
$\Theta(0)=1$ and $\Theta(1)=\t_1$, and the length of the
curve $[0,1]\to\Theta(t)v$ is less than $C|v|$,
where $C$ depends only on the embedding of $K^\circ$ to
$\O(d)$, hence on $\mu$.
Then there is a sequence of rotations
\[\s_0=1,\s_1,\s_2,\ldots,\s_N=\t_1\in K^{\circ}\]
with
$N\le 2Cr|v|+1$
such that for any $1\le i\le N$
\[
|\s_i v-\s_{i-1}v|<r^{-1}/2.
\]

By the triangle inequality, there is an index $1\le i\le N$, such that
\[
\langle \s_{i-1} v-\s_{i}v, u_0\rangle\ge
\langle v, u_0\rangle/N\ge cr^{-1}
\]
with a suitably small constant $c>0$.
($c$ depends on $C$ and the constants appearing in the previous lemma.)

Now let $g_i\in  B(4\e)$ be such that $\t(g_i)=\s_i$; such elements
can be found by virtue of Lemma \ref{lm_rota}.
The proof is finished by an easy verification of the stated properties
for the element
\[
\g_1':=g_{i-1}\g_1g_{i-1}^{-1}
g_{i}\g_1^{-1}g_{i}^{-1}.
\]
\end{proof}

\begin{proof}[Proof of Lemma \ref{lm_high2}]
We assume to the contrary that (\ref{eq_contra})
fails for both $i=1$ and $i=2$.
Let $c>0$ be the constant from Lemma \ref{lm_shorten}.
Clearly, there is a point $u_0\in S^{d-1}$, such that
\be\label{eq_concentr}
\int_{|\xi-u_0|<c/2}|\f|^2d\xi>c',
\ee
with a constant $c'>0$ that depends only on $c$ and $d$. 

By Lemma \ref{lm_shorten}, there is an element $\g_1'\in  B(26\e)$
such that $\t(\g_1')=1$ and $v':=v(\g_1')$ satisfies
$|v'|<r^{-1}/2$, and $|\langle v',u_0\rangle|>cr^{-1}$.
This leads to the inequality
\[
\|\f-\rho_r(\g_1')\f\|_2=\int_{S^{d-1}}|(1-e(r\langle\xi,v'\rangle)\f(\xi)|^2d\xi
<(26\e)^2.
\]
If $|\xi-u_0|<c/2$, then $c/2<|r\langle\xi,v'\rangle|<1/2$.
This and (\ref{eq_concentr}) gives
\[
|1-e(c/2)|^2c'<(26\e)^2,
\]
which is a contradiction if we choose $\e$
to be sufficiently small.
Since $c$ and $c'$ depends only on $\mu$,
it follows that $\e$ depends only on $\mu$.
We chose $C_1$ depending on $\mu$ and $\e$ in Lemma \ref{lm_rota}
and $C_2$ depending on $C_1$ and $\mu$ in Lemma \ref{lm_translation}.
Thus all this parameters depend only on $\mu$.
\end{proof}

%%%%%%%%%%%%%%%%%%%%%%%%%%%%%%%%%%%%%%%%%%%%%%%%%%%%%%%%%%%%%%%%
\section{Estimates for low frequencies}
%%%%%%%%%%%%%%%%%%%%%%%%%%%%%%%%%%%%%%%%%%%%%%%%%%%%%%%%%%%%%%%%
\label{sc_low}

We recall some of our notation: $\mu$ is a fixed
probability measure on $\Isom(\R^d)$,
$K$ is the closure of the rotation group
generated by $\supp\t(\widetilde\mu*\mu)$, and
we assume that
$\supp\t(\mu)\subset \t_0 K$, where $\t_0\in \O(d)$
is a rotation which normalizes $K$. 

Fix a point $x_0\in\R^d$, the starting point of the random walk,
and fix a real number $r\ge 0$.
Define
\[
\psi_0(\xi)=e(r\langle x_0,\xi\rangle)=\Res_r(\wh{\d_{x_0}})(\xi)
\]
for $\xi\in S^{d-1}$, which
is the Fourier transform of the measure $\d_{x_0}$ restricted
to the sphere of radius $r$.
Our objective in this section is to estimate $S_r^l\psi_0=\Res_r(\wh\nu_l)$.
The estimate will be useful in the range $r<l^{-1/2}\log l$, i.e.
when the frequency is sufficiently small.

The next proposition shows that $S_r^l\psi_0$ is approximated
by the Fourier transform of a Gaussian distribution with covariance
matrix $\D$ which depends on $\mu$.

\begin{prp}\label{pr_low}
Assume that $\mu$ is almost non-degenerate,
has finite moments of order $\a$
for some $\a\ge 2$ and satisfies $(C)$.
There is a constant $C$ and a symmetric positive definite quadratic form
$\Delta(\xi,\xi)$ on $\R^d$ invariant under the action of $K$ and
$\t_0$, such that the following holds
\be\label{eq_plow1}
\|S_r^l\psi_0-e^{-r^2l\Delta}\|_2<C(r^{\min\{1,\a-2\}}+|x_0|^2r^2).
\ee
Moreover, if $\mu$ is symmetric or satisfies $(E)$,
then we have the better
bound:
\be\label{eq_plow2}
\|S_r^l\psi_0-e^{-r^2l\Delta}\|_2<C(r^{\min\{2,\a-2\}}+|x_0|^2r^2).
\ee
$C$ and $\Delta$  depend only on $\mu$.
When $\a=2$, we can replace the right hand sides of \eqref{eq_plow1} and \eqref{eq_plow2}
by $o(1)+C(|x_0|^2r^2)$ as $r\to 0$.
\end{prp}

The role of the last sentence is simply that we can conclude
the Central Limit Theorem even in the case $\a=2$.

The rest of this section is devoted to the proof of this proposition.
We will give a slight improvement  in Section \ref{sc_improv}
for the range $r>l^{-1/2}$.
However, this improvement requires the assumption $(SSR)$ and
it is based on the results of Section \ref{sc_high}.

Throughout this section
we make the following assumptions.
We assume that $r$ is small, i.e.
$r<c\min\{1,|x_0|^{-1}\}$, where $c$ is a suitable small constant.
For $r$ larger, the statement of the proposition is vacuous.
We assume that $\mu$ is almost non-degenerate, has finite moments of order
$\a\ge2$ and satisfies $(C)$.
In addition, at certain parts we assume
that $\mu$ is symmetric or satisfies $(E)$,
but we always mention these explicitly. 

The argument is based on Tutubalin's paper \cite{Tut-CLT}.
The most significant difference is
that we consider the following, more general,
decomposition of the space $\HH:=L^2(S^{d-1})$.
This is due to the fact that we do not assume $K=\SO(d)$.

Let $\HH_0$ be the subspace
of functions $\f\in\HH$, which are fixed by the action of $K$, i.e.
$\f(\t \xi)=\f(\xi)$ for every $\t\in K$.
For later reference we note that if $\f\in\HH$, then the orthogonal
projection of $\f$ to $\HH_0$ is obtained by the formula:
\be\label{eq_projinv}
\int \f(\t\xi)d m_K(\t).
\ee

Denote by $\PP_k\subset\HH$ the space of functions, which
are restrictions of degree $k$ polynomials
to $S^{d-1}$.
We define the spaces $\HH_k$, $k\ge1$ recursively.
Once $\HH_k$ is defined, let $\HH_{k+1}$ be the orthogonal
complement of $\HH_k$ in the space
\[
\vspan\{\psi\f:\psi\in \PP_{k+1}, \f\in\HH_0\},
\]
where $\vspan\{\cdot\}$ denotes the smallest closed subspace
that contains the functions inside the brackets.
Since $\PP_k\subset\HH_0\oplus\ldots\oplus\HH_k$, we have indeed
\[
\HH=\HH_0\oplus\HH_1\oplus\ldots.
\]
Denote by $\HH_\infty=\HH_4\oplus\HH_5\oplus\ldots$.
Finally, let $P_i:\HH\to\HH_i$ be the orthogonal projection
operator for each $i\in\{0,1,\ldots,\infty\}$.

In the special case $K=\SO(d)$, $\HH_k$ is the familiar space of spherical
harmonics of degree $k$, which was considered in Tutubalin's paper.

Taking $r=0$, it is easy to see that the above subspaces are invariant
for $S_0$.
Below, we will show that they are ``almost invariant" for small $r$;
more precisely, we will bound the norm of $P_i S_rP_j$ by a
polynomial of $r$, for $i\neq j$.
Additionally, we will see that the norm of
$P_iS_rP_i$ for all $1\le i<\infty$ is strictly less than 1.
However, the dependence on $i$ would require a more careful analysis.
Fortunately, we do not need to do this here, since as we will see, the
contribution of the spaces $\HH_i$, $i\ge 4$ is negligible
compared to other error terms, (this is why we introduced
the notation for the space $\HH_\infty$).
These estimates, which are simply based on Taylor expansion,
will be given in Section \ref{sc_Taylor}.

To simplify notation, we write $\psi_l=(P_0S_rP_0)^l\psi_0$ for $l\ge1$.
We will use the almost invariance of $S_r$ mentioned in the
previous paragraph to show that $\psi_l$ is a good approximation
to $S_r^l\psi_0$.
This is done in two steps.
We set $P=P_0+P_1+P_2+P_3$,
and consider another sequence, defined by
$\psi'_l=(PS_rP)^{l}\psi_{0}$ for $l\ge0$.
The next two Lemmata that will be proved in Section \ref{sc_psi}
claims that $\psi_l'$ approximates $S_r^l\psi_0$ and $\psi_l$
approximates $\psi_l'$.

\begin{lem}\label{lm_Papprox}
There is a constant $C>0$ such that the following holds:
\[
\|\psi'_l-S_r^l\psi_0\|_2\le C(r^{\min\{\a-2,2\}}+(|x_0|r)^4).
\]
When $\a=2$, we can replace the right hand side by
$o(1)+C(|x_0|r)^4$ as $r\to0$.
\end{lem}

\begin{lem}\label{lm_psiapprox}
There are constants $C,c>0$ depending only on $\mu$
such that the following holds for
$l\ge C\log(r^{-1}|x_0|+2)$:
\[
\|\psi_l-\psi'_l\|_2\le Ce^{-cr^2l}r.
\]
If $\mu$ is symmetric, then
\[
\|\psi_l-\psi'_l\|_2\le Ce^{-cr^2l}r^2.
\]
If $\mu$ satisfies $(E)$, (but not necessarily symmetric), then
\[
\|\psi_l-\psi'_l\|_2\le Ce^{-cr^2l}(r^{\min\{\a-1,2\}}+|x_0|r^2).
\]
\end{lem}

In light of these Lemmata, it remains to understand the
operator $P_0S_rP_0$.
This is essentially a multiplication operator as the next formula shows.
For $\f\in\HH_0$:
\begin{align}
P_0S_rP_0\f(\xi)&=\iint
e(r\langle\s\xi,v(\g)\rangle)\f(\t(\g)^{-1}\s\xi)
d\mu(\g)dm_K(\s)\nonumber\\
\label{eq_Sr0}
&=F(\xi)\f(\t_0^{-1}\xi),
\end{align}
where
\be\label{eq_defF}
F(\xi)=\iint
e(r\langle\s\xi,v(\g)\rangle)
d\mu(\g)dm_K(\s).
\ee
Recall that $\supp\t(\mu)\subset \t_0 K$, and $\t_0$ normalizes $K$.

Based on this formula and the Taylor expansion of the function $F$,
we will prove the following lemma in Section \ref{sc_lmpsi}.
\begin{lem}\label{lm_psi}
There are constants $C,c>0$ and a quadratic form $\Delta$ on $\R^d$
depending only on $\mu$
such that
\[
\|\psi_l-e^{-lr^2\Delta}\|_2<Ce^{-clr^2}(r^{\min\{1,\a-2\}}+|x_0|^2r^2).
\]
$\D$ is invariant under $K$ and $\t_0$.
When $\a=2$, we can replace the right hand side by $o(1)+C(|x_0|^2r^2)$
as $r\to0$.

If $\mu$ is symmetric or satisfies $(E)$,
then we have the better
estimate
\[
\|\psi_l-e^{-lr^2\Delta}\|_2<Ce^{-clr^2}(r^{\min\{2,\a-2\}}+|x_0|^2r^2).
\]
\end{lem}

Proposition \ref{pr_low} immediately follows from Lemmata
\ref{lm_Papprox}--\ref{lm_psi}.

%%%%%%%%%%%%%%%%%%%%%%%%%%%%%%%%%%%%%%%%%%%%%%%%%%%%%%%%%%%%%%%%%%%%
\subsection{Taylor expansion and approximate invariance}
\label{sc_Taylor}
%%%%%%%%%%%%%%%%%%%%%%%%%%%%%%%%%%%%%%%%%%%%%%%%%%%%%%%%%%%%%%%%%%%%
We give some estimates for the norm of the
operators $P_iS_r P_j$  in this section.
These will be deduced from the following lemma, which is
based on the Taylor series expansion of the function
$e(r\langle\xi,v(\g)\rangle)$ which is the multiplier in the 
representation $\rho_r$.

\begin{lem}\label{lm_Taylor}
There is an absolute constant $C>0$ such that for any $\f\in\HH$ with
$\|\f\|_2=1$ and $\g\in \Isom(\R^d)$ with $\t(\g)\in \t_0K$ we have
\begin{align}\label{eq_Taylor1}
\|P_i\rho_r(\g)P_i\f-\rho_0(\g)P_i\f\|_2&<Cr|v(\g)|,\\
\label{eq_Taylor2}
\|P_j\rho_r(\g)P_i\f\|_2&<\min\{1,C(r|v(\g)|)^{|i-j|}\}.
\end{align}
\end{lem}
\begin{proof}
For the proof, we can assume that $\f\in\HH_i$ that is $P_i\f=\f$.
By Taylor's theorem,
\begin{align}
\rho_r(\g)\f(\xi)&=e(r\langle\xi,v(\g)\rangle)\f(\t(\g)^{-1}\xi)\nonumber\\
&=\left[\sum_{m=0}^{M-1}C_mr^m\langle\xi,v(\g)\rangle^m+
O(r^{M}|v(\g)|^M)\right]\f(\t(\g)^{-1}\xi),\label{eq_Taylor3}
\end{align}
where $C_0=1$, and $C_m$ and the implied constant are absolute.

To deduce (\ref{eq_Taylor1}), take $M=1$ in \eqref{eq_Taylor3},
apply $P_i$ to both sides and subtract
\[
\rho_0(\g)\f(\xi)=\f(\t(\g)^{-1}\xi)=P_i(\f(\t(\g)^{-1}\xi)).
\]

To deduce (\ref{eq_Taylor2}) when $j>i$, take $M=j-i$.
Write
\[
q(\xi)=\sum_{m=0}^{j-i-1}C_mr^m\langle\xi,v(\g)\rangle^m\in \PP_{j-i-1}.
\]
Since $\f\in \HH_i$, we have
\[
\f=p_1\psi_1+\ldots+ p_k\psi_k
\]
with some $p_1,\ldots,p_k\in\PP_i$ and $\psi_1,\ldots,\psi_k\in\HH_0$.
Then
\begin{align*}
\sum_{m=0}^{j-i-1}C_m&r^m\langle\xi,v(\g)\rangle^m\f(\t(\g)^{-1}\xi)\\
&=q(\xi)[p_1(\t(\g)^{-1}\xi)\psi_1(\t_0^{-1}\xi)+\ldots
+p_k(\t(\g)^{-1}\xi)\psi_k(\t_0^{-1}\xi)],
\end{align*}
where 
$q(\xi)p_n(\t(\g)^{-1}\xi)\in\PP_{j-1}$
for any $1\le n\le k$.
Thus after applying $P_j$, these terms vanish in \eqref{eq_Taylor3}.
Then we get
\[
P_j\rho_r(\g)\f=O(r^{j-i}|v(\g)|^{j-i})\f(\t(\g)^{-1}\xi), 
\]
which proves (\ref{eq_Taylor2}) when $j>i$.

Let now $j<i$, and let $\psi\in\HH_j$ with $\|\psi\|_2=1$ be such that
\[
\|P_j\rho_r(\g)\f\|_2=\langle\rho_r(\g)\f,\psi\rangle=
\langle\f,\rho_r(\g^{-1})\psi\rangle\le\|P_i\rho_r(\g^{-1})\psi\|_2.
\]
Then the claim follows from (\ref{eq_Taylor2})
applied for $\psi$ and $\g^{-1}$ and the role of $i$ and $j$
reversed.
\end{proof}

\begin{lem}\label{lm_norm1}
There are constants $c<1$ and
$C$ depending only on $\mu$ such that the following hold
\begin{align*}
\|P_i S_rP_i\|&\le c \quad {\rm for}\quad r<c \quad {\rm and}
\quad 1\le i\le3,\\
\|P_i S_r P_j\|&\le C r^{\min\{|i-j|,\a\}}.
\end{align*}
When $|i-j|>\a$, we can
replace the second estimate by $o(r^{\min\{|i-j|,\a\}})$.
\end{lem}

\begin{proof}
To prove the first inequality,
we integrate (\ref{eq_Taylor1}) with respect to $d\mu(\g)$:
\begin{align}
\|P_iS_rP_i\f-S_0 P_i\f\|_2
&=\left\|\int P_i\rho_r(\g)P_i\f-\rho_0(\g)P_i\f d\mu(\g)\right\|_2\nonumber\\
&<Cr\int|v(\g)|d\mu(\g).\label{eq_norm11}
\end{align}

This inequality shows that it is enough to estimate the norm of $S_0$
on $\HH_1\oplus\HH_2\oplus\HH_3$.
Denote by $\PP'$ the orthogonal complement of $\HH_0\cap\PP_3$ in
$\PP_3$.
For each $\f\in\PP'$, $\|S_0^*S_0\f\|_2<\|\f\|_2$,
because otherwise $\f$ would be invariant under $K$, i.e.
we would have $\f\in\HH_0$.
Since $\PP'$ is
finite dimensional, there is a constant
$c<1$ such that $\|S_0^*S_0\f\|_2<c\|\f\|_2$ for $\f\in\PP'$.

Let $\f_1,\ldots,\f_k$ be an orthonormal basis of $\PP'$
consisting of eigenfunctions of $S_0^*S_0$.
Observe that the spaces $\f_i\cdot\HH_0$ are eigenspaces
of $S_0^*S_0$ with the same eigenvalues as $\f_i$.
Note that any $\f\in \HH_1\oplus\HH_2\oplus\HH_3$
is of the form $p_1\psi_1+\ldots+ p_k\psi_k$
with $p_i\in\PP'$ and $\psi_i\in\HH_0$.
Hence the eigenspaces  $\f_i\cdot\HH_0$  span $\HH_1\oplus\HH_2\oplus\HH_3$.
Then we are able to conclude that
\[
\|P_iS_0P_i\|^2\le\|P_iS_0^*S_0P_i\|<c
\]
for $i=1,2,3$.
This combined with \eqref{eq_norm11}
proves the first claim provided $r$ is sufficiently small.

For the second claim, we integrate (\ref{eq_Taylor2}):
\[
\|P_jS_rP_i\|<\int\min\{1, C(r|v(\g)|)^{|i-j|}\}d\mu(\g).
\]
If $|i-j|\le\a$, we can simply write
\[
\|P_jS_rP_i\|<Cr^{|i-j|}\int|v(\g)|^{|i-j|}d\mu(\g),
\]
and the claim follows from the moment condition on $\mu$.
If $|i-j|>\a$, we write
\[
\|P_jS_rP_i\|<Cr^{\a}\int|v(\g)|^{\a}
\min\{(r|v(\g)|)^{-\a},(r|v(\g)|)^{|i-j|-\a}\}d\mu(\g).
\]
Observe that
\[
\min\{(r|v(\g)|)^{-\a},(r|v(\g)|)^{|i-j|-\a}\}\le1
\]
and that it tends to 0 for all $\g$ as $r\to0$.
Now the claim follows by the dominated convergence theorem.
\end{proof}

The bound on $P_1S_rP_0$
in Lemma \ref{lm_norm1} is not optimal.
Indeed, it is easy to see that the linear terms
in the Taylor expansions cancel thanks to condition $(C)$.

\begin{lem}\label{lm_normlinear}
There is a constant $C$ such that
\[
\|P_1S_rP_0\|\le Cr^2.
\]
If $\mu$ also satisfies $(E)$, then we get the better bound
\[
\|P_1S_rP_0\|\le Cr^{\min\{3,\a\}}.
\]
\end{lem}
\begin{proof}
Take $\f\in\HH_0$, and as in the proof of Lemma \ref{lm_Taylor},
write the Taylor expansion:
\[
\rho_r(\g)\f(\xi)=\left[1+C_1r\langle\xi,v(\g)\rangle+
C_2r^2\langle\xi,v(\g)\rangle^2+
O(r^{3}|v(\g)|^3)\right]\f(\t_0^{-1}\xi).
\]

Similarly to the proof of Lemma \ref{lm_norm1}, we integrate
this inequality.
To get the first claim, we only need to note that
\[
\int\langle\xi,v(\g)\rangle d\mu(\g)=
\left\langle\xi,\int v(\g)d\mu(\g)\right\rangle=0
\]
because of the assumption $(C)$.

Assumption $(E)$ implies that $\HH_0$ consists of even
functions, and hence $\HH_1$ contains only odd ones.
Since
\[
\int\langle\xi,v(\g)\rangle^2 d\mu(\g)\cdot\f(\t_0^{-1}\xi)
\]
is an even function of $\xi$, it is in the kernel of $P_1$.
This establishes the second claim.
\end{proof}

We also need a norm estimate for $P_0S_rP_0$.
As we remarked in (\ref{eq_Sr0}), this operator is essentially
a multiplication
operator by the function $F$ defined in \eqref{eq_defF}.
Hence what we need to understand is the
behavior of $F$ near the origin.

\begin{lem}\label{lm_TaylorF}
There is a constant $C$ such that
\[
|F(\xi)-(1-r^2\Delta(\xi,\xi))|<Cr^{\min\{3,\a\}},
\]
where $\Delta(\xi,\xi)$ is a positive definite quadratic form
depending on $\mu$.
If $\a<3$, then the above bound can be improved to
$o(r^\a)$.
Furthermore, if $\mu$ is symmetric or satisfies $(E)$, then we have the
improved bound:
\[
|F(\xi)-(1-r^2\Delta(\xi,\xi))|<Cr^{\min\{4,\a\}}.
\]
\end{lem}

As an immediate corollary, we get that $\|P_0S_rP_0\|<1-cr^2$
for some constant $c>0$.

\begin{proof}
We expand \eqref{eq_defF},  the definition of $F$, in Taylor series
 the same way as we did in the previous lemmata:
\begin{align}
F(\xi)&=\iint 1+C_1r\langle\s\xi,v(\g)\rangle
-C_2r^2\langle\s\xi,v(\g)\rangle^2\nonumber\\
&\quad+C_3r^3\langle\s\xi,v(\g)\rangle^3dm_K(\s)d\mu(\g)
+O(r^{\min\{\a,4\}}),\label{eq_TaylorF1}
\end{align}
where $C_1,C_2, C_3$ are absolute constants, $C_2>0$, and the implied constant
is depends only on $\mu$.

First we note that as in the proof of Lemma \ref{lm_normlinear},
$(C)$ implies that
\[
\int\langle\s\xi,v(\g)\rangle d\mu(\g)=0
\]
for all $\s$ and $\xi$.
Hence the linear term vanishes in the Taylor expansion \eqref{eq_TaylorF1}.

Second, the quadratic term in \eqref{eq_TaylorF1}:
\[
\Delta(\xi,\xi):=\iint C_2r^2\langle\s\xi,v(\g)\rangle^2dm_K(\s)d\mu(\g)
\]
is clearly a $K$ invariant positive semi-definite quadratic form.
We only need to show that it is strictly positive definite.
Denote by $V$ the maximal subspace of $\R^d$ on which $\D(\xi,\xi)$
vanishes.
By the definition of $\Delta$, all $v(\g)$ is orthogonal to $V$,
which would contradict almost non-degeneracy if $V\neq\{0\}$.

If $\a\ge3$ and 
$(E)$ is satisfied, then for all
$\xi$, there is $\s_\xi\in K$ such that $\s_\xi\xi=-\xi$.
Then
\[
2\int \langle\s\xi,v(\g)\rangle^3dm_K(\s)=
\int\langle\s\xi,v(\g)\rangle^3+
\langle\s\s_\xi\xi,v(\g)\rangle^3dm_K(\s)=0,
\]
hence the cubic term in \eqref{eq_TaylorF1} vanishes.

Finally, if $\mu$ is symmetric and $\a\ge3$,
then we also have
\begin{align*}
2\iint \langle\s\xi,v&(\g)\rangle^3dm_K(\s)d\mu(\g)\\
&=\iint\langle\xi,\s v(\g)\rangle^3
+\langle\xi,\s \t(\g)v(\g^{-1})\rangle^3dm_K(\s)d\mu(\g)=0.
\end{align*}
The first equality follows since $\mu$ is symmetric and $m_K$ is
invariant under multiplication by $\t(\g)$ from the left.
The second one follows from $\t(\g)v(\g^{-1})=-v(\g)$.
The claim now follows from \eqref{eq_TaylorF1}.
\end{proof}

%%%%%%%%%%%%%%%%%%%%%%%%%%%%%%%%%%%%%%%%%%%%%%%%%%%%%%%%%%%%%%%%%%%%%%%%%%
\subsection{Approximating \texorpdfstring{$S_r$ by $P_0S_rP_0$}{Sr by P0SrP0}}
\label{sc_psi}
%%%%%%%%%%%%%%%%%%%%%%%%%%%%%%%%%%%%%%%%%%%%%%%%%%%%%%%%%%%%%%%%%%%%%%%%%%

The purpose of this section is to prove Lemmata \ref{lm_Papprox}
and \ref{lm_psiapprox}.
For both of them, we need the next lemma that
provides some estimates for the projections of $\psi'_l$
to the spaces $\HH_i$.

\begin{lem}\label{lm_P}
There are constants $c$ and $C$, such that the following hold:
\begin{align}
\|P_1\psi'_l\|_2&\le C(r^2+e^{-cl}|x_0|r)\quad{\rm and}
\label{eq_P1psi1}\\
\|P_i\psi'_l\|_2&\le C(r^{\min\{i,\a\}}+e^{-cl}|x_0|^{i}r^{i})
\label{eq_Pipsi1}
\end{align}
for $i\in\{2,3\}$ and $l\ge0$.

For $l\ge C\log(r^{-1}|x_0|+2)$, we have
\begin{align}
\|P_0\psi'_l\|_2&\le Ce^{-cr^2l}\label{eq_P0psi2}\\
\|P_1\psi'_l\|_2&\le Cr^2e^{-cr^2l}\quad{\rm and}\label{eq_P1psi2}\\
\|P_i\psi'_l\|_2&\le Cr^{\min\{i,\a\}}e^{-cr^2l}\label{eq_Pipsi2}
\end{align}
for $i\in\{2,3\}$.

Moreover, if $\mu$ satisfies $(E)$, then we can replace $r^2$
by $r^{\min\{3,\a\}}$ in $(\ref{eq_P1psi1})$ and $(\ref{eq_P1psi2})$.
\end{lem}

The following proof is
very technical, although the idea behind it is very simple.
The argument is based on induction on $l$, the norm estimates
of the previous section and triangle inequality.

We first give a brief sketch, which explains why the induction works.
For simplicity, take $x_0=0$ and suppose that the lemma holds for some
$\log(r^{-1}|x_0|+2)<l<r^{-2}$.
(For simplicity, we consider only this range in this informal discussion.)
We can write
\[
P_i\psi'_{l+1}=\sum_{j=0}^3P_iS_rP_j\psi'_l.
\]
We use the induction hypothesis and the lemmata of the previous
section to bound the terms.

What we need for the argument to work are essentially the inequalities
\be\label{eq_weneedthis}
(1-\|P_iS_rP_i\|)X_i
>\sum_{j\neq i}\|P_iS_rP_j\|\|P_j\psi'_l\|_2,
\ee
where $X_i$ is the bound for $\|P_i\psi'_l\|_2$ claimed in \eqref{eq_P0psi2}--\eqref{eq_Pipsi2}.
Notice that if we plug in our estimates that we obtained
in the previous section, then all terms on the right hand
side is of the same or smaller order of magnitude than the left hand side
for any $i$ and $j$.
For example take $i=2$:
Then $(1-\|P_2S_rP_i\|)X_i\ge C r^2$ by Lemma \ref{lm_norm1}.
For the terms on the right, we have
$\|P_2S_rP_j\|\|P_j\psi'_l\|_2\le Cr^{|2-j|}X_j$ by Lemma \ref{lm_norm1} again.
We see that $\|P_2S_rP_0\|\|P_0\psi'_l\|_2\le Cr^2$ and all the other terms are of lower order.

In the following diagram we draw a directed edge from vertex $j$ to
$i$, if the corresponding term on the right of \eqref{eq_weneedthis} is of the same order
as the left hand side.

\bigskip

\begin{center}
\includegraphics[scale=.5]{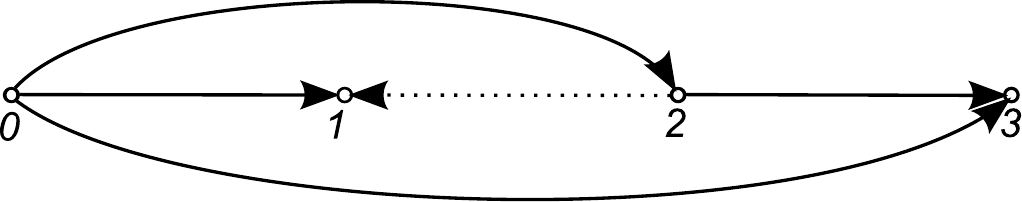}
\end{center}

\bigskip

\noindent (The dotted edge is present when (E) is assumed.)
Since it does not have a directed cycle, we can set the constants
in (\ref{eq_P0psi2})--(\ref{eq_Pipsi2})
following the edges of the diagram such that the induction will work.

\begin{proof}[Proof of Lemma \ref{lm_P}]
We do not give a separate proof for the last statement, but it will be
clear that using the improved estimates in Lemma \ref{lm_normlinear}
available under $(E)$, we get the better error terms.
Let $\b_0$ and $B_0$ be constants whose existence is guaranteed by
Lemmata \ref{lm_norm1}--\ref{lm_TaylorF} such that
\begin{align}
\|P_0S_rP_0\|&\le e^{-\b_0r^2},\label{eq_norm1}\\
\|P_iS_rP_i\|&\le e^{-\b_0}\quad {\rm for\;} i\in\{1,2,3\} \label{eq_norm2}\\
\|P_iS_rP_j\|&\le B_0r^{\min\{\a,|i-j|\}}\quad{\rm and}\label{eq_norm3}\\
\|P_1S_rP_0\|&\le B_0r^2.\label{eq_norm4}
\end{align}

The proof is by induction, and we begin with (\ref{eq_P1psi1}) and
(\ref{eq_Pipsi1}).
We suppose that $|x_0|\ge1$.
In the opposite case the argument is identical;
we only need to replace $|x_0|^2$ and $|x_0|^3$ everywhere by $|x_0|$.
We show that
\begin{align*}
\|P_1\psi'_l\|_2&\le C_1(r^2+e^{-\b_0l/2}|x_0|r)\quad{\rm and}\\
\|P_i\psi'_l\|_2&\le C_i(r^{\min\{i,\a\}}+e^{-\b_0l/2}|x_0|^{i}r^{i})
\end{align*}
for $i\in\{2,3\}$ and $l\ge0$, where $C_i>1$ are suitable constants to
be specified later.
For $l=0$, the claim is verified easily by the Taylor expansion
of $\psi_0$.

Suppose that the claim holds for $l$, and we prove it for $l+1$.
To estimate
\be\label{eq_triangle}
\|P_i\psi'_{l+1}\|_2\le\sum_{j=0}^3\|P_iS_rP_j\|\cdot\|P_j\psi'_l\|,
\ee
we use the induction hypothesis
for $\|P_j\psi'_{l}\|_2$ and the norm estimates
(\ref{eq_norm1})--(\ref{eq_norm4}).
We write for $i=1$:
\begin{align*}
\|P_1&\psi'_{l+1}\|_2\le
B_0r^2
+e^{-\b_0}C_1(r^2+e^{-\b_0l/2}|x_0|r)\\
&+B_0rC_2(r^2+e^{-\b_0l/2}|x_0|^2r^2)
+B_0r^2C_3(r^{\min\{3,\a\}}+e^{-\b_0l/2}|x_0|^3r^3)\\
&\le\left([B_0(1+rC_2+r^2C_3)+e^{-\b_0}C_1]e^{\b_0/2}\right)\cdot
(r^2+e^{-\b_0(l+1)/2}|x_0|r),
\end{align*}

To obtain the last line, we use inequalities of type
\[
(r^2+e^{-\b_0l/2}|x_0|^jr^j)\le
e^{\b_0/2}(r^2+e^{-\b_0(l+1)/2}|x_0|^jr^j)
\]
and also $r|x_0|<1$ that we can suppose without loss of generality
as we mentioned it at the beginning of Section \ref{sc_low}.
For $i=2$:

\begin{align*}
\|P_2&\psi'_{l+1}\|_2\le
B_0r^2
+B_0rC_1(r^2+e^{-\b_0l/2}|x_0|r)\\
&+e^{-\b_0}C_2(r^2+e^{-\b_0l/2}|x_0|^2r^2)
+B_0rC_3(r^{\min\{3,\a\}}+e^{-\b_0l/2}|x_0|^3r^3)\\
&\le\left([B_0(1+C_1+rC_3)+e^{-\b_0}C_2]e^{\b_0/2}\right)\cdot
(r^2+e^{-\b_0(l+1)/2}|x_0|^2r^2),
\end{align*}

We derive the last line the same way as before, but we also use
the inequality $|x_0|\ge 1$.
For $i=3$:

\begin{align*}
\|&P_3\psi'_{l+1}\|_2\le
B_0r^{\min\{3,\a\}}
+B_0r^2C_1(r^2+e^{-\b_0l/2}|x_0|r)\\
&+B_0rC_2(r^2+e^{-\b_0l/2}|x_0|^2r^2)
+e^{-\b_0/2}C_3(r^{\min\{3,\a\}}+e^{-\b_0l/2}|x_0|^3r^3)\\
&\le\left([B_0(1+C_1+C_2)+e^{-\b_0}C_3]e^{\b_0/2}\right)\cdot
(r^{\min\{3,\a\}}+e^{-\b_0(l+1)/2}|x_0|^3r^3).
\end{align*}

Now the claim is satisfied, if we take
\begin{align*}
C_1&=2e^{\b_0/2}B_0/(1-e^{-\b_0/2})\\
C_2&=2e^{\b_0/2}B_0(1+C_1)/(1-e^{-\b_0/2})\\
C_3&=e^{\b_0/2}B_0(1+C_1+C_2)/(1-e^{-\b_0/2})
\end{align*}
and $r$ is so small that $rC_2+r^2C_3<1$ and $rC_3<1$.

The proof of (\ref{eq_P0psi2})--(\ref{eq_Pipsi2}) is very similar.
We begin by choosing $l_0$ such that $e^{-\b_0l_0/2}|x_0|r<r^2$,
but $l_0<2\b_0^{-1}\log(|x_0|/r)+1$.
We show by induction that for $l\ge l_0$ the following holds:
\begin{align*}
\|P_0\psi'_l\|_2&\le C'_0e^{-\b_0r^2l/2},\\
\|P_1\psi'_l\|_2&\le C'_1r^2e^{-\b_0r^2l/2}\quad{\rm and}\\
\|P_i\psi'_l\|_2&\le C'_ir^{\min\{i,\a\}}e^{-\b_0r^2l/2}
\end{align*}
for $i\in\{2,3\}$ and $l\ge0$, where $C_0'=e^{\b_0r^2l_0/2}$
and for $i>0$, $C'_i\ge C_ie^{\b_0r^2l_0/2}$
are suitable constants to
be specified later.
We note that
\[
\b_0l_0/2\le\log(|x_0|/r)+1\le\log(r^{-2})+1<r^{-2}
\]
since $|x_0|r<1$.
Hence $e^{\b_0r^2l_0/2}<e$, so the above bounds on the constants $C_i'$
are independent of $x_0$ and $r$.

From the first part of the proof it follows that the claim holds
for $l=l_0$.
Now we suppose that it holds for a particular
$l\ge l_0$ and prove it for $l+1$.
As above, we use (\ref{eq_triangle}) along with
the induction hypothesis and
(\ref{eq_norm1})--(\ref{eq_norm4}).
We get:
\begin{align*}
&\|P_0\psi'_{l+1}\|_2\le e^{-\b_0r^2}C_0'e^{-\b_0r^2l/2}
+B_0rC'_1r^2e^{-\b_0r^2l/2}\\
&+B_0r^2C'_2r^2e^{-\b_0r^2l/2}
+B_0r^{\min\{3,\a\}}C'_3r^{\min\{3,\a\}}e^{-\b_0r^2l/2}\\
&\le\left([C_0'e^{-\b_0r^2}+B_0(r^3C'_1+r^4C'_2+r^{\min\{6,2\a\}}C'_3)]
e^{\b_0r^2/2}\right)\cdot e^{-\b_0r^2(l+1)/2},
\end{align*}
\begin{align*}
&\|P_1\psi'_{l+1}\|_2\le B_0r^2C_0'e^{-\b_0r^2l/2}
+e^{-\b_0}C'_1r^2e^{-\b_0r^2l/2}\\
&+B_0rC'_2r^2e^{-\b_0r^2l/2}
+B_0r^2C'_3r^{\min\{3,\a\}}e^{-\b_0r^2l/2}\\
&\le\left([e^{-\b_0}C'_1+B_0(C_0'+rC'_2+r^{\min\{3,\a\}}C'_3)]
e^{\b_0r^2/2}\right)\cdot r^2e^{-\b_0r^2(l+1)/2},
\end{align*}
\begin{align*}
&\|P_2\psi'_{l+1}\|_2\le B_0r^2C_0'e^{-\b_0r^2l/2}
+B_0rC'_1r^2e^{-\b_0r^2l/2}\\
&+e^{-\b_0}C'_2r^2e^{-\b_0r^2l/2}
+B_0rC'_3r^{\min\{3,\a\}}e^{-\b_0r^2l/2}\\
&\le\left([e^{-\b_0}C'_2+B_0(C_0'+rC'_1+r^{\min\{2,\a-1\}}C'_3)]
e^{\b_0r^2/2}\right)\cdot r^2e^{-\b_0r^2(l+1)/2},
\end{align*}
\begin{align*}
&\|P_3\psi'_{l+1}\|_2\le B_0r^{\min\{3,\a\}}C_0'e^{-\b_0r^2l/2}
+B_0r^2C'_1r^2e^{-\b_0r^2l/2}\\
&+B_0rC'_2r^2e^{-\b_0r^2l/2}
+e^{-\b_0}C'_3r^{\min\{3,\a\}}e^{-\b_0r^2l/2}\\
&\le\left([e^{-\b_0}C'_3+B_0(C_0'+rC'_1+C'_2)]
e^{\b_0r^2/2}\right)\cdot r^{\min\{3,\a\}}e^{-\b_0r^2(l+1)/2}.
\end{align*}
Now choose the constants in such a way that
\[
C'_1\ge \frac{2C_0'B_0e^{\b_0r^2/2}}{(1-e^{-\b_0+\b_0r^2/2})},
C'_2\ge \frac{2C_0'B_0e^{\b_0r^2/2}}{(1-e^{-\b_0+\b_0r^2/2})},
C'_3\ge \frac{2B_0(C_0'+C'_2)e^{\b_0r^2/2}}{(1-e^{-\b_0+\b_0r^2/2})}
\]
and observe that the claim holds for $l+1$ if $r$ is sufficiently small.
\end{proof}

\begin{proof}[Proof of Lemma \ref{lm_Papprox}]
By the triangle inequality, we have
\begin{align*}
\|\psi'_l-S_r^l\psi_0\|_2&\le
\sum_{k=0}^{l-1}\|S_r^{l-k-1}(S_r-PS_rP)(PS_rP)^k\psi'_0\|_2\\
&\le\sum_{k=0}^{l-1}\|(S_r-PS_rP)\psi'_k\|_2
\end{align*}

To estimate the terms, we write:
\begin{align}
\|(S_r-PS_rP)\psi'_k\|_2
&=\left\|S_rP_\infty\psi'_k+\sum_{j=0}^3(S_rP_j-PS_rP_j)P_j\psi'_k\right\|_2\nonumber\\
&\le\|P_\infty\psi'_k\|_2+\sum_{j=0}^3\|P_\infty S_rP_j\|
\cdot\|P_j\psi'_k\|_2.\label{eq_S-PSP1}
\end{align}
Recall that $P_\infty$ is the projection to the complement
of $\HH_0\oplus\ldots\oplus\HH_3$.
Note that $P_\infty\psi'_k=0$ for $k\ge1$, since
$\psi_k'\in \HH_0\oplus\ldots\oplus\HH_3$.

We use Lemmata \ref{lm_P} and \ref{lm_norm1}.
For $1\le k\le C\log(r^{-1}|x_0|+2)$, we have
\be\label{eq_S-PSP2}
\|(S_r-PS_rP)\psi'_k\|_2<
C(r^{\min\{4,\a\}}+|x_0|^3r^4e^{-ck}),
\ee
while for $k=0$, we have to add $\|P_\infty\psi'_k\|_2\le C|x_0|^4r^4$ to the above estimate.
For $k\ge C\log(r^{-1}|x_0|+2)$, we have
\be\label{eq_S-PSP3}
\|(S_r-PS_rP)\psi'_k\|_2<
Cr^{\min\{4,\a\}}e^{-cr^2k}.
\ee
Summing for $k$, we get the statement of the lemma.

When $\a=2$, the constants in Lemma \ref{lm_norm1} are arbitrarily small
as $r\to0$.
If we plug these in \eqref{eq_S-PSP1} we see that the constants in
 \eqref{eq_S-PSP2} and  \eqref{eq_S-PSP3} are also arbitrarily small.
\end{proof}

\begin{proof}[Proof of Lemma \ref{lm_psiapprox}]
By the triangle inequality:
\begin{align*}
\|\psi_l-\psi_l'\|_2&=\|(P_0S_rP_0)^l\psi_0-(PS_rP)^l\psi_0\|_2\\
&\le\sum_{k=0}^{l-1}
\|(P_0S_rP_0)^{l-k-1}(P_0S_rP_0-PS_rP)(PS_rP)^k\psi_0\|_2\\
&\le\sum_{k=0}^{l-2}\|P_0S_rP_0\|^{l-k-1}\cdot
\|(P_0S_rP_0-P_0S_rP)\psi'_k\|_2\\
&\quad+\|(P_0S_rP_0-PS_rP)\psi'_{l-1}\|_2.
\end{align*}

As in the previous proof, we write
\[
\|(P_0S_rP_0-P_0S_rP)\psi'_k\|_2\le\sum_{j=1}^3\|P_0S_rP_j\|\cdot
\|P_j\psi'_k\|_2.
\]
Again, we use Lemmata \ref{lm_P}, \ref{lm_norm1}, \ref{lm_normlinear}
and the estimate $\|P_0S_rP_0\|\le1-cr^2$, which follows from
Lemma \ref{lm_TaylorF}.
For $k\le\log(r^{-1}|x_0|+2)$ we can write
\[
\|P_0S_rP_0\|^{l-k-1}\cdot
\|(P_0S_rP_0-P_0S_rP)\psi'_k\|_2\le Ce^{-cr^2l}(|x_0|r^2e^{-ck}+r^3)
\]
While for $k\ge\log(r^{-1}|x_0|+2)$ we get
\[
\|P_0S_rP_0\|^{l-k-1}\cdot
\|(P_0S_rP_0-P_0S_rP)\psi'_k\|_2\le Cr^3e^{-cr^2(l-1)}
\]
and
\begin{align*}
\|(P_0&S_rP_0-PS_rP)\psi'_{l-1}\|_2\le
\|(P_0S_rP_0-P_0S_rP)\psi'_{l-1}\|_2\\
&+\|(P_0S_rP-PS_rP)\psi'_{l-1}\|_2
\le Cr^3e^{-cr^2(l-1)}+Cr^2e^{-cr^2(l-1)}.
\end{align*}
Summing up for $k$, we get the first statement of the lemma.

If $\mu$ is symmetric then $S_r$ is selfadjoint, hence
\[
\|P_0S_rP_1\|=\|P_1S_rP_0\|\le Cr^2.
\]
Using this instead of Lemma \ref{lm_norm1}, we get
\[
\|P_0S_rP_0\|^{l-k-1}\cdot
\|(P_0S_rP_0-P_0S_rP)\psi'_k\|_2\le Cr^4e^{-cr^2(l-1)},
\]
and the better estimate follows after summation.

If $\mu$ satisfies $(E)$ instead, then the better estimate in
Lemma \ref{lm_P} gives
\[
\|P_0S_rP_0\|^{l-k-1}\cdot
\|(P_0S_rP_0-P_0S_rP)\psi'_k\|_2\le Cr^{\min\{4,\a+1\}}e^{-cr^2(l-1)}.
\]
\end{proof}

%%%%%%%%%%%%%%%%%%%%%%%%%%%%%%%%%%%%%%%%%%%%%%%%%%%%%%%%%%%%%%%%%%%%%%%
\subsection{Proof of Lemma \ref{lm_psi}}\label{sc_lmpsi}
%%%%%%%%%%%%%%%%%%%%%%%%%%%%%%%%%%%%%%%%%%%%%%%%%%%%%%%%%%%%%%%%%%%%%%%

Again, we only show the first inequality in the lemma, the second
follows along the same lines by applying the improved estimate in
Lemma \ref{lm_TaylorF}.

Using the Taylor series expansion of $\psi_0$ together with $(C)$
and finite second moments, we see that
\be\label{eq_proofF1}
\psi_0(\xi)=1+O(|x_0|^2r^2),
\ee
where the implied constant only depends on $\mu$.
Furthermore, we have by (\ref{eq_Sr0}) that
\be\label{eq_proofF2}
\psi_l(\xi)=\left[\prod_{j=0}^{l-1} F(\t_0^{-j}\xi)\right]\psi_0(\t_0^{-l}\xi),
\ee
where $F$ is given by (\ref{eq_defF}).

Let $\Delta$ be the quadratic form appearing in Lemma \ref{lm_TaylorF}.
Define $\Delta_0$ to be its symmetrization by the group generated
by $\t_0$, i.e.
\[
\Delta_0(\xi,\xi):=\lim_{n\to\infty}\frac{1}{n}\sum_{i=0}^{n-1}
\Delta(\t_0^{-i}\xi,\t_0^{-i}\xi).
\]

In light of \eqref{eq_proofF1} and \eqref{eq_proofF2} it is enough to show that
\be\label{eq_lmFgoal}
\left|\prod_{j=0}^{l-1} F(\t_0^{-j}\xi)-e^{-lr^2\Delta_0(\xi,\xi)}\right|
<Ce^{-clr^2}r^{\min\{1,\a-2\}}
\ee
for all $\xi$.

The rest of the proof is devoted to this inequality.
By Lemma \ref{lm_TaylorF}, we have
\be\label{eq_proofF3}
\sum_{j=0}^{l-1}\log F(\t_0^{-j}\xi)=-r^2\sum_{j=0}^{l-1}
\Delta(\t_0^{-j}\xi,\t_0^{-j}\xi)+O(lr^{\min\{3, \a\}}).
\ee
Denote by $W$ the space of quadratic forms on $\R^d$, and denote by
$\Theta_0\in {\rm End}(W)$ the linear transformation induced by
$\theta_0$.
It is easily seen that $\Theta_0$ is diagonalizable and all its
eigenvalues are on the unit circle of $\C$.
Denote by $W_0$ the 1-eigenspace of $\Theta_0$.
Hence $\Delta_0$ is the projection of $\D$ to $W_0$.
Then on the orthogonal complement of $W_0$, we have
\[
\sum_{j=0}^{l-1}\Theta_0^{-j}=\frac{\Theta_0^{-l}-1}{\Theta_0^{-1}-1}.
\]

Thus it follows that
\[
\sum_{j=0}^{l-1}
\Delta(\t_0^{-j}\xi,\t_0^{-j}\xi)
=\sum_{j=0}^{l-1}\Theta_0^{-j}\Delta(\xi,\xi)
=l\Delta_0(\xi,\xi)+O(1).
\]
The implied constant depends on the distance of the non-trivial
eigenvalues of $\Theta_0$ to 1.

If we combine our inequalities, we get
\[
\sum_{j=0}^{l-1}\log F(\t_0^{-j}\xi)=-lr^2\Delta_0(\xi,\xi)
+O(r^2+lr^{\min\{3, \a\}}).
\]
This immediately implies that there is a constant $c>0$ such that
\[
\prod_{j=0}^{l-1} F(\t_0^{-j}\xi)<e^{-clr^2}.
\]
If we use the inequality $|e^A-e^B|<(A-B)\max\{e^A,e^{B}\}$, then
we get
\be\label{eq_proofF4}
\left|\prod_{j=0}^{l-1} F(\t_0^{-j}\xi)-e^{-lr^2\Delta_0(\xi,\xi)}\right|
<Ce^{-clr^2}(r^2+lr^{\min\{3, \a\}}).
\ee
To obtain (\ref{eq_lmFgoal}) and hence the lemma, we only need to
note that
\[
e^{-clr^2}l\le \frac{2}{c}\cdot e^{-clr^2/2}r^{-2}.
\]

If $\a=2$, the term $O(lr^{\min\{3,\a\}})$ in \eqref{eq_proofF3} can be improved to
$o(lr^2)$ by Lemma \ref{lm_TaylorF}.
Hence the right hand side of \eqref{eq_proofF4} can be improved to $o(1)$ as claimed.

%%%%%%%%%%%%%%%%%%%%%%%%%%%%%%%%%%%%%%%%%%%%%%%%%%%%%%%%%%%%%%%%%%%%%%%
\subsection{Some improvements using \texorpdfstring{$(SSR)$}{(SSR)}}\label{sc_improv}
%%%%%%%%%%%%%%%%%%%%%%%%%%%%%%%%%%%%%%%%%%%%%%%%%%%%%%%%%%%%%%%%%%%%%%%

The purpose of this section is to give the following slight improvement
of the bounds in Proposition \ref{pr_low}.
The proof depends on the results of Section \ref{sc_high}, so it is
important to note that Proposition \ref{pr_low} itself is enough
for the arguments of Section \ref{sc_high}.
In fact, we can even get Theorem \ref{th_main} without the results of
this section
at the modest expense of multiplying the first error term
by $\log l$.

\begin{prp}\label{pr_low2}
Assume that $\mu$ is almost non-degenerate, has finite moments of order
$\a\ge2$ and satisfies $(C)$ and $(SSR)$.
Then there are constants $C$, $c$ and a symmetric positive definite quadratic
form
$\Delta(\xi,\xi)$ on $\R^d$ invariant under the action of $K$ and
$\t_0$, such that the following holds
\be\label{eq_Kimproves}
\|S_r^l\psi_0-e^{-r^2l\Delta}\|_2<C(r^{\min\{1,\a-2\}}+|x_0|^2r^2)
\cdot(e^{-clr^2}+r^{10d})
\ee
for $r<l^{-1/3}$.
The constants $C,c$ and the form $\Delta$ depend only on $\mu$.
Moreover, if $\mu$ is symmetric or satisfies $(E)$, then we have
the better
bound:
\[
\|S_r^l\psi_0-e^{-r^2l\Delta}\|_2<C(r^{\min\{2,\a-2\}}+|x_0|^2r^2)
\cdot (e^{-clr^2}+r^{10d}).
\]
\end{prp}

The quantities $l^{-1/3}$ and $r^{10d}$ appearing in the proposition could be replaced by
other powers of $l$ and $r$.
Notice that the estimate differs only by the factor $(e^{-clr^2}+r^{10d})$
compared to Proposition \ref{pr_low}.
We indicate how to modify the argument in the previous sections to obtain
this improvement.
We only need to sharpen Lemma \ref{lm_Papprox} by the same factor.

First we note, that Proposition \ref{pr_low2} gives an improvement
only if $l>r^{-2}$ and recall that $l<r^{-3}$ is assumed in the
proposition.
Hence, we only consider the range $r^{-2}<l<r^{-3}$.

Similarly to the proof of Lemma \ref{lm_Papprox}, we write
\begin{align*}
\|S_r^l\psi_0-(PS_rP)^l\psi_0\|_2
&\le\sum_{j=0}^{l-1}
\|S_r^{l-j-1}(S_r-PS_rP)\psi_j'\|_2.
\end{align*}
To obtain the improvement of Lemma \ref{lm_Papprox},
we need to show the following improved
estimates for the terms in the above sum:
\be\label{eq_termgoal}
\|S_r^{l-j-1}(S_r-PS_rP)\psi_j'\|_2\le C(r^{\min\{4,\a\}}+|x_0|^2r^2e^{-cj})
\cdot(e^{-clr^2}+r^{20d}).
\ee

We have already seen that
\[
\|(S_r-PS_rP)\psi_j'\|_2<C(r^{\min\{\a,4\}}+|x_0|^2r^2e^{-cj})e^{-cr^2j}.
\]
We utilize Proposition \ref{pr_high} to estimate the norm of this function
when $S_r$ is applied.

If we have
\[
\|S_r^{m}(S_r-PS_rP)\psi_j'\|_2\le C(r^{\min\{4,\a\}}+|x_0|^2r^2e^{-cj})
\cdot r^{20d}
\]
for any $m<l-j-1$, then (\ref{eq_termgoal}) immediately follows.
In the opposite case, we have $\|S_r^{m}(S_r-PS_rP)\psi_j'\|_2\ge r^{4+20d}$.
We will show below that in this case we have
\be\label{eq_Lipest}
\left\|\frac{S_r^{m}(S_r-PS_rP)\psi_j'}{\|S_r^{m}(S_r-PS_rP)\psi_j'\|_2}\right\|_{\Lip(K)}
\le Cr^{-20d-4}m.
\ee
Since $r^2<\log^{-3}((r+1)Cr^{-20d-4}m+2)$, we have
\[
\|S_r^{m+1}(S_r-PS_rP)\psi_j'\|_2<e^{-cr^2}\|S_r^{m}(S_r-PS_rP)\psi_j'\|_2
\]
by Proposition \ref{pr_high}.
Repeated application of this inequality gives the claim (\ref{eq_termgoal}).

We turn to the proof of \eqref{eq_Lipest},
which finishes the proof of Proposition \ref{pr_low2}.
We introduce some notation:
Let $\f\in C(S^{d-1})$ and $\xi\in S^{d-1}$ and write
\[
\f^\xi(\t):=\f(\t\xi)\in C(K).
\]
We note that
\be\label{eq_normcompare}
\sup_{\xi\in S^{d-1}}\|\f^{\xi}\|_{\Lip}
\le\|\f\|_{\Lip(K)}
\le\|\f\|_\infty+\sup_{\xi\in S^{d-1}}\|\f^{\xi}\|_{\Lip}.
\ee
Write
\[
\PP_3^{\xi}:=\{\f^{\xi}:\f\in\PP_3\}.
\]
This is the space of polynomials of degree at most $3$
restricted to the $K$-orbit of $\xi$ pulled back to the group $K$.
Denote by $P^{\xi}$ the orthogonal projection to the space $\PP_3^{\xi}$ in $L^2(K)$.
We note that a function $\f\in C(S^{d-1})$ is in $\HH_0\oplus\ldots\oplus\HH_3$
if and only if we have
$\f^{\xi}\in\PP_3^\xi$ for all $\xi\in S^{d-1}$.
Moreover, $(P\f)^{\xi}=P^{\xi}\f^\xi$ for every $\f\in C(S^{d-1})$.

We first estimate $\|\psi_j'\|_{\Lip(K)}$.
Note that any function in $\PP_3^{\xi}$ is a polynomial of degree at most $C_d$ in the entries
of the matrices representing the elements of $K<\O(d)$.
Here $C_d$ is a number depending only on $d$.
Then the space spanned by all $P^\xi$ for $\xi\in S^{d-1}$ is finite dimensional, hence
we can write $\|\f\|_{\Lip}\le C\|\f\|_2$ for any function $\f\in \PP_3^\xi$
with a number $C$ independent of $\xi$.
This follows from the fact that any two norms on a finite dimensional space are equivalent.

We show that $\|(\psi_j')^{\xi}\|_2\le 1$ for every $\xi\in S^{d-1}$.
By the previous paragraph, this implies that $\|(\psi_j')^{\xi}\|_\Lip\le C$ and also
$\|\psi_j'\|_{\Lip(K)}\le C$ by \eqref{eq_normcompare}.
Clearly $(\psi_0')^{\xi}=\psi_0^{\xi}$ is of $L^2$-norm 1.
The claim will follow by induction, if we show that $\|(P\f)^\xi\|_2\le 1$
and $\|(S_r\f)^\xi\|_2\le 1$ for all functions $\f\in C(S^{d-1})$ that satisfy
$\|\f^\xi\|_2\le 1$ for all $\xi\in S^{d-1}$.
The first inequality holds since $P^{\xi}$ is a projection.
For the second inequality, we observe
\be\label{eq_psixi}
(S_r\f)^{\xi}(\s)=
\int e(r\langle \s\xi,v(\g)\rangle)\f^{\t_0^{-1}\xi}(\t(\g)^{-1}\s\t_0)d\mu(\g).
\ee
Recall that $\supp\t(\mu)\subset \t_0 K$, and $\t_0$ normalizes $K$ so $\t(\g)^{-1}\s\t_0\in K$
for $\g\in \supp\mu$.
The right hand side of \eqref{eq_psixi} is simply the average of functions of norm 1,
so the inequality we need holds.

The proof of \eqref{eq_Lipest} will be completed by the following lemma.
\begin{lem}\label{lm_LipK}
We have
\[
\|S_r\f\|_\infty\le\|\f\|_\infty\quad{\rm and}\quad\|S_r\f\|_{\Lip(K)}\le Cr\|\f\|_\infty+\|\f\|_{\Lip(K)}
\]
for any function $\f\in C(S^{d-1})$, where $C$ is a number depending only on $\mu$.
\end{lem}
\begin{proof}
The first claim is trivial.
For the second one, we write
\begin{align*}
|\rho_r(\g)&\f(\s\xi)-\rho_r(\g)\f(\xi)|\\
&=|e(r\langle v(\g),\s\xi\rangle)\f(\t(\g^{-1})\s\xi)
-e(r\langle v(\g),\xi\rangle)\f(\t(\g^{-1})\xi)|\\
&\le|e(r\langle v(\g),\s\xi\rangle)-e(r\langle v(\g),\xi\rangle)||\f(\t(\g^{-1})\xi)|\\
&\quad+|\f(\t(\g^{-1})\s\xi)-\f(\t(\g^{-1})\xi)|
\end{align*}
The first term is estimated by $Cr|v(\g)|\cdot|\xi-\s\xi|\cdot\|\f\|_\infty$.
For the second term we observe that $\t(\g^{-1})\s\xi=[\t(\g^{-1})\s\t(\g)]\t(\g^{-1})\xi$,
hence
\[
|\f(\t(\g^{-1})\s\xi)-\f(\t(\g^{-1})\xi)|\le\|\f\|_{\Lip(K)}\cdot\dist(1,\s).
\]

We integrate $\g$ and use the moment condition:
\[
|S_r\f(\s\xi)-S_r\f(\xi)|\le(Cr\|\f\|_\infty+\|\f\|_{\Lip(K)})\cdot\dist(1,\s).
\]
This proves the lemma.
\end{proof}

%%%%%%%%%%%%%%%%%%%%%%%%%%%%%%%%%%%%%%%%%%%%%%%%%%%%%%%%%%%%%%%%
\section{The main theorem}
\label{sc_proof}
%%%%%%%%%%%%%%%%%%%%%%%%%%%%%%%%%%%%%%%%%%%%%%%%%%%%%%%%%%%%%%%%

We turn to the proof of the main result of this paper,
Theorem \ref{th_main}.
As usual, we denote by $\mu$ the common law of $X_{i}$.
By assumption, $\mu$ has finite moments of order $\a>2$
and satisfies $(SSR)$.
Without loss of generality, we can replace $\mu$ by $\mu^{*(l_0)}$
for some fixed integer $l_0$, hence by Lemmata \ref{lm_almostnd}
and \ref{lm_normalize},
we can assume that $\mu$ is almost non-degenerate and that $K$
is normalized by $\t_0$.
Furthermore, we assume that $\mu$ also satisfies $(C)$ and prove
the estimates with $y_0=0$.
Lemma \ref{lm_O} in the previous section shows that we can
reduce the general case of the theorem to this one by changing
the coordinate system.

Denote by $\nu_l=\mu^{*(l)}.\d_{x_0}$ the distribution of the random
walk after $l$ steps starting from the point $x_0$.
To evaluate the left hand sides of the formulae in the statement
in Theorem \ref{th_main}, we use Plancherel's formula:
\[
\E[f(Y_l)]=\int f(x)d\nu_l=\int \wh f(\xi)\wh\nu_l(\xi)d\xi.
\]

We break the latter integral into two regions.
First we consider $|\xi|<l^{-1/3}$ and use Proposition \ref{pr_low2}
in this region.

Recall from Section \ref{sc_nota}, that $\Res_r: C(\R^d)\to C(S^{d-1})$
is the restriction to the sphere of radius $r$ and that
\[
\int_{|\xi|=r}|\f(\xi)|^2d\xi=r^{d-1}\Vol(S^{d-1})\|\Res_r \f\|_2^2.
\]
(The factor $\Vol(S^{d-1})$ is due to our normalization convention for the $L^2$-norm
on $S^{d-1}$.)
Recall also that
\[
\psi_0(\xi)=\psi_{0,r}(\xi)=
e(r\langle x_0,\xi\rangle)=\Res_r(\wh{\d_{x_0}})(\xi),
\]
and
\[
\Res_r\wh{\nu_l}=S_r^{l}\psi_{0,r}.
\]

For $r\le l^{-1/3}$, we write:
\begin{align}
\int_{|\xi|=r} \wh f(\xi)&\wh\nu_l(\xi)d\xi
=r^{d-1}\int_{S^{d-1}}[\Res_r\wh f](\xi)\cdot
[S_r^l \psi_{0,r}](\xi)d\xi\nonumber\\
&=\int_{|\xi|=r}\wh f(\xi)e^{-l\D(\xi,\xi)}d\xi
+r^{d-1}\int_{S^{d-1}}[\Res_r\wh f](\xi)\Psi(\xi)d\xi,\label{eq_error1}
\end{align}
where $\D(\xi,\xi)$ is the quadratic form that appears
in Proposition \ref{pr_low2} and
$\Psi=S_r^l\psi_0-e^{-r^2l\Delta}$.
By Proposition \ref{pr_low2},
\[
\|\Psi\|_2\le
C(r^{\min\{1,\a-2\}}+|x_0|^2r^2)(e^{-clr^2}+r^{10d}).
\]
Using $|\wh f(\xi)|\le\|f\|_1$ and the Cauchy-Schwartz inequality,
we can bound the second term in (\ref{eq_error1}) by
\[
Cr^{d-1}(r^{\min\{1,\a-2\}}+|x_0|^2r^{2})(e^{-clr^2}+r^{10d})\|f\|_1.
\]
Integrating for $0\le r\le l^{-1/3}$, we can write:
\begin{align}
&\int_{|\xi|\le l^{-1/3}} \wh f(\xi)\wh\nu_l(\xi)d\xi\nonumber\\
&\quad=\int_{\xi\in\R^d}\wh f(\xi)e^{-l\D(\xi,\xi)}d\xi
+O(l^{-\frac{d+\min\{1,\a-2\}}{2}}+|x_0|^2l^{-\frac{d+2}{2}})\|f\|_1.\label{eq_lowfrq}
\end{align}

It is well known that $e^{-l\D(\xi,\xi)}$ is the Fourier transform
of a Gaussian measure, i.e. there is a quadratic form $\D'$ and a
constant $C_{\D'}$ such that
\[
\int \wh f(\xi)e^{-l\D(\xi,\xi)}d\xi=
C_{\Delta'}l^{-d/2}\int f(x)e^{-\D'(x,x)/l}dx.
\]

We recognize the first term on the right
of (\ref{eq_lowfrq}) as the main term in Theorem \ref{th_main}, while
the second one is the first error term.
It is also clear that if $\mu$ is symmetric or satisfies $(E)$
and we use the improved bounds of Proposition \ref{pr_low2}, then
we get the improved error term claimed in the theorem.

It is left to show that
\be\label{eq_lastgoal}
\left|\int_{|\xi|\ge l^{-1/3}} \wh f(\xi)\wh\nu_l(\xi)d\xi\right|\le
Ce^{-cl^{1/4}}\|\f\|_W^{2,(d+1)/2},
\ee
and the proof will be finished.

To this end, we prove a lemma using Proposition \ref{pr_high}.
\begin{lem}\label{lm_high}
Let $l$ be an integer and
suppose that $e^{l^{1/4}}>r>l^{-1/3}$ and $|x_0|<e^{l^{1/4}}$.
As above, write
$\psi_{0,r}(\xi)=e(r\langle x_0,\xi\rangle)$ for $|\xi|=1$.
There is a constant $c$ depending only on $\mu$ such that
\[
\|S_r^l\psi_{0,r}\|_2\le e^{-cl^{1/4}}.
\]
\end{lem}

\begin{proof}
Choose $1>c>0$ to be sufficiently small, to be specified below.
Assume to the contrary that the statement is false for
some $r,l$ and $x_0$.
Then for each  $j\le l$, we have
\[
\|S_r^{j}\psi_{0,r}\|_2> e^{-l^{1/4}},
\]
otherwise we get a contradiction from $\|S_r\|\le1$.
(Recall that $c<1$.)

To use Proposition \ref{pr_high}, we need to estimate the Lipschitz
norm of the function $S_r^{j}\psi_{0,r}$.
Using $\|\psi_{0,r}\|_\Lip\le Cr|x_0|$ and Lemma \ref{lm_LipK} repeatedly, we get
\[
\|S_r^{j}\psi_{0,r}\|_{\Lip(K)}\le Cr( j+|x_0|)\le Ce^{2l^{1/4}}
\]
for all $j\le l$.
Now define
$\f_j={S_r^{j}\psi_{0,r}}/{\|S_r^{j}\psi_{0,r}\|_2}$.

Then we have
\[
(r+1)\|\f_j\|_\Lip<Ce^{4l^{1/4}}.
\]

We can apply Proposition \ref{pr_high} for $\f_j$, and get
\[
\|S_r\f_j\|_2\le e^{-c'l^{-3/4}},
\]
where $c'$ is a number depending only on the constant
$c$ from Proposition \ref{pr_high}.
Note that $r^2\ge l^{-2/3}\ge l^{-3/4}$.

If we multiply these inequalities together for $1\le j\le l$, we get
\[
\|S_r^l\psi_{0,r}\|_2\le e^{-c'l^{1/4}},
\]
a contradiction if we choose $c$ to be less than $c'$.
\end{proof}

Similarly as above, we use the Cauchy-Schwartz inequality:
\begin{align*}
\int_{|\xi|=r} \wh f(\xi)&\wh\nu_l(\xi)d\xi
=r^{d-1}\int_{S^{d-1}}\Res_r\wh f(\xi)\cdot S_r^l\psi_{0,r}(\xi)d\xi\\
&\le r^{(d-1)/2}\left(\int_{|\xi|=r}|\wh f(\xi)|^2d\xi\right)^{1/2}
\cdot\Vol(S^{d-1})^{1/2}\|S_r^l\psi_{0,r}\|_2.
\end{align*}
We integrate for $r>l^{-1/3}$, and then use the Cauchy-Schwartz inequality
again:
\begin{align*}
\int_{|\xi|\ge l^{-1/3}}&\wh f(\xi)\wh\nu(\xi)d\xi
\le C\int_{l^{-1/3}}^{\infty}
\left(r^{d+1}\int_{|\xi|=r}|\wh f(\xi)|^2d\xi\right)^{1/2}
\cdot\left(r^{-1}\|S_r^l\psi_{0,r}\|_2\right)dr\\
&\le C
\left(\int_{|\xi|\ge l^{-1/3}}
|\wh f(\xi)|^2|\xi|^{d+1}d\xi\right)^{1/2}\cdot
\left(\int_{l^{-1/3}}^{\infty}\|S_r^l\psi_{0,r}\|_2^2r^{-2}dr\right)^{1/2}.
\end{align*}
The first integral on the right hand side is bounded by
$\|f\|_{W^{2,(d+1)/2}}$.
By Lemma \ref{lm_high}, we have
\[
\|S_r^l\psi_{0,r}\|_2^2r^{-1/2}<e^{-cl^{1/4}}
\]
for all $r$.
Indeed, one can use the lemma for $r\le e^{-l^{1/4}}$, and
simply $\|S_r^l\psi_{0,r}\|_2\le 1$ for larger $r$.
This implies (\ref{eq_lastgoal}), and Theorem \ref{th_main}
is proved.

%%%%%%%%%%%%%%%%%%%%%%%%%%%%%%%%%%%%%%%%%%%%%%%%%%%%%%%%%%%%%%%%%
\section{The Central Limit Theorem}
\label{sc_Pcentral}
%%%%%%%%%%%%%%%%%%%%%%%%%%%%%%%%%%%%%%%%%%%%%%%%%%%%%%%%%%%%%%%%%

The purpose of this section is to prove Theorem \ref{th_central}.
Recall the notation from Sections \ref{sc_intro} and \ref{sc_nota}.
We will deduce the theorem from Proposition \ref{pr_low} very similarly
to the methods of the previous section.

Notice that the limiting distribution of $Y_l/l^{1/2}$
does not depend on the starting point $x_0$.
Indeed, let $Y_l'$ be the random walk obtained from the same $X_l$,
but from a different point $x_0'$.
Since isometries preserve distance, we have
\[
|Y_l/\sqrt{l}-Y_l'/\sqrt{l}|=|x_0-x_0'|/l^{1/2}\to 0.
\]
For the rest of this section, take $x_0=0$.

Denote by $W\subset\R^d$ the linear subspace of vectors fixed by $K$.
By Lemma \ref{lm_O}, we can choose the origin in such a way, that
\[
v_0:=\E(Y_1)=\int\g(0)d\mu(\g)\in W.
\]
Define the random isometries $X_l'$ by $X_l'(x)=X_l(x)-v_0$.
Since $v_0\in W$, we have
\be\label{eq_mean}
Y_l':=X'_l(X'_{l-1}(\ldots(0)))=Y_l-lv_0.
\ee
Denote by $\mu'$ the law of the random isometry $X_1'$.
Then $\mu'$ satisfies $(C)$.
In what follows, we assume that $\mu=\mu'$, and prove the theorem
with $v_0=0$.
In light of (\ref{eq_mean}), this implies the theorem without the
assumption $\mu=\mu'$, as well.

Let $\nu_l=\mu^{*(l)}.\d_0$ be the law of $Y_l$.
Denote by $\psi_0\in L^2(S^{d-1})$ the constant function $\psi_0(\xi)\equiv1$.
Similarly to Section \ref{sc_proof}, we can write
\[
\Res_r\wh\nu_l(\xi)=S_{r}^l\psi_0(\xi).
\]

Fix an arbitrary constant $R>0$.
Let $\Delta$ be the quadratic form from Proposition \ref{pr_low}.
Let $\l$ be the Gaussian measure with Fourier transform
\[
\wh\l(\xi)=e^{-\Delta(\xi,\xi)}.
\]
Proposition \ref{pr_low} implies
\[
\|\Res_{r/\sqrt{l}}\wh\nu_l-\Res_r\wh\l\|_2\to 0
\]
as $l\to\infty$, uniformly for $r<R$.
Let $f$ be a function, such that $\wh f(\xi)=0$ for $|\xi|>R$.
Then by Plancherel's formula and the Cauchy-Schwartz inequality:
\begin{align*}
&\left|\int f(x/\sqrt{l})d\nu_l(x)-\int f(x)d\l(x)\right|
=\left|\int\wh f(\xi)(\wh\nu_l(\xi/\sqrt{l})-\wh\l(\xi))d\xi\right|\\
&\qquad\qquad\qquad\qquad\le\|f\|_1\cdot\int_{|\xi|<R}|\wh\nu_l(\xi/\sqrt{l})-\wh\l(\xi)|d\xi\\
&\qquad\qquad\qquad\qquad\le\|f\|_1\cdot C_{R,d}\int_{0}^R
\|\Res_{r/\sqrt{l}}\wh\nu_l-\Res_r\wh\l\|_2 dr
\to 0,
\end{align*}
where $C_{R,d}$ is a constant depending on $R$ and $d$.

Since we can approximate any continuous function by those which
have compactly supported Fourier transform, the proof is complete.

%%%%%%%%%%%%%%%%%%%%%%%%%%%%%%%%%%%%%%%%%%%%%%%%%%%%%%%%%%%%%%%%%
\section{The Local Limit Theorem}
\label{sc_Plocal}
%%%%%%%%%%%%%%%%%%%%%%%%%%%%%%%%%%%%%%%%%%%%%%%%%%%%%%%%%%%%%%%%%

We finish the paper with the proof of Theorem \ref{th_local}.
The proof is again based on Plancherel's formula and estimates on
 $\wh\nu_l$, the Fourier transform of the law of the random walk.
For the frequency range $|\xi|\le l^{-1/2+\e}$, we again use Proposition \ref{pr_low}.
However,  we do not assume that $\mu$ satisfies $(SSR)$, so
we need to find a suitable replacement for Proposition \ref{pr_high}
in estimating $\wh\nu_l(\xi)$ in the frequency range $ l^{-1/2+\e}\le|\xi|\le R$,
where $R$ is an arbitrary fixed number.
Our bounds will depend on $R$ in an uncontrolled fashion,
so we will be able to conclude the 
Local Limit Theorem only on scales $O(1)$ in contrast to
Theorem \ref{th_main}.

We introduce some notation.
Recall that $G$ is the closure of the
group generated by $\supp\wt\mu*\mu$ and
$\supp\mu$ is contained in the coset $\g_0G$.
We denote by $K$ the closure of $\t(G)$ and by $K^\circ$ its connected
component.
By Lemma \ref{lm_normalize}, we can assume that $K$ is normalized
by $\t(\g_0)$.

It is easy to see, that we can decompose $\R^d$ as an orthogonal
sum of subspaces $V_{ss}\oplus V_{a}\oplus V_o$, such that the
action of $K^\circ$ is semi-simple on $V_{ss}$, Abelian on $V_a$ and
trivial on $V_o$.
Since $K^\circ$ is invariant under conjugation by elements of $K$
and $\t(\g_0)$, it follows that these subspaces
are invariant under $K$ and $\t(\g_0)$, as well.
We denote by $S^{i}$ the unit sphere in $V_i$, where $i$ is
either $ss$, $a$ or $o$.
We denote by $\pi_i$ the orthogonal projection $\R^d\to V_i$.
We write $\t_i(\g)$ for the restriction of $\t(\g)$ to $V_i$ for $\g\in G$,
and we also write $v_i(\g)=\pi_i(v(\g))$.
In addition, by abuse of notation, we write $\pi_i(\g)$ for the isometry
$x\mapsto v_i(\g)+\t_i(\g)x$ on $V_i$.
In addition, we will denote by $\pi_i(\mu)$ the probability measure
on $\Isom(V_i)$ which is the pushforward of $\mu$ under $\pi_i$.

We give an estimate on $\wh\nu_l(\xi)$ in the region
$|\pi_{ss}(\xi)|+|\pi_{a}(\xi)|\ge l^{-1/2}\log^{1/2} l$
in Section \ref{sc_largeSSA}.
The methods will be similar to Section \ref{sc_high}.
We define unitary representations $\rho_{r_{ss},r_a,r_o}$ of $G$ and
consider operators similar to $S_r$.
We show that if a function is almost fixed by such an operator, then it must be almost fixed
by $\rho_r(\g)$ for pure translations $\g$ pointing in any direction in $V_{ss}\oplus V_a$.
The results of Section \ref{sc_high} can be reused to find translations in $V_{ss}$ and
the method of Guivarc'h \cite{Gui-uniform} can be used to produce translations in $V_a$.
The essence of the latter method is taking
commutators of isometries with commuting rotation part.

We estimate $\wh\nu_l(\xi)$ in the region
$|\pi_{ss}(\xi)|+|\pi_{a}(\xi)|\le l^{-1/2}\log^{1/2} l$
in Section \ref{sc_smallSSA}.
We need to use different methods, since it may happen that all pure translations in $G$
are orthogonal to $V_o$.
For example, consider the group generated by a $1$ parameter family of screw rotations
and all translations perpendicular to their axes.
If $G$ is this group, then $\mu$ will be non-degenerate.
So instead of finding translations, we will approximate $\wh\nu_l(\xi)$
by polynomials of a suitably large but fixed degree
in the $\pi_{ss}(\xi)$ and $\pi_a(\xi)$ variables using Taylor expansion.
This allows us to work with operators on finite dimensional spaces and use
continuity arguments.

We combine the above mentioned estimates to conclude
Theorem \ref{th_local} in Section \ref{sc_pfllt}.

%%%%%%%%%%%%%%%%%%%%%%%%%%%%%%%%%%%%%%%%%%%%%%%%%%%%%%%%%%%%%%%%
\subsection{Estimates using translations}
\label{sc_largeSSA}
%%%%%%%%%%%%%%%%%%%%%%%%%%%%%%%%%%%%%%%%%%%%%%%%%%%%%%%%%%%%%%%%

The purpose of this section is to prove the following estimate on the $L^2$ average of $\wh\nu_l$
on the direct product of spheres in $V_{ss}$, $V_a$ and $V_o$.

\begin{prp}\label{pr_large}
With the assumptions of Theorem \ref{th_local} and the notations explained at the beginning
of Section \ref{sc_Plocal} the following holds.
There are constants $C,c>0$ depending only on $\mu$ and $R$ such that
\begin{align*}
r_{ss}^{1-\dim V_{ss}}r_a^{1-\dim V_a}r_o^{1-\dim V_o}
&\mathop{\int}_{|\pi_{ss}(\xi)|=r_{ss},|\pi_{a}(\xi)|=r_{a},|\pi_{o}(\xi)|=r_{o}}
|\wh\nu_l(\xi)|^2 d\xi\\
&\quad\quad\quad\quad\quad\quad\quad\quad
\le C e^{-c\min\{(r_{ss}+r_a)^2l,l^{1/4}\}}
\end{align*}
holds for all $0<r_{ss},r_a,r_o<R$.
\end{prp}

We fix three non-negative real parameters $r_{ss},r_a$ and
$r_o$.
Analogously to $\rho_r$, we introduce the unitary representation
$\rho_{r_{ss},r_a,r_o}$ of the group generated by $G$ and $\g_0$
on the space $L^2(S^{ss}\times S^a\times S^o)$ via the following
formula:
\begin{align*}
\rho_{r_{ss},r_a,r_o}(\g)\f(\xi_{ss},\xi_a,\xi_o)
=e(r_{ss}\langle\xi_{ss}&,v_{ss}(\g)\rangle
+r_a\langle\xi_a,v_a(\g)\rangle+
r_o\langle\xi_o,v_o(\g)\rangle)\\
&\cdot\f(\t_{ss}(\g)^{-1}\xi_{ss},\t_a(\g)^{-1}\xi_a,\t_o(\g)^{-1}\xi_o).
\end{align*}
Here $\f\in L^2(S^{ss}\times S^a\times S^o)$, $\xi_{ss}\in S^{ss}$,
$\xi_{a}\in S^{a}$ and $\xi_o\in S^o$.
This representation corresponds to the action of $\g$ on the Fourier
transform of a measure restricted to the product of the spheres of
radii $r_{ss},r_a,r_o$ (resp.) in $V_{ss},V_a,V_o$ (resp.).
For a measure $\eta$ on the group generated by $G$ and $\g_0$, we define the operators
\be\label{eq_rhoofeta}
\rho_{r_{ss},r_a,r_o}(\eta)=\int\rho_{r_{ss},r_a,r_o}(\g)d\eta(\g)
\ee
acting on $L^2(S^{ss}\times S^a\times S^o)$.
These are analogues of $S_r$ and we prove the following estimate for them
similar to Proposition \ref{pr_high}.

\begin{prp}
\label{pr_highssa}
Suppose that $\mu$ is almost non-degenerate and has finite moments of order
$2$.
Then for every $R>0$, there is a constant $c>0$
depending only on $\mu$ and $R$ such that the following hold.

Let $R\ge r_{ss},r_a,r_o\ge 0$.
Let $\f\in \Lip(S^{ss}\times S^{a}\times S^o)$ with $\|\f\|_2=1$.
Then
\[
\|\rho_{r_{ss},r_a,r_o}(\mu)\f\|_2
\le1-c\min\{r_{ss}^2+r_a^2,{\log^{-3}(\|\f\|_\Lip+2)}\}.
\]
\end{prp}

Proposition \ref{pr_large} can be deduced from Proposition \ref{pr_highssa}
exactly the same way as the proof of Lemma \ref{lm_high}.
The rest of the section is devoted to the proof of  Proposition \ref{pr_highssa}.
We fix $r_{ss},r_a,r_o$ and write $\rho$ instead of
$\rho_{r_{ss},r_a,r_o}$  saving a considerable amount of ink.
The hypothesis of Proposition \ref{pr_highssa} on $\mu$ is assumed throughout
the section.

Our first goal is to
replace $\mu$ with a symmetric measure $\mu_1$
such that $\supp\t(\mu_1)\subset K^\circ$.

\begin{lem}
\label{lm_connected}
We can write $(\wt\mu*\mu)^{*(L)}=p\mu_1+q\mu_2$, with $1\ge p>0$,
where $\mu_1$ and $\mu_2$
are probability measures on $\Isom(\R^d)$ and $L\ge1$ is an integer
depending on $\mu$.
Furthermore, $\mu_1$ is almost non-degenerate, symmetric, has
finite moments of order $2$, and the closure of the
group generated by $\supp\t(\mu_1)$ is $K^{\circ}$.
\end{lem}

\begin{proof}
We fix an integer $L$ and write
\[
G^\circ=\{\g\in G:\t(\g)\in K^\circ\},
\]
let $p=(\wt\mu*\mu)^{*(L)}(G^\circ)$
and let $\mu_1$ be  $1/p$ times the restriction of $(\wt\mu*\mu)^{*(L)}$
to $G^\circ$.
The only non-trivial property to check is that almost non-degeneracy
holds if $L$ is sufficiently large.
It is enough to check for an arbitrary  point $x$ the condition that the points $\g(x)$
for $\g\in\supp \mu_1 $ do not lie in a proper affine subspace,
if $L$ is sufficiently large possibly depending on $x$.
Then the claim follows from the same Noetherian property argument as in
Lemma \ref{lm_almostnd}.

Denote by $o$ the order of $K/K^\circ$.
Using the Central Limit Theorem for the measure $\wt\mu*\mu$, we can find
an integer $L_0$, and a finite set
\[
A\subset \{\g(x):\g\in\supp(\wt\mu*\mu)^{*(L_0)}\}
\]
which approximates an $(o+1)\times\cdots\times(o+1)$ grid.
The approximation can be arbitrarily good if $L_0$ is
sufficiently large.
All that we need is that
a proper affine subspace intersects $A$ in at most $|A|/(o+1)$ points.

Then by the pigeon hole principle, there is $\t_1\in K$ such that
\[
B:=\{\g(x):\g\in\supp (\wt\mu*\mu)^{*(L_0)},\t(\g)\in\t_1 K^{\circ}\}
\]
is not contained in a proper affine subspace.
Now the claim follows for $L=2L_0$:
Indeed, take any $\g_1\in\supp(\wt\mu*\mu)^{*(L_0)}$ with
$\t(\g_1)\in\t_1 K^{\circ}$ and observe that
$\g_1^{-1}(B)$ is in the
set of images of $x$ under elements of $\supp(\mu_1)$.
\end{proof}

For the rest of the proof we work with $\mu_1$ and assume that
it satisfies the properties claimed in Lemma \ref{lm_connected}.
Moreover, we assume that $\mu_1$ has property $(C)$ which is justified
by Lemma \ref{lm_O} after changing the origin.
Then we also need to multiply the function $\f$ appearing in Proposition \ref{pr_highssa}
with a character, possibly increasing its Lipschitz norm by a factor depending on $\mu$ and $R$.
(Compare with the discussion on page \pageref{pg_high} in Section \ref{sc_high}.)
We set out to prove  an inequality analogous to the one claimed in
Proposition \ref{pr_highssa} for the operator $\rho(\mu_1)$.

We fix  $\f\in C^1(S^{ss}\times S^{a}\times S^{o})$.
As in Section \ref{sc_low}, the heart of the proof is the
study of the set
\[
B(\e):=\{\g\in\Isom(\R^d):\|\rho(\g)\f-\f\|_2<\e\}.
\]
The next two Lemmata is obtained by a simple variation on the arguments
in Section \ref{sc_low}.

\begin{lem}
\label{lm_rota2}
Let $\e>0$ and
let $l_1=C_1(r_{ss}^{-2}+\log^{3}(\|\f\|_\Lip+2))$,
where $C_1$ is a suitably large
constant depending on $\mu$ and $\e$.
Suppose that $\mu_1^{*(l_1)}(B(\e))>9/10$.
Then there is a set $X\subset B(64\e)$ such that
\[
\t_{ss}(X)=\pi_{ss}(K^\circ)
\quad{\rm and}\quad\pi_a(X)=\{1\}=\pi_o(X).
\]
\end{lem}
\begin{proof}
Following the proof of Lemma \ref{lm_rota}, it is easy to find a subset
$X_0\subset B(4\e)$ such that $\pi_{ss}(\t(X_0))=\pi_{ss}(K^\circ)$.
Consider $X_1=[X_0,X_0]$ and $X=[X_1,X_1]$.
Clearly $\pi_o(X_1)=\{1\}$, and $\pi_a(X_1)$ consists
of translations.
Therefore $\pi_a(X)=\{1\}=\pi_o(X)$, and $X\subset B(64\e)$
follows from the triangle inequality.
Since every element is a commutator
in a connected semi-simple compact Lie group (see \cite{Got-commutator}),
we have $\pi_{ss}(X)=\pi_{ss}(K^\circ)$ which finishes the proof.
\end{proof}

\begin{lem}
\label{lm_transSS}
Let $\e>0$ be arbitrary and $l_1$ and $C_1$ be as in the previous lemma.
Let $l_2=C_2(r_{ss}^{-2}+\log^{3}(\|\f\|_\Lip+2))$,
where $C_2$ is a suitably large
constant depending on $\mu_1$ and $C_1$.
Suppose that $\mu_1^{*(l_i)}(B(\e))>9/10$, for $i=1,2$.
Then there is a constant $c>0$ depending on $\mu_1$ such that the
following hold.
For any unit vector $u_0\in S^{ss}$, there is an element
$\g_1'\in B(386\e)$
such that
\begin{align*}\t_{ss}(\g_1')=1,\;
|v_{ss}(\g_1')|<r_{ss}^{-1}/2&, \;\langle
v_{ss}(\g_1'),u_0\rangle > cr_{ss}^{-1},
\quad{\rm and}\\
\pi_a(\g_1')=& 1=\pi_0(\g_1').
\end{align*}
\end{lem}
\begin{proof}
Consider the projection to $V_{ss}$ and repeat the argument in Lemmata
\ref{lm_lowlength}--\ref{lm_shorten}, except that instead
of the set $X$ constructed in Lemma \ref{lm_rota}, use the one constructed
in Lemma \ref{lm_rota2}.
We need to use six elements of $X$, and since now they are in $B(64\e)$
instead of $B(4\e)$, the resulting element $\g_1'$ will be in $B(386\e)$.
Recall from the proof of Lemma \ref{lm_shorten}
that $\g_1'$ is of the form $g_1\g_1g_1^{-1}g_2\g_1^{-1}g_2^{-1}$,
where $g_1,g_2\in X$.
Since $\pi_a(g_1)=\pi_a(g_2)=1$, we have
$\pi_a(\g_1')=\pi_a(\g_1)\pi_a(\g_1^{-1})=1$, and a similar
calculation applies to the projection to $V_o$.
This finishes the proof.
\end{proof}

In the above lemma we constructed a translation in $V_{ss}$.
The next goal will be to construct a translation in $V_{a}$.
This is done in the next two lemmata
by adapting the method of Guivarc'h \cite{Gui-uniform}.
Denote by $G_1$ the closure of the group generated by $\supp(\mu_1)$.

\begin{lem}
\label{lm_comm}
$\pi_a([G_1,G_1])$ is the additive group of the vector space $V_a$.
\end{lem}

Here $[G_1,G_1]$ denotes the derived subgroup of $G_1$, not
just the set of commutators.

\begin{proof}
Clearly $H:=\pi_a([G_1,G_1])$ is a subgroup of the
additive group of $V_a$, and it is invariant under the action of
$\pi_a(K^\circ)$.
Since $\pi_a(K^\circ)$ is connected, every connected component of $H$ is
invariant under $\pi_a(K^\circ)$.
Every such component is an affine subspace of $V_a$.
The point of such an affine subspace which is closest to the origin
is a fixed point of $\pi_a(K^\circ)$.
By the definition of $V_a$, the only fixed point is the origin.
Therefore it follows that $H$ is a linear subspace of $V_a$ invariant
under the action of $\pi_a(K^\circ)$.

Assume to the contrary that $H$ is a proper subspace of $V_a$.
Let $W$ be a two dimensional subspace of $V_a$ which is invariant under
$\pi_a(K)$ and orthogonal to $H$.
By projecting the translation part to $W$, $G_1$
naturally embeds to $\Isom(W)$; denote by $G_W$ the image.
Then $G_W$ is commutative (since it has a trivial commutator)
but has a non-trivial rotation part (since $\pi_a(K^\circ)$ acts on $W$
non-trivially),
hence it consists of rotations around the same point $x\in W$.
This means that $\mu_1$ almost every image of $x$ is orthogonal to $W$,
a contradiction to almost non-degeneracy.
\end{proof}

\begin{lem}
\label{lm_transA}
Let $\e$, $l_1$ and $C_1$ be as in Lemma \ref{lm_rota2}.
Suppose that $\mu_1^{*(l_1)}(B(\e))>9/10$.
Then for every $u_0\in V_a$, there are $c>0$, $v\in V_a$ with
$|v-u_0|<|u_0|/10$ and
an integer $L$ such that the following holds.
Let $M$ be an arbitrary positive integer and
assume that $\mu_1^{*(2)}(B(\e/M))>1-c$.
Then there is $\g_1'\in B(L\e)$ such that
\[
v_a(\g_1')=Mv, \quad \t_a(\g_1')=1
\quad{\rm and}\quad\pi_{ss}(\g_1')=1=\pi_o(\g_1').
\]
The vector $v$ may depend on $\f$, but $c$ and $L$ depend only on
$\mu$ and $u_0$.
\end{lem}

This lemma allows us to find pure translations in $B(L\e)$ approximating an
arbitrary direction in $V_a$.
This (or Lemma \ref{lm_transSS}) lead to contradiction
if we set the parameters in such a manner that
$M|u_0|\approx 1/r_a$ and $L\e$ is sufficiently small, so
one of the assumptions of Lemma \ref{lm_transA} must fail.
We can derive the claim of Proposition \ref{pr_highssa} from either
 $\mu_1^{*(l_1)}(B(\e))<9/10$ or $\mu_1^{*(2)}(B(\e/M))<1-c$.
In the second case e.g., we can get $\|\rho(\mu_1)\f\|_2\le1-c\e^2/M^2$.
This will imply the claim if we set $M\approx\max\{1,r_a^{-1}\}$.

Observe that the numbers $c$ and $L$ depend on $u_0$ in an
uncontrolled way.
Hence it is important to note that we will apply the lemma with
choosing $u_0$ from a fixed finite collection depending on the parameter $R$.

\begin{proof}
There is $\g_2\in G_1$, such that
$\t_a(\g_2)$ does not have any fixed vectors in $V_a$ except for 0.
This is an open condition, so we can assume that
$\g_2\in\supp\mu_1^{*(m)}$ for some integer $m$ depending on $\mu$.
Thus $\g_2=g_1\cdots g_{m}$ for some $g_i\in\supp\mu_1$.

There is a vector $u_1\in V_a$ such that $u_0=u_1-\t_a(\g_2)u_1$.
(Since $\t_a(\g_2)$ has no fixed vectors, $1-\t_a(\g_2)$ has
trivial kernel.)
We can find a small ball $U_i$ around each $g_i$ such that
$|u_1-\t_a(g_1'\cdots g_{m}')u_1-u_0|<|u_0|/20$ for any choice of $g_i'\in U_i$.
We set the constant $c$ in the lemma so that $\mu_1^{*(2)}(U_i)>c$
for all $i$.
This allows us to find an element
\[
\g_2'=g_1'\cdots g_{m}'\in B(m\e/M)
\]
such that $|u_1-\t_a(\g_2')u_1-u_0|<|u_0|/20$.

For reasons that will be clear at the end of the proof, we now
search for an element $\g_3'\in B(L\e/M)$  such that $v_a(\g_3')$
approximates $u_1$ instead of
$u_0$ which is the objective in the lemma.
By Lemma \ref{lm_comm}, we have elements $\g_4,\g_5\in G_1$
such that $u_1=\pi_a([\g_4,\g_5])$.
Using an argument very similar to the one above, we can
approximate $\g_4$ and $\g_5$ by elements in $B(m\e/M)$
(by taking  $m$ larger and $c$ smaller perhaps).
Then we can
find a vector $v_1$ and an element $\g_3'\in B(4m\e/M)$
such that $v_a(\g_3')=v_1$, $|v_1-u_1|<|u_0|/40$ and $\t_a(\g_3')=1$.

Now we make use of the set $X$
constructed in Lemma \ref{lm_rota2} to cancel the rotation parts
of $\g_2'$ and $\g_3'$ in the $V_{ss}$ component.
Let $h_2,h_3\in X$ be such that
$\t_{ss}(h_2)=\t_{ss}(\g_2')^{-1}$ and
$\t_{ss}(h_3)=\t_{ss}(\g_3')^{-M}$.
Then $h_2\cdot\g_2'$ and $h_3\cdot(\g_3')^M$ act on $V_{ss}\oplus V_o$ by
translation,
hence
\[
\g_1':=[h_3\cdot(\g_3')^M,h_2\g_2']
\]
acts trivially on $V_{ss}$ and $V_o$.
On the other hand, an easy calculation shows that
\[
\pi_a(\g_1')=(M(v_1-\t_a(\g_2')v_1),1)
\]
and
\[
|v_1-\t_a(\g_2')v_1-u_0|\le|u_1-\t_a(\g_2')u_1-u_0|+2|u_0|/40\le|u_0|/10
\]
and $\g_1'\in B((10m+256)\e)$ which was to be proved.
\end{proof}

\begin{proof}[Proof of Proposition \ref{pr_highssa}]
Without any significant changes to the argument in the proof
of Proposition \ref{pr_high}, we can deduce
from Lemma \ref{lm_transSS} the estimate
\be\label{eq_rss1}
\|\rho(\mu_1)\f\|_2\le 1-
c'\min\{r^2_{ss},{\log^{-3}(\|\f\|_\Lip+2)}\}.
\ee
We suppress the details but
carry out a similar argument which
proves
\be
\|\rho(\mu_1)\f\|_2
\le1-c'\min\{r_a^2,{\log^{-3}(\|\f\|_\Lip+2)}\}\label{eq_ra1}.
\ee

There is a unit vector $u\in S^a$ such that
\be\label{eq_eps0}
\int_{\xi_{ss}\in S^{ss},\xi_o\in S^{o},|\xi_a-u|<1/10}|
\f(r_{ss}\xi_{ss},r_a\xi_a,r_o\xi_o)|^2d\xi_{ss}d\xi_a d\xi_o>\e_0^2
\ee
for a constant $\e_0$ which depends only on the dimension
of $V_a$.
Moreover, we can choose $u$ from a fixed finite
sufficiently dense subset of $S^{a}$.
If $r_a<1$, let $M=\lfloor 10 r_a^{-1}\rfloor$ and let $M=1$
otherwise.
If $r_a<1$, then let $u_0=u/20$, otherwise let
$u_0=5u/\lceil 10 r_a \rceil$.
Let $C_1$, $c$ and  $L$ be the constants from Lemma \ref{lm_transA} with
this choice of $u_0$.
(Note that the possible values for $u_0$ are in a finite set which depends
only on the dimension of $V_a$ and $R$.)
Set $\e=\e_0/L$.

Assume to the contrary that $\mu_1^{*(l_1)}(B(\e))>9/10$ and
$\mu_1^{*(2)}(B(\e/M))>1-c$.
Then we can apply Lemma \ref{lm_transA}.
Let $v\in V_a$ and $\g_1'$ be as in the Lemma.
Then
\begin{align*}
\e_0^2&=L^2\e^2\ge\|\rho(\g_1')\f-\f\|_2^2\\
&\ge
\mathop{\int}_{\xi_{ss}\in S^{ss},\xi_o\in S^{o},|\xi_a-u|<1/10}
|1-e(\langle Mv,r_a\xi_a\rangle)|^2
|\f(r_{ss}\xi_{ss},r_a\xi_a,r_o\xi_o)|^2d\xi_{ss}d\xi_a d\xi_o
\end{align*}
With the above definitions, $Mu_0$ and $Mv$ approximate
$r_a^{-1} u/2$, hence $e(\langle Mv,\xi_a\rangle)$ is close to $-1$
in the domain of integration.
In particular, $|1-e(\langle Mv,\xi_a\rangle)|>1$ which is in contradiction
to (\ref{eq_eps0}).
(This somewhat vague discussion can be made precise by a straightforward
calculation.)

Now there are two possibilities:
Either $\mu_1^{*(l)}(B(\e))\le9/10$, which implies (\ref{eq_ra1})
as we have seen in the proof of Proposition \ref{pr_high}.
Or else $\mu_1^{*(2)}(B(\e/M))\le 1-c$.
If $\g_1^{-1}\cdot\g_2\notin B(\e/M)$, then
\[
{\rm Re}(\langle \rho(\g_2)\f,\rho(\g_1)\f \rangle)
={\rm Re}(\langle \rho(\g_1^{-1}\cdot\g_2)\f,\f \rangle)
\le1-\e^2/2M^2.
\]
Hence (\ref{eq_ra1}) follows:
\[
\| \rho(\mu_1)\f\|_2^2
=\int{\rm Re}(\langle \rho(\g_2)\f,\rho(\g_1)\f \rangle)
d\mu_1(\g_1)d\mu_1(\g_2)
\le 1-c\e^2/2M^2.
\]
Note that $1/M\ge\min\{r_a,1\}$.

If $p$ and $q$ are as in Lemma \ref{lm_connected}
and $L'$ is the number $L$ from that lemma, then we can conclude
{}from \eqref{eq_rss1} and \eqref{eq_ra1}
\begin{align*}
\|(\rho(\mu)^*\rho(\mu))^{L'}\f\|_2&= \|p\rho(\mu_1)\f+q\rho(\mu_2)\f\|_2\\
&\le1-pc'\min\{r_{ss}^2+r_a^2,{\log^{-3}(\|\f\|_\Lip+2)}\},
\end{align*}
which in turn implies the proposition.
\end{proof}

%%%%%%%%%%%%%%%%%%%%%%%%%%%%%%%%%%%%%%%%%%%%%%%%%%%%%%%%%%%%%%%
\subsection{Estimates using continuity arguments}
\label{sc_smallSSA}
%%%%%%%%%%%%%%%%%%%%%%%%%%%%%%%%%%%%%%%%%%%%%%%%%%%%%%%%%%%%%%%

We continue to use the notation, $V_{ss},V_a,V_0,$
etc. introduced in the beginning of Section \ref{sc_Plocal}.
Our goal is to prove the following estimate that complements
the results of the previous section.

\begin{prp}
\label{pr_smallSSA}
Assume that $\mu$ is non-degenerate and has finite moments of order $\a$
for some $\a\ge2$, and let $R>0$ be a number.
Then there is a number $C$ depending on $\mu, x_0, R$ and $\a$ such that
the following holds.
Let $0\le r_{ss},r_a\le R$ be numbers, $l$ a positive integer and $0\le s,\d\le 1$
numbers such that 
$l>C\log(s^{-1})\d^{-2}$,
$s>r_{ss}+r_a$ and $C^{-1}\ge\d\ge C(r_{ss}+r_a)$.
Then
\begin{align*}
r_{ss}^{1-\dim V_{ss}}r_{a}^{1-\dim V_{a}}&
\mathop{\int}_{\pi_{ss}(\xi)=r_{ss},\pi_{a}(\xi)=r_a,\d\le|\pi_{o}(\xi)|\le R}
|\wh\nu_l(\xi)| d\xi\\
&<C\log(s^{-1})^{1/2}s\d^{\dim V_o}
+C\log(s^{-1})^{\a/2}s^{\a}\d^{-\a-2}.
\end{align*}
\end{prp}

We indicate the approximate values of the parameters that we will
set in the next section:
We take $s\approx(\log^{1/2} l)l^{-1/2}$ and
$\d\approx l^{-\b}$, where $\b$ is slightly smaller than $1/2$.

To outline the idea of the proof, we temporarily assume that
$\t_{o}(G)$ is trivial and $\g_0=1$.
(We will reduce the problem to this situation by defining a measure
$\mu_1$ similar to the one we had in the previous section.)
We can restrict the action of $G$ on Fourier space to sets of the form
$\{\xi:|\pi_{ss}(\xi)|=r_{ss},|\pi_{a}(\xi)|=r_{a},\pi_o(\xi)=\xi_o\}$.
This gives rise to a unitary representation
\[
\rho_{r_{ss},r_a,\xi_o}(\g)\f(\xi_{ss},\xi_a)
=e(\langle r_{ss}\xi_{ss}+r_a\xi_a+\xi_o,v(\g)\rangle)
\f(\t_{ss}^{-1}(\g)\xi_{ss},\t_a^{-1}(\g)\xi_a).
\]
 of $G$ for each $r_{ss},r_a\ge0$ and $\xi_o\in V_o$ acting on the
space $L^{2}(S^{ss}\times S^a)$.

We will study the operators $\rho_{r_{ss},r_a,\xi_o}(\mu)$
defined analogously to \eqref{eq_rhoofeta}.
We consider the finite dimensional subspace $\PP_{\a-1}\subset L^{2}(S^{ss}\times S^a)$,
which we define as the restriction of polynomials of degree at most $\a-1$
to $S^{ss}\times S^a$.
This space is invariant for $\rho_{0,0,\xi_o}(\mu)$, and we will show that there
are only finitely many ``bad''  points in the ball $\{\xi_o:|\xi_o|\le R\}$ such that 
$\|\rho_{0,0,\xi_o}(\mu)|_{\PP_\a}\|=1$.
We will also understand the behavior of the function $\|\rho_{0,0,\xi_o}(\mu)|_{\PP_\a}\|$
in small neighborhoods of those ``bad''  points.
We then combine this with a continuity argument (essentially using that the above norm
function is continuous and attains its extrema on compact sets) to obtain bounds
for $\|\rho_{0,0,\xi_o}(\mu)|_{\PP_\a}\|$ for $\xi_o$ not too close to the ``bad'' points.

We also show that $\rho_{r_{ss},r_a,\xi_o}(\mu)$ is a small perturbation of $\rho_{0,0,\xi_o}(\mu)$
and the norm bounds are valid for the former operator, as well.
Then we show that we can approximate $\wh\nu_l$ by polynomials of degree $\a-1$ in the $\pi_{ss}(\xi)+\pi_a(\xi)$
coordinates, and using the norm bounds of $\rho_{r_{ss},r_a,\xi_o}(\mu)$ iteratively we get the
desired bound on $\wh\nu_l$.

We need to give a separate argument in the neighborhood of ``bad'' points.
The bounds in this case will be substantially weaker.
We will show that the only ``bad'' point for $\a=0$ is the origin,
and we can do a similar argument as above.

We note that there are examples, when ``bad'' points do occur.
We recommend to the reader to analyze the instructive example mentioned earlier:
when $G$ is generated by a one parameter family of skew rotations and all translations
perpendicular to the axes.

Some of the above ideas are related to the arguments of Section \ref{sc_low}
and hence motivated by Tutubalin's paper \cite{Tut-CLT}.

We give a lemma similar to Lemma \ref{lm_connected}, which introduces the measure $\mu_1$
mentioned above.

\begin{lem}
\label{lm_connected2}
Let $\mu$ be as in Proposition \ref{pr_smallSSA}.
Then we can write $\wt\mu^{*(L)}*\mu^{*(L)}=p\mu_1+q\mu_2$, with $1\ge p>0$,
where $\mu_1$ and $\mu_2$
are probability measures on $\Isom(\R^d)$ and $L\ge1$ is an integer
depending on $\mu$, $R$ and $x_0$.
In addition, the set
\be\label{eq_end}
\{v_o(\g):\g\in\supp\mu_1, |v_o(\g)|<1/(2R)\}
\ee
is not contained in a proper affine subspace of $V_o$.
Furthermore, $\mu_1$ is  symmetric, has
finite moments of order $\a$, $1\in\supp\mu_1$ and  the closure of the
group generated by $\supp\t(\mu_1)$ is $K^{\circ}$.
\end{lem}

\begin{proof}
The proof is very similar to that of Lemma \ref{lm_connected}.
The main difference is that we use non-degeneracy instead of the
Central Limit Theorem.
We fix a sufficiently large integer $L$.
We write
\[
G^\circ=\{\g\in G:\t(\g)\in K^\circ\}.
\]
Let $p=\wt\mu^{*(L)}*\mu^{*(L)}(G^\circ)$ and let $\mu_1$ be $1/p$ times the restriction of
$\wt\mu^{*(L)}*\mu^{*(L)}$ to $G^\circ$
The only non-trivial property
to check is that \eqref{eq_end} is not contained in a proper affine
subspace if $L$ is sufficiently large.

Denote by $o$ the order of $K/K^\circ$.
Fix an arbitrary $\g_0\in\supp \mu$ and let $x_0$ be the starting point of the random walk.
We show that if $L$ is sufficiently large, then we can find a set
\[
A\subset\{\pi_o(\g(x_0)):\g\in\supp \mu^{*(L)}\},
\]
which approximates an $(o+1)\times\cdots\times(o+1)$ grid contained
in a $1/(4R)$ neighborhood of $\pi_o(\g_0^L(x_0))$.
The approximation can be arbitrarily good, and what we need below is that
no proper affine subspace contains more than $|A|/(o+1)$ points of $A$.

To this end, we consider an $(o+1)\times\cdots\times(o+1)$ grid $A'$ contained in 
the $1/(4R)$ neighborhood of $\pi_o(x_0)$, and for each point $x\in A'$ let $D_x$ be the
complement of a small open neighborhood of $x$.
By non-degeneracy, we know that if $L$ is sufficiently large, then for any $x\in A'$
we have
\[
\{\pi_o(\g(x_0)):\g\in\supp \mu^{*(L)}\}\not\subset \pi_o(\g_0^L)(D_x).
\]
This implies the claim on the existence of the set $A$.

By the pigeon hole principle, there is $\t_1\in K$ such that
\[
B:=\{\pi_o(\g(x_0)):\g\in\supp\mu^{*(L)},\t(\g)\in\t(\g_0)\t_1 K^\circ\}\cap A
\]
is not contained in a proper affine subspace.
We choose an arbitrary element $\g_1\in\supp\mu^{*(L)}$ with $\t(\g_1)\in\t(\g_0)\t_1 K^\circ$
and $\pi_o(\g_1(x_0))\in B$.
We observe that
\[
\pi_o(\g_1^{-1})(B)\subset\{\pi_o(\g(x_0)):\g\in\supp \mu_1\}.
\]
We also note that the set $\pi_o(\g_1^{-1})(B)$ contains $x_0$ by construction and its
diameter is at most $1/(2R)$.
Since $\pi_o(\g)$ is a translation for all $\g\in\supp \mu_1$, we have
$v_o(\g)=\pi_o(\g(x_0))-x_0$ for those $\g$.
Therefore, $(\pi_o(\g_1^{-1})(B)-x_0)\subset \eqref{eq_end}$,
which is not contained in a proper affine subspace.
This finishes the proof.
\end{proof}

We start the program described above by
giving a few lemmata on the properties of the operators
$\rho_{r_{ss},r_a,\xi_o}(\mu_1)$.
Throughout the section, we assume that $\mu_1$ satisfies the properties stated
in Lemma \ref{lm_connected2} and in addition that it satisfies $(C)$.
By changing the origin, $\wh\nu_l$ gets multiplied by a character and this does not change
the statement of Proposition \ref{pr_smallSSA}.
Hence the assumption $(C)$ is justified by Lemma \ref{lm_O}.
We first show that $\rho_{r_{ss},r_a,\xi_o}(\mu_1)$
is a small perturbation of 
$\rho_{0,0,\xi_o}(\mu_1)$.

\begin{lem}
\label{lm_r=0}
There is a constant $C$ depending only on $\mu_1$ such that
\[
\|\rho_{r_{ss},r_a,\xi_o}(\mu_1)-\rho_{0,0,\xi_o}(\mu_1)\|<C(r_{ss}+r_a).
\]
\end{lem}
\begin{proof}
Let $\f\in L^2(S^{ss}\times S^{a})$.
Then by Taylor's theorem
\begin{align*}
\rho_{r_{ss},r_a,\xi_o}(\mu_1)\f(\xi_{ss},\xi_a)=
\int &(1+O(r_{ss}|v_{ss}(\g)|+r_a|v_a(\g)|))e(\langle\xi_{o},v_o(\g)\rangle)\\
&\cdot  \f(\t_{ss}(\g)^{-1}\xi_{ss},\t_a(\g)^{-1}\xi_a)d\mu_1(\g).
\end{align*}
Then
\[
\|\rho_{r_{ss},r_a,\xi_o}(\mu_1)\f-\rho_{0,0,\xi_o}(\mu_1)\f\|_2
\le C \int (r_{ss}|v_{ss}(\g)|+r_a|v_a(\g)|)
\|\f\|_2 d\mu_1(\g),
\]
which proves the claim.
\end{proof}

The next lemma is about the behavior of $\rho_{r_{ss},r_a,\xi_o'}(\mu_1)$
in a neighborhood of a ``bad'' point.

\begin{lem}
\label{lm_eigenf}
There are  constants $c$ and $C$
which depend only on
$\mu_1$ and $R$ such that the following holds.
Suppose that $\f\in L^2(S^{ss}\times S^a)$ and  $|\xi_o|\le R$ are
such that $\rho_{0,0,\xi_o}(\mu_1)\f=\f$.
Then
\[
\|\rho_{r_{ss},r_a,\xi_o'}(\mu_1)\f\|_2<1-c|\xi_o-\xi_o'|^2+C(r_{ss}^2+r_a^2)
\]
for every $r_{ss},r_a\ge0$ and $\xi_o'\in V_o$ with $|\xi_o'|<R$.
\end{lem}
\begin{proof}
Since $\rho_{0,0,\xi_o}(\mu_1)$ is an average of unitary operators,
we must have $\rho_{0,0,\xi_o}(\g)\f=\f$ for all $\g\in\supp (\mu_1)$.
Then
\begin{align}
&\rho_{r_{ss},r_a,\xi_o'}(\mu_1)\f=
\int \rho_{r_{ss},r_a,\xi_o'}(\g)\rho_{0,0,\xi_o}(\g^{-1})\f
d\mu_1(\g)\nonumber\\
&\qquad=\f\cdot\int e(r_{ss}\langle\xi_{ss}, v_{ss}(\g)\rangle+
r_a\langle\xi_a, v_a(\g)\rangle+
\langle\xi_o'-\xi_o, v_o(\g)\rangle) d\mu_1(\g)\nonumber\\
&\qquad=\f\cdot[\int (1-2\pi ir_{ss}\langle\xi_{ss}, v_{ss}(\g)\rangle
-2\pi ir_a\langle\xi_a, v_a(\g)\rangle)\nonumber\\
&\qquad\qquad\qquad\qquad\quad\cdot e(\langle\xi_o'-\xi_o, v_o(\g)\rangle)
d\mu_1(\g)+O(r_a^2+r_{ss}^2)].
\label{eq_eigenf}
\end{align}

For every positive $c_0$, we can find $C'$ such that the following
estimate holds for the linear term in (\ref{eq_eigenf}):
\begin{align}
&\left|\int (2\pi ir_{ss}\langle\xi_{ss}, v_{ss}(\g)\rangle
+2\pi ir_a\langle\xi_a, v_a(\g)\rangle)
e(\langle\xi_o'-\xi_o, v_o(\g)\rangle) d\mu_1(\g)\right|\nonumber\\
&\le
\left|\int 2\pi ir_{ss}\langle\xi_{ss}, v_{ss}(\g)\rangle
+2\pi ir_a\langle\xi_a, v_a(\g)\rangle
 d\mu_1(\g)\right|
+C(r_{ss}+r_a)|\xi_o'-\xi_o|\nonumber\\
&\le C'(r_{ss}+r_a)^2+c_0|\xi_o'-\xi_o|^2\label{eq_lintermest}.
\end{align}
For the second inequality, we used $(C)$
to show that the first term vanishes, and the inequality between the
geometric and arithmetic means to estimate the second term.

Consider the function
\[
\Phi(\xi)=\int e(\langle\xi, v_o(\g)\rangle) d\mu_1(\g)
\]
on $V_o$.
Note that $\Phi$ depends only on $\mu_1$.
Combining (\ref{eq_eigenf}) and (\ref{eq_lintermest}) we get
\[
\|\rho_{r_{ss},r_a,\xi_o'}(\mu_1)\f\|_2\le
\Phi(\xi_o'-\xi_o)\cdot\f+
C'(r_{ss}^2+r_a^2)+c_0|\xi_o'-\xi_o|^2.
\]

If $\Phi(\xi)=1$ for some $\xi\neq 0$, then $\langle\xi,v_o(\g)\rangle$
is an integer for all $\g\in\supp(\mu_1)$ which contradicts to the property that
\eqref{eq_end} is not contained in a proper affine subspace,
hence impossible.
Using Taylor series expansion, $(C)$ and the moment condition,
we can show that
$\Phi(\xi)\le 1-c_1|\xi|^2$
for some $c_1>0$ and $|\xi|<R$.
These estimates prove the lemma if we set $c_0<c_1$.
\end{proof}

Let $X_\a\subset V_o$ be the set of those $\xi_o$ for which
there is $\f\in\PP_{\a-1}$ such that $\rho_{0,0,\xi_o}(\mu_1)\f=\f$.
This is the set whose elements we called ``bad'' points above.
We note that $\|\rho_{0,0,\xi_o}(\mu_1)\f\|_2=\|\f\|_2$ implies $\rho_{0,0,\xi_o}(\mu_1)\f=\f$
since $1\in\supp \mu_1$.
If $\rho_{0,0,\xi_o}(\mu_1)\f=\f$ and $\rho_{0,0,\xi_o'}(\mu_1)\f'=\f'$ for $\xi_o\neq\xi_o'$,
then $\f$ and $\f'$ are both eigenfunctions of $\rho_{0,0,\xi_o}(\mu_1)$
with different eigenvalues hence they are orthogonal.
Indeed, $\rho_{0,0,\xi_o}(\mu_1)\f'=\f'\cdot\int e(\langle\xi_o-\xi_o',v_o(\g)\rangle) d\mu_{1}(\g)$,
as the previous proof shows.
Since $\PP_{\a-1}$ is finite dimensional, $X_\a$ is finite.

We now combine Lemma \ref{lm_eigenf} with a continuity argument to obtain norm estimates
for $\rho_{r_{ss},r_a,\xi_o}$ on $\PP_{\a-1}$.

\begin{lem}\label{lm_rhomu1norm}
There are  numbers $c,C>0$ depending only on $\mu_1$, $R$ and $\a$
such that the following holds.
Let $r_{ss},$ $r_a$ be numbers and $\xi_o\in V_o$
such that, $|\xi_o|\le R$ and  $\dist(\xi_o,X_\a)>C(r_{ss}+r_a)$.
Then
\[
\|\rho_{r_{ss},r_a,\xi_o}(\mu_1)\f\|_2
\le(1-c\,\dist(\xi_o,X_\a)^2)\|\f\|_2
\]
for any $\f\in \PP_{\a-1}$.
\end{lem}

\begin{proof}
We assume that $\|\f\|_2=1$.
For each point $\xi_o'\in X_\a$ we choose a compact set $D_{\xi_o'}\subset V_o$
such that their union cover the $R$-ball and $\xi_o'$ is the only element
of $X_\a$ in $D_{\xi_o'}$.
Denote by $D_{\xi_o''}$ one of the regions that contain $\xi_o$.
Write $W$ for the 1-eigenspace  of $\rho_{0,0,\xi_o''}(\mu_1)$ in $\PP_{\a-1}$,
and write $U$ for the orthogonal complement.
Write $\pi_W$ and $\pi_U$ for the orthogonal projections respectively.
Set $a=\|\pi_W\f\|_2$ and $b=\|\pi_U\f\|_2$.

Since $W$ and $U$ are invariant under $\rho_{0,0,\xi_o}(\mu_1)$, we have
\[
\pi_U\rho_{0,0,\xi_o}(\mu_1)\pi_W\f=0=\pi_W\rho_{0,0,\xi_o}(\mu_1)\pi_U\f.
\]

The function $\|\rho_{0,0,\xi_o}|_U\|$ is continuous in $D_{\xi_o''}$.
Denote by $1-c_1$ its maximum.
Observe that $c_1>0$ and it depends only on $R$, $\a$, $\mu_1$ and $\xi_o''$,
and that there are a finite number of possibilities for  $\xi_o''$, so $c_1$
can be bounded below by a positive number depending only on  $R$, $\a$, $\mu_1$.
Then
\[
\|\rho_{0,0,\xi_o}(\mu_1)\pi_U\f\|_2<(1-c_1)b.
\]

Combining the above inequalities with Lemma \ref{lm_r=0} we get
\begin{align}
&\|\pi_U\rho_{r_{ss},r_a,\xi_o}(\mu_1)\pi_W\f\|_2\le C(r_{ss}+r_a)a\label{eq_UW1}\\
&\|\pi_W\rho_{r_{ss},r_a,\xi_o}(\mu_1)\pi_U\f\|_2\le C(r_{ss}+r_a)b\label{eq_UW2}\\
&\|\pi_U\rho_{r_{ss},r_a,\xi_o}(\mu_1)\pi_U\f\|_2<(1-c_1/2)b\label{eq_UW3}.
\end{align}
if $r_{ss}$ and $r_a$ are sufficiently small (depending on $c_1$).

We get 
\be
\|\pi_W\rho_{r_{ss},r_a,\xi_o}(\mu_1)\pi_W\f\|_2<(1-c\,\dist(\xi_o,X_\a)^2)a
\label{eq_UW4}
\ee
from Lemma \ref{lm_eigenf}.

Combining estimates (\ref{eq_UW1}--\ref{eq_UW4}) we can write
\begin{align*}
\|&\rho_{r_{ss},r_a,\xi_o}(\mu_1)\f\|_2^2
\le[(1-c\,\dist(\xi_o,X_\a)^2)a+C(r_{ss}+r_a)b]^2\\
&\qquad\qquad\qquad\qquad\qquad\qquad\qquad
+[(1-c_1/2)b+C(r_{ss}+r_a)a]^2\\
&\qquad\le(1-c\,\dist(\xi_o,X_\a)^2)a^2+(1-c_1/2)b^2\\
&\qquad\qquad\qquad\qquad\qquad\qquad\qquad
+4C(r_{ss}+r_a)ab+
C^2(r_{ss}+r_a)^2\\
&\qquad\le\left(1-\frac{c\,\dist(\xi_o,X_\a)^2}{2}\right)a^2
+\left(1-c_1/2+
\frac{C_2(r_{ss}+r_a)^2}{\dist(\xi_o,X_\a)^2}\right)b^2\\
&\qquad\qquad\qquad\qquad\qquad\qquad\qquad
+C^2(r_{ss}+r_a)^2.
\end{align*}
We used the inequality between the geometric and
the arithmetic means in the last line.
We can assume
$10C_2(r_{ss}+r_a)^2<c_1\,\dist(\xi_o,X_\a)^2$, and the lemma
follows.
\end{proof}

The following lemma allows us to approximate
$\wh\nu_l$ by polynomials in the $\xi_{ss}$ and $\xi_a$
variables using Taylor expansion.

\begin{lem}
\label{lm_Brkhldr}
Let $\mu$ be a probability measure on $\Isom(\R^d)$ with finite moments of order $\a$
and suppose that there are no points but the origin that is fixed by all $\t(\g)$ for $\g\in\supp \mu$.
Then there is a constant $C$ depending on $\a$ and $\mu$ such that
\[
\int |v(\g)|^\a d\mu^{*(l)}(\g)\le C l^{\a/2}.
\]
\end{lem}
\begin{proof}
Changing the origin changes the $\a$th order moments by an additive constant
at most, so for the purposes of this proof, we can assume that $\mu$ satisfies $(C)$
due to Lemma \ref{lm_O}.
Let $X_1,\ldots, X_l$ be independent random isometries with law $\mu$.
Consider the sequence of random vectors
\[
Y_l=v(X_1\cdots X_l)=v(X_1)+\t(X_1)v(X_2)+\ldots+\t(X_1)\cdots
\t(X_{l-1})v(X_l).
\]
By $(C)$, these form a martingale, and
its conditional
moments of order $\a$ are uniformly bounded.
Thus the lemma follows from Burkholder's inequality, see 
\cite[Theorem 3.2]{Bur-martingale}.
\end{proof}

Note that if we apply the above lemma for the measures $\pi_{ss}(\mu)$ and $\pi_a(\mu)$,
then we get
\be\label{eq_Brkhldr}
\int |v_{ss}(\g)|^\a+|v_a(\g)|^\a d\mu^{*(l)}(\g)\le C l^{\a/2}.
\ee

We denote by $Y_\a$ the largest subset of $X_\a$ invariant under
$\t_o(\supp \mu)$ and we write $Z_\a=X_\a\backslash Y_\a$.
Let $\d$ be a number, which satisfies the inequalities
$C^{-1}\ge \d\ge C(r_{ss}+r_a)$ as in Proposition \ref{pr_smallSSA},
where $C$ is a number that may depend
on $\mu,x_0,R$ and $\a$.
We write
\begin{align*}
D_1&=\{\xi: |\pi_{ss}(\xi)|=r_{ss},|\pi_{a}(\xi)|=r_{a},\dist(\pi_o(\xi),Y_\a\backslash\{0\})\le\d\},\\
D_2&=\{\xi: |\pi_{ss}(\xi)|=r_{ss},|\pi_{a}(\xi)|=r_{a},
\dist(\pi_o(\xi),Y_\a)\ge\d,|\pi_o(\xi)|\le R\},\\
D_3&=\{\xi: |\pi_{ss}(\xi)|=r_{ss},|\pi_{a}(\xi)|=r_{a},
\dist(\pi_o(\xi),X_\a)\ge\d,|\pi_o(\xi)|\le R\}.
\end{align*}
Observe that $D_1\cup D_2$ is the domain of integration in Proposition \ref{pr_smallSSA}.
We also note that $D_1$ and $D_2$ are invariant under $\t(\supp \mu)$, while
$D_3$ is invariant under $\t(\supp \mu_1)$.
These features will be important in what follows.
We denote by $\|\cdot\|_{L^2(D_i)}$ the $L^2$ norm with respect to the natural
volume measure on these manifolds normalized to have total mass 1.
That is, we have $\|1\|_{L^2(D_i)}=1$ by our convention.

We now estimate the Fourier transform of $\mu_1.\nu_l$ on $D_3$
using Lemma \ref{lm_rhomu1norm} and approximating $\nu_l$ by polynomials
in $\pi_{ss}(\xi)\oplus\pi_a(\xi)$.

\begin{lem}\label{lm_D3mu1}
There are numbers $c,C$ depending only on $\mu_1, x_0, R$  and $\a$ such that for any integer
$l$, we either have
\begin{align*}
\left\|\int e(\langle v(\g),\xi\rangle)\wh\nu_{l}(\t(\g)^{-1}\xi)d\mu_1(\g)\right\|_{L^2(D_3)}
&\le (1-c\d^2)\|\wh\nu_{l}\|_{L^2(D_3)}\qquad{\rm or}\\
\|\wh\nu_{l}\|_{L^2(D_3)}
&\le C\d^{-2}(r_{ss}+r_a)^{\a}l^{\a/2}.
\end{align*}
\end{lem}

\begin{proof}
Note that
\be\label{eq_FRW}
\wh\nu(\xi)=\int e(\langle v(\g)+\t(\g) x_0,\xi\rangle) d\mu^{*(l)}(\g),
\ee
where $x_0$ is the starting point of the random walk.
We fix a point $\xi_o\in V_o$ such that $\dist(\xi_o, X_\a)>\d$.
We take the Taylor expansion of \eqref{eq_FRW} around $\xi_o$.
Using Lemma \ref{lm_Brkhldr} and its corollary \eqref{eq_Brkhldr},
we find a polynomial $\f_{\xi_o}\in \PP_{\a-1}$
such that
\be
|\wh\nu_{l}(\xi)-\f_{\xi_o}(\pi_{ss}(\xi)/r_{ss},\pi_{a}(\xi)/r_{a})|
\le C (r_{ss}+r_a)^{\a}l^{\a/2}.\label{eq_alphapoly}
\ee
for all $\xi$ satisfying $\pi_o(\xi)=\xi_o$, $\pi_{ss}(\xi)=r_{ss}$
and $\pi_{a}(\xi)=r_{a}$.

By Lemma \ref{lm_rhomu1norm}
we have
\be\label{eq_rhomu1norm}
\|\rho_{r_{ss},r_a,\xi_o}(\mu_1)\f_{\xi_o}\|_2\le(1-c\d^2)\|\f_{\xi_o}\|_2.
\ee
We integrate \eqref{eq_alphapoly} and \eqref{eq_rhomu1norm}
for $\xi_o$:
\begin{align*}
&\left\|\int e(\langle v(\g),\xi\rangle)\wh\nu_{l}(\t(\g)^{-1}\xi)d\mu_1(\g)\right\|_{L^2(D_3)}\\
&\le \left\|\int e(\langle v(\g),\xi\rangle)
\f_{\pi_o(\xi)}(\t_{ss}(\g)^{-1}\pi_{ss}(\xi)/r_{ss},\t_{a}(\g)^{-1}\pi_{a}(\xi)/r_{a})
d\mu_1(\g)\right\|_{L^2(D_3)}\\
&\qquad\qquad+C(r_{ss}+r_a)^{\a}l^{\a/2}
\le(1-c\d^2)\|\wh\nu_{l}\|_{L^2(D_3)}
+C(r_{ss}+r_a)^{\a}l^{\a/2}.
\end{align*}
This finishes the proof.
\end{proof}

We give a similar estimate on $D_1$.
The argument is essentially the same, but for $\a=1$.

\begin{lem}\label{lm_D1mu1}
There are numbers $c,C$ depending only on $\mu_1, x_0, R$ and $\a$ such that for any integer
$l$, we either have
\begin{align*}
\left\|\int e(\langle v(\g),\xi\rangle)\wh\nu_{l}(\t(\g)^{-1}\xi)d\mu_1(\g)\right\|_{L^2(D_1)}
&\le (1-c)\|\wh\nu_{l}\|_{L^2(D_1)}\qquad{\rm or}\\
\|\wh\nu_{l}\|_{L^2(D_1)}
&\le C(r_{ss}+r_a)l^{1/2}.
\end{align*}
\end{lem}

\begin{proof}
First we prove that $X_1=\{0\}$.
Indeed, let $\xi_o\in X_1$.
Then for every $\g\in\supp \mu_1$ we have $\f=\rho_{0,0,\xi_o}(\g)\f$ for a constant function $\f$.
Hence $e(\langle v_o(\g),\xi_o\rangle)=1$ for all such $\g$.
Since the set \eqref{eq_end} is not contained in a proper affine subspace, this implies that
$\xi_o=0$ proving the claim.

We now fix a point $\xi_o\in V_o$ such that $\dist(\xi_o,X_\a)\le\d$.
Similarly to \eqref{eq_alphapoly}, we can find a constant function $\f_{\xi_o}$ such that
$|\wh\nu_l(\xi)-\f_{\xi_o}|\le C(r_{ss}+r_a)l^{1/2}$ for all $\xi$ satisfying
$\pi_o(\xi)=\xi_o$, $\pi_{ss}(\xi)=r_{ss}$
and $\pi_{a}(\xi)=r_{a}$.

Note that there is a number $c_1>0$ depending only on  $\mu_1, R$ and $\a$ such that
$\dist(\xi_o,X_1)=|\xi_o|>c_1$.
Then by Lemma \ref{lm_rhomu1norm}, we have
\[
\|\rho_{r_{ss},r_a,\xi_o}(\mu_1)\f_{\xi_o}\|_2\le(1-c)\|\f_{\xi_o}\|_2.
\]

As in the proof of Lemma \ref{lm_D3mu1}, we can deduce
\[
\left\|\int e(\langle v(\g),\xi\rangle)\wh\nu_{l}(\t(\g)^{-1}\xi)d\mu_1(\g)\right\|_{L^2(D_1)}
\le(1-c)\|\wh\nu_{l}\|_{L^2(D_1)}
+C(r_{ss}+r_a)l^{1/2},
\]
which proves the claim.
\end{proof}

We turn Lemma \ref{lm_D3mu1} into an estimate on $\wh\nu_{l+L}$, where $L$ is the number
from Lemma \ref{lm_connected2}.
We use a trick similar to \eqref{eq_symmetrize}.
Notice that the estimate is useful only if a large proportion of the $L^2$-mass of $\wh\nu_l$
is on $D_3$.

\begin{lem}\label{lm_D2muL}
There are numbers $c,C$ depending only on $\mu, x_0, R$  and $\a$ such that for any integer
$l$, we either have
\begin{align*}
\|\wh\nu_{l+L}\|_{L^2(D_2)}
&\le \|\wh\nu_{l}\|_{L^2(D_2)}-c\d^2\|\wh\nu_{l}\|_{L^2(D_3)}\qquad{\rm or}\\
\|\wh\nu_{l}\|_{L^2(D_3)}
&\le C\d^{-2}(r_{ss}+r_a)^{\a}l^{\a/2}.
\end{align*}
\end{lem}

\begin{proof}
We suppose that the second alternative of the conclusion does not hold.
Then it follows from Lemma \ref{lm_D3mu1} that:
\[
\left\|\int e(\langle v(\g),\xi\rangle)\wh\nu_{l}(\t(\g)^{-1}\xi)d\mu_1(\g)\right\|_{L^2(D_2)}
\le \|\wh\nu_{l}\|_{L^2(D_2)}-c\d^2\|\wh\nu_{l}\|_{L^2(D_3)}.
\]
We use here that both $D_2$ and $D_3$ are invariant under $\t(\supp \mu_1)$ and
their volumes are bounded by a constant multiple of each other.

Recall that $\wt\mu^{*(L)}*\mu^{*(L)}=p\mu_1+q\mu_2$, hence 
\be\label{eq_D32}
\left\|\int
e(\langle v(\g),\xi\rangle)\wh\nu_{l}(\t(\g)^{-1}\xi)d\wt\mu^{*(L)}*\mu^{*(L)}(\g)\right\|_{L^2(D_2)}
\le \|\wh\nu_{l}\|_{L^2(D_2)}-cp\d^2\|\wh\nu_{l}\|_{L^2(D_3)}.
\ee
We can write
\begin{align*}
\|\nu_{l+L}&\|^2_{L^2(D_2)}=
\fint_{D_2}\left|\int e(\langle v(\g),\xi\rangle)\wh\nu_{l}(\t(\g)^{-1}\xi)d\mu^{*(L)}(\g)\right|^2d\xi\\
&=\fint_{D_2}\iint e(\langle v(\g_1),\xi\rangle)\wh\nu_{l}(\t(\g_1)^{-1}\xi)\\
&\qquad\qquad\qquad\cdot e(-\langle v(\g_2),\xi\rangle)\overline{\wh\nu_{l}(\t(\g_2)^{-1}\xi)}
d\mu^{*(L)}(\g_2)d\mu^{*(L)}(\g_1)d\xi\\
&=\fint_{D_2}\iint e(\langle -\t(\g_2)^{-1}v(\g_2)+ \t(\g_2)^{-1}v(\g_1),\xi\rangle)\\
&\qquad\qquad\qquad\cdot \wh\nu_{l}(\t(\g_1)^{-1}\t(\g_2)\xi)\overline{\wh\nu_{l}(\xi)}
d\mu^{*(L)}(\g_2)d\mu^{*(L)}(\g_1)d\xi\\
&=\fint_{D_2}\int e(\langle v(\g),\xi\rangle)\wh\nu_{l}(\t(\g)^{-1}\xi)
d\wt\mu^{*(L)}*\mu^{*(L)}(\g)\overline{\wh\nu_{l}(\xi)}d\xi\\
&\le \|\wh\nu_l\|_{L^2(D_2)}\cdot
\left\|\int e(\langle v(\g),\xi\rangle)\wh\nu_{l}(\t(\g)^{-1}\xi)
d\wt\mu^{*(L)}*\mu^{*(L)}(\g)\right\|_{L^2(D_2)}\\
&\le\|\wh\nu_l\|_{L^2(D_2)}(\|\wh\nu_{l}\|_{L^2(D_2)}-cp\d^2\|\wh\nu_{l}\|_{L^2(D_3)})\\
&\le\left(\|\wh\nu_{l}\|_{L^2(D_2)}-\frac{cp\d^2}{2}\|\wh\nu_{l}\|_{L^2(D_3)}\right)^2,
\end{align*}
which proves the lemma.
We used the symbol $\fint$ to denote integration with respect to the normalized volume measure
of total mass 1.
\end{proof}

The same proof using Lemma \ref{lm_D1mu1} instead of Lemma \ref{lm_D3mu1}
gives the following:

\begin{lem}\label{lm_D1}
There are numbers $c,C$ depending only on $\mu, x_0, R$  and $\a$ such that for any integer
$l$, we either have
\begin{align*}
\|\wh\nu_{l+L}\|_{L^2(D_1)}
&\le(1-c) \|\wh\nu_{l}\|_{L^2(D_1)}\qquad{\rm or}\\
\|\wh\nu_{l}\|_{L^2(D_1)}
&\le C(r_{ss}+r_a)l^{1/2}.
\end{align*}
\end{lem}

Now we use rotations $\t_o(\g)$ for $\g\in\supp \mu$ to move the $L^2$ mass
away from $Z_\a$.
This allows us to prove that there is a number $0\le k\le |Z_\a|$ such that the
$L^2$-mass of $\wh\nu_{l+k}$ on $D_2$ is not concentrated near the points in $Z_\a$,
and we can upgrade Lemma \ref{lm_D2muL} into:

\begin{lem}\label{lm_D2}
There are numbers $c,C$ depending only on $\mu, x_0, R$  and $\a$ such that for any integer
$l$, we either have
\begin{align*}
\|\wh\nu_{l+|Z_\a|+L}\|_{L^2(D_2)}
&\le (1-c\d^2)\|\wh\nu_{l}\|_{L^2(D_2)}\qquad{\rm or}\\
\|\wh\nu_{l}\|_{L^2(D_2)}
&\le C\d^{-2} (r_{ss}+r_a)^{\a}l^{\a/2}.
\end{align*}
\end{lem}
\begin{proof}
It is clear that $\|\wh\nu_{l}\|_{L^2(D_2)}$ decreases as $l$ grows,
so it will be sufficient to prove the inequality for $\wh\nu_{l+k+L }$
for some $0\le k\le |Z_\a|$.

We show that there is some $c>0$  depending only on
$R$, $\mu$ and $\a$  such that there is $0\le k\le |Z_\a|$ such that
\be\label{eq_moveaway}
\|\wh\nu_{l+k }\|_{L^2(D_3)}\ge c \|\wh\nu_{l}\|_{L^2(D_2)}
\quad{\rm or}\quad\|\wh\nu_{l+k}\|_{L^2(D_2)}\le(1-c)\|\wh\nu_{l}\|_{L^2(D_2)}.
\ee
This combined with Lemma \ref{lm_D2muL}  finishes the proof.

Suppose that the first inequality in \eqref{eq_moveaway} does not
hold for $k=0$ say with $c=1/2$.
Then there is $\xi_o\in Z_\a$ such that
\[
\|1_{\{\xi:|\pi_o(\xi)-\xi_o|\le\d\}}\wh\nu_{l}\|^2_{L^2(D_2)}
\ge  \frac{\|\wh\nu_{l}\|_{L^2(D_2)}^2}{2|Z_\a|}.
\]
Here $1_{\{\xi:|\pi_o(\xi)-\xi_o|\le\d\}}$ denotes the indicator function of the set
$\{\xi:|\pi_o(\xi)-\xi_o|\le\d\}$.
It is clear that there is some $c_0>0$ depending only on $\mu,R$ and $\a$
such that there is $k\le|Z_\a|$
and $\g_0\in \supp(\mu^{*(k)})$ with
$\dist(\t_o(\g_0)^{-1}\xi_o,X_\a)>c_0$.
We can assume that $\d<c_0/10$.
Hence there is a neighborhood of $\g_0$ in $\Isom (\R^d)$ that we denote by $U$
such that $\dist(\t_o(\g)^{-1}\xi_o,X_\a)>2\d$ for each $\g\in U$.
Thus for each $\g\in U$, we have
\[
\|1_{\{\xi:\dist(\pi_o(\xi),X_\a)\le\d\}}(\xi)\wh\nu_{l}(\t(\g)^{-1}\xi)\|^2_{L^2(D_2)}
\le (1-1/({2|Z_\a|}))\|\wh\nu_{l}\|_{L^2(D_2)}^2.
\]

Recall that
\[
\wh\nu_{l+k}(\xi)=\int e(\langle \xi, v(\g)) \wh\nu_l(\t(\g)^{-1}\xi) d\mu^{*(k)}(\g).
\]
Since $\mu^{*(k)}(U)\ge c_1>0$ for some number $c_1$ depending only on $\mu$, $R$ and $\a$,
we have
\[
\|1_{\{\xi:\dist(\pi_o(\xi),X_\a)\le\d\}}\wh\nu_{l+k}\|^2_{L^2(D_2)}
\le  (1-c_1/({2|Z_\a|}))\|\wh\nu_{l}\|_{L^2(D_2)}^2.
\]
This implies \eqref{eq_moveaway} if the number $c$ there is sufficiently small.
\end{proof}

\begin{proof}[Proof of Proposition \ref{pr_smallSSA}]
We apply Lemma \ref{lm_D1} iteratively for $l=0,L,$ $2L,\ldots,\lceil A\log s^{-1}\rceil L$.
If the first alternative of the lemma holds always, then we have
$\|\wh\nu_{\lceil A\log s^{-1}\rceil L}\|_{L^2(D_1)}\le s^{cA}$.
In the opposite case we get
\begin{align*}
\|\wh\nu_{\lceil A\log s^{-1}\rceil L}&\|_{L^1(D_1)}
\le\|\wh\nu_{\lceil A\log s^{-1}\rceil L}\|_{L^2(D_1)}\\
&\le C(r_{ss}+r_a)(\lceil A\log s^{-1}\rceil L)^{1/2}
\le CA^{1/2} 
\log^{1/2}( s^{-1})s.
\end{align*}
If we choose $A$ sufficiently large depending on $R,\mu$ and $\a$, the last expression
will be larger than $s^{A/c}$.

We also apply Lemma \ref{lm_D2} iteratively for $l=0,L+|Z_\a|,2(L+|Z_\a|),
\ldots,$ $\lceil A\d^{-2}\log s^{-1}\rceil (L+|Z_\a|)$.
If the first alternative of the lemma holds always, then we have
$\|\wh\nu_{\lceil A\d^{-2}\log s^{-1}\rceil (L+|Z_\a|)}\|_{L^2(D_2)}\le s^{cA}$.
In the opposite case we get
\begin{align*}
\|\wh\nu_{\lceil A\d^{-2}\log s^{-1}\rceil (L+|Z_\a|)}&\|_{L^1(D_2)}
\le\|\wh\nu_{\lceil A\d^{-2}\log s^{-1}\rceil (L+|Z_\a|)}\|_{L^2(D_2)}\\
&\le C \d^{-2}
(r_{ss}+r_a)^{\a}(\lceil A\d^{-2}\log s^{-1}\rceil (L+|Z_\a|))^{\a/2}\\
&\le CA^{\a/2}\d^{-\a-2}
\log^{\a/2}( s^{-1})s^{\a}.
\end{align*}
If we choose $A$ sufficiently large depending on $R,\mu$ and $\a$, the last expression
will be larger than $s^{A/c}$.
Summing the above two  estimates and taking into account
that $\Vol(D_1)\le Cr_{ss}^{\dim V_{ss}-1}r_{a}^{\dim V_{a}-1}\d^{\dim V_o}$ and
$\Vol(D_2)\le Cr_{ss}^{\dim V_{ss}-1}r_{a}^{\dim V_{a}-1}$
we get the claim.
\end{proof}

%%%%%%%%%%%%%%%%%%%%%%%%%%%%%%%%%%%%%%%%%%%%%%%%%%%%%%%%%%%%%%%%%%%%%%%%
\subsection{Proof of the Local Limit Theorem}
\label{sc_pfllt}
%%%%%%%%%%%%%%%%%%%%%%%%%%%%%%%%%%%%%%%%%%%%%%%%%%%%%%%%%%%%%%%%%%%%%%%%

Recall from the statement of the theorem that
$X_1,X_2\ldots$ are independent
identically distributed random isometries.
By the assumptions of the Theorem, the common law  of $X_i$
is non-degenerate and has finite moments of order $\a>d^2+3d$.

By Lemma \ref{lm_O}, we can choose the origin
in such a way that $v:=\E[X_1(x_0)-x_0]$ is fixed by $K$.
Now let $\g_v\in\Isom(\R^d)$ be translation by $-v$
and consider the random isometries $X_i\cdot\g_v$
and denote by $\mu$ their common law.
Then $\mu$ also satisfies $(C)$ besides non-degeneracy and the
above moment condition,
and clearly it is enough to prove the theorem for these modified
random isometries.

We can approximate any compactly supported continuous function
in $L^\infty$ norm by functions which have smooth (say $C^\infty$)
and compactly supported Fourier transform.
Therefore we consider an arbitrary function $f$ such that
$\wh f$ is smooth and compactly supported, and prove the conclusion
of Theorem \ref{th_local} for it.
Then this will prove the theorem by approximation.
Let $R>0$ be a number such that the support of $\wh f$ is contained in
the ball of radius $R$ around the origin.

We again write $\nu_l=\mu^{*(l)}.\d_{x_0}$ and use Plancherel's formula
\[
\int f(x) d\nu_l(x)=\int \wh f(\xi)\wh\nu_l(\xi)d\xi.
\]
Let $\Delta$ be the quadratic form from Proposition \ref{pr_low}.
It is easily seen that
\[
\lim_{l\to\infty}l^{d/2}\int \wh f(\xi) e^{-l\Delta(\xi,\xi)}d\xi
=c\wh f(0),
\]
where $c$ is a constant depending on $\Delta$.
Since $\wh f(0)=\int f(x)dx$, it is enough to show that
\[
\lim_{l\to\infty}l^{d/2}\int \wh f(\xi) (\wh\nu(\xi)
-e^{-l\Delta(\xi,\xi)})d\xi
=0.
\]

The rest of the proof is devoted to estimating the above integral.
We break it up into several regions.
Let $\d=l^{-\b}$ with $\b>d/(2d+2)$ and also
\[
\b(\a+2)-\frac{\a}{2}<-\frac{d}{2},
\]
which is possible since $\a>d^2+3d$.
(This will also be the $\d$ that we set in Proposition \ref{pr_smallSSA}.)
The first region is defined as $\Omega_1:=\{\xi:|\xi|\le \d\}$.
Proposition \ref{pr_low} implies that
\[
r^{-d+1}\int_{|\xi|=r}|\wh\nu_l(\xi)-e^{-l\Delta(\xi,\xi)}|^2d\xi\le C r^2.
\]
By the Cauchy-Schwartz inequality, we have
\[
r^{-d+1}\int_{|\xi|=r}|\wh\nu_l(\xi)-e^{-l\Delta(\xi,\xi)}|d\xi\le C r.
\]
After integrating for $0\le r\le \d=l^{-\b}$
and using $|\wh f(\xi)|\le\|f\|_1$, we get
\[
\left|\int_{\Omega_1} \wh f(\xi) (\wh\nu(\xi)-e^{-l\Delta(\xi,\xi)})d\xi\right|
\le C\|f\|_1l^{-\b(d+1)}.
\]
Since $\b>d/(2d+2)$, the right side is $o(l^{-d/2})$.

Recall the notation form the beginning of
Section \ref{sc_Plocal}, where we decomposed
$\R^d$ as an
orthogonal sum $V_{ss}\oplus V_a\oplus V_o$.
To simplify the notation, we write $\xi=(\xi_{ss},\xi_a,\xi_o)$, where $\xi_i$ is
the component of $\xi$ in the corresponding subspace $V_i$.

The second region we consider is
\[
\Omega_2:=\{\xi=(\xi_{ss},\xi_a,\xi_o):|\xi_{ss}|+|\xi_a|
>C_0l^{-1/2}\log^{1/2} l,|\xi|<R\},
\]
where $C_0$ is a suitable constant depending on $\mu$.
(This region has an overlap with the first one.)
We integrate the bound in Proposition
\ref{pr_large} for $C_0l^{-1/2}\log^{1/2} l<r_{ss}+r_a<R$ and $0\le r_o\le R$
and obtain
\[
\int_{\Omega_2}|\wh\nu_l(\xi)|^2d\xi
\le C \Vol(\Omega_2) e^{-c(C_0l^{-1/2}\log^{1/2} l)^2l}\le C\Vol(\Omega_2) l^{-(d+1)}
\]
 if we take $C_0$ sufficiently
large.
The number $d+1$ in the exponent is arbitrary.
Using the Cauchy-Schwartz inequality as above and $|\wh f(\xi)|\le\|f\|_1$,
we get
\[
\left|\int_{\Omega_2}
\wh f(\xi) (\wh\nu(\xi)-e^{-l\Delta(\xi,\xi)})d\xi\right|
\le C\Vol(\Omega_2)\|f\|_1 l^{-(d+1)/2}.
\]
Note that $e^{-l\Delta(\xi,\xi)}$ is negligible in the region
of integration.
The right side is again $o(l^{-d/2})$.

The third region we consider is given by the inequalities
\[
\Omega_3=\{\xi:
|\xi_{ss}|+|\xi_a|<C_0l^{-1/2}\log^{1/2} l,\;\d<|\xi_o|<R\}.
\]
Note that $C_0l^{-1/2}\log^{1/2}l$ is much smaller than $\d$, so if
$\xi\notin\Omega_1$, that is $|\xi|>2\d$, then $\xi\in\Omega_2\cup\Omega_3$.
We use Proposition \ref{pr_smallSSA} with $s=C_0l^{-1/2}\log^{1/2 }l$ and
integrate the bound for $0<r_{ss}+r_a<s$ and $\d\le r_o\le R$
and get
\be\label{eq_O3est}
\int_{\Omega_3} |\wh\nu_l(\xi)|d\xi
\le C s^{\dim V_{ss}+\dim V_a}
(\log^{1/2}(s^{-1})s\d^{\dim V_o}+\log^{\a/2}(s^{-1})s^\a\d^{-\a-2}).
\ee

For the first term on the right, we write
\begin{align*}
C s^{\dim V_{ss}+\dim V_a}&\log^{1/2}(s^{-1})s\d^{\dim V_o}\\
&=C l^{d/2} l^{(1/2-\b)\dim V_o}l^{-1/2}\log^{1+(\dim V_{ss}+\dim V_a)} l
\end{align*}
Since $\b>d/(2d+2)$, we have $\dim V_o(1/2-\b)<1/2$ so the right hand side is
$o(l^{-d/2})$.
For the second term in \eqref{eq_O3est}, we write
\begin{align*}
 C s^{\dim V_{ss}+\dim V_a}&\log^{\a/2}(s^{-1})s^\a\d^{-\a-2}\\
&\le C l^{-(\dim V_{ss}+\dim V_a)/2}l^{\b(\a+2)-\a/2}\log^{\a+(\dim V_{ss}+\dim V_a)/2}l.
\end{align*}
Since $\b(\a+2)-\a/2<-{d}/{2}$, the right hand side is again $o(l^{-d/2})$.

Using $|\wh f(\xi)|\le\|f\|_1$ again, we get
\[
\left|\int_{\Omega_3}
\wh f(\xi) (\wh\nu(\xi)-e^{-l\Delta(\xi,\xi)})d\xi\right|
\le o(l^{-d/2})\|f\|_1.
\]
Combining the estimates for the three regions above, we get the theorem.

\bibliography{varju}
\bibliographystyle{abbrv}

\end{document}